\documentclass[a4paper,leqno,10pt]{amsart}

\usepackage{epsf,graphicx,epsfig}
\usepackage{amscd}
\usepackage{amsmath,latexsym,amssymb,amsthm}
\usepackage[nospace,noadjust]{cite}
\usepackage{textcomp}
\usepackage{setspace,cite}
\usepackage{lscape,fancyhdr,fancybox}
\usepackage{stmaryrd}
\usepackage[all,cmtip]{xy}
\usepackage{tikz}
\usepackage{cancel}
\usepackage{mathtools}
\usetikzlibrary{shapes,arrows,decorations.markings}
\usepackage{graphicx}

\usepackage{color}
\usepackage{url}
\usepackage{enumerate}
\usepackage[mathscr]{euscript}

\newcommand*{\Scale}[2][4]{\scalebox{#1}{$#2$}}%

\newcommand{\verteq}{\rotatebox{90}{$\,=$}}
\newcommand{\equalto}[2]{\underset{\scriptstyle\overset{\mkern4mu\verteq}{#2}}{#1}}

  \theoremstyle{plain}

\swapnumbers
    \newtheorem{thm}{Theorem}[section]
    \newtheorem{prop}[thm]{Proposition}
     
   \newtheorem{lemma}[thm]{Lemma}

    \newtheorem{subsec}[thm]{}
\theoremstyle{definition}
    \newtheorem{defn}[thm]{Definition}
        \newtheorem{remark}[thm]{Remark}
    \newtheorem{exam}[thm]{Example}

    \newtheorem{notation}[thm]{Notation}

\theoremstyle{remark}

\title{}
\author{}
\date{}
\usepackage{amssymb}

\usepackage{hyperref}
\hypersetup{
	colorlinks,
	citecolor=blue,
	filecolor=black,
	linkcolor=blue,
	urlcolor=black
}

\begin{document}

\title{Controlling structures, deformations and homotopy theory for averaging algebras}

\author{Apurba Das}
\address{Department of Mathematics,
Indian Institute of Technology, Kharagpur 721302, West Bengal, India.}
\email{apurbadas348@gmail.com, apurbadas348@maths.iitkgp.ac.in}

\begin{abstract}
An averaging operator on an associative algebra $A$ is an algebraic abstraction of the time average operator on the space of real-valued functions defined in time-space. In this paper, we consider relative averaging operators on a bimodule $M$ over an associative algebra $A$. A relative averaging operator induces a diassociative algebra structure on the space $M$. The full data consisting of an associative algebra, a bimodule and a relative averaging operator is called a relative averaging algebra. We define bimodules over a relative averaging algebra that fits with the representations of diassociative algebras. We construct a graded Lie algebra and a $L_\infty$-algebra that are respectively controlling algebraic structures for a given relative averaging operator and relative averaging algebra. We also define cohomologies of relative averaging operators and relative averaging algebras and find a long exact sequence connecting various cohomology groups. As applications, we study deformations and abelian extensions of relative averaging algebras. Finally, we define homotopy relative averaging algebras and show that they induce homotopy diassociative algebras. 
\end{abstract}

\maketitle



\medskip





\thispagestyle{empty}

\tableofcontents


\medskip

\section{Introduction}\label{sec-intro}
The notion of averaging operator was first implicitly studied by O. Reynolds in 1895 \cite{rey} in the turbulence theory of fluid dynamics. In the mathematical study of turbulence theory, such an operator appears as the time average operator of real-valued functions defined in time-space
\begin{align*}
f(x,t) \mapsto \overline{f}(x,t)  = \lim_{T \to \infty} \frac{1}{2T} \int_{-T}^T f(x, t + \tau) d \tau.
\end{align*}
The explicit description of an averaging operator was first defined by Kamp\'{e} de F\'{e}riet \cite{kampe}. Let $A$ be an associative algebra. A linear map $P: A \rightarrow A$ is said to be an averaging operator on $A$ if 
\begin{align}\label{paper-first}
P(a) \cdot P(b) = P (P(a) \cdot b) = P (a \cdot P(b)), \text{ for } a, b \in A.
\end{align}
A pair $(A, P)$ consisting of an associative algebra $A$ and an averaging operator $P : A \rightarrow A$ is called an averaging algebra.
In the last century, most studies on averaging operators had been done for various areas of functional analysis and applied mathematics. For the convenience of the reader, here we mention a few. In \cite{birkhoff} G. Birkhoff showed that a positive bounded projection operator on the Banach algebra $C(X)$ of real-valued continuous functions on a compact Hausdorff space $X$, is an idempotent averaging operator. Later, J. L. Kelly \cite{kelly} characterizes idempotent averaging operators on the algebra $C_\infty(X)$ of real-valued continuous functions on a locally compact Hausdorff space $X$ that vanish at the infinity. In \cite{gam-mill,mill}, J. L. B. Gamlen and J. B. Miller discussed spectrum and resolvent sets of averaging operators on Banach algebras. Besides all these, averaging operators are also studied in connection with probability theory. In \cite{moy} S.-T. C. Moy finds relations between averaging operators, conditional expectations and general integration theory.

\medskip

In the algebraic study of averaging operators, B. Brainerd \cite{brain} considered the conditions for which an averaging operator can be realized as a generalization of the integral operator on the ring of real-valued measurable functions. In 2000, W. Cao \cite{cao} studied averaging operators from algebraic and combinatorial points of view. In particular, he studied free commutative averaging algebras and described the induced Lie and Leibniz algebras. During the same period, J.-L. Loday \cite{loday-di} introduced a notion of diassociative algebra to study the universal enveloping algebra of a Leibniz algebra. A diassociative algebra is a vector space equipped with two associative multiplications satisfying three other associative-like compatibilities. M. Aguiar \cite{aguiar-pre} showed that an averaging operator on an associative algebra induces a diassociative algebra structure. The general algebraic study of averaging operators on any binary operad and their relations with bisuccessors, duplicators and Rota-Baxter operators on operad was systematically developed in \cite{pei-rep,pei-rep2,bai-imrn}. More recently, J. Pei and L. Guo \cite{pei-guo} constructed free associative averaging algebras using a class of bracketed words, called averaging words, and discovered their relations with Schr\"{o}der numbers. Averaging operators also appeared in the context of Lie algebras. They are often called embedding tensors and find important connections with Leibniz algebras, tensor hierarchies and higher gauge theories \cite{bon,lavau,strobl}. In \cite{kotov} Kotov and Strobl construct a $L_\infty$-algebra from an embedding tensor that explains the tensor hierarchy of the bosonic sector of gauged supergravity theories. In \cite{sheng-embedding} Y. Sheng, R. Tang and C. Zhu studied the cohomology and deformations of embedding tensors by considering the controlling algebras. See also \cite{das-rota,sheng-guo} for cohomological study of Rota-Baxter operators.

\medskip

In this paper, we consider the notion of a relative averaging operator as a generalization of an averaging operator. Let $A$ be an associative algebra and $M$ be an $A$-bimodule. A linear map $P : M \rightarrow A$ is called a relative averaging operator (on $M$ over the algebra $A$) if it satisfies the identity (\ref{avg-iden}), which generalizes (\ref{paper-first}). A triple $(A, M, P)$ consisting of an associative algebra $A$, an $A$-bimodule $M$ and a relative averaging operator $P$ is called a relative averaging algebra. For our convenience, we denote a relative averaging algebra $(A,M, P)$ by the notation $M \xrightarrow{P} A$. We give some characterizations of a relative averaging operator. We also construct the free relative averaging algebra over any $2$-term chain complex $V \xrightarrow{f} W$. Next, we show that a relative averaging algebra naturally induces a diassociative algebra structure. Conversely, any diassociative algebra is induced from a suitable relative averaging algebra constructed from the given diassociative algebra. We also define a notion of bimodule over a relative averaging algebra and construct the corresponding semidirect product. We show that a bimodule over a relative averaging algebra gives rise to some representations of the induced diassociative algebra.

\medskip

Then we first focus on the cohomology of relative averaging operators. Given an associative algebra $A$ and an $A$-bimodule $M$, we construct a graded Lie algebra that characterizes relative averaging operators as its Maurer-Cartan elements. This graded Lie algebra is obtained by applying the derived bracket construction to the graded Lie algebra constructed by Majumdar and Mukherjee \cite{maj-muk}. Using this characterization, we can define the cohomology of a relative averaging operator. We further show that this cohomology can be seen as the cohomology of the induced diassociative algebra with coefficients in a suitable representation. We also remark that this cohomology can be used to study formal deformations of the relative averaging operator by keeping the underlying algebra and the bimodule intact. Next, we focus on the cohomology of relative averaging algebras. To do this, we first construct a $L_\infty$-algebra that characterizes relative averaging algebras as its Maurer-Cartan elements. This helps us to define the cohomology of a relative averaging algebra (with coefficients in the {adjoint bimodule}). Subsequently, we consider the cohomology of a relative averaging algebra with coefficients in an arbitrary bimodule.


\medskip

In the next, we give two applications of our cohomology theory of relative averaging algebras. At first, we consider deformations of a relative averaging algebra $M \xrightarrow{P} A$, where we simultaneously deform the associative multiplication on $A$, the $A$-bimodule structure on $M$ and the relative averaging operator $P$. In particular, we consider formal and infinitesimal deformations of a relative averaging algebra. Our main result in deformation theory classifies the equivalence classes of infinitesimal deformations of a relative averaging algebra $M \xrightarrow{P} A$ by the second cohomology group ${H^2_\mathrm{rAvg}( M \xrightarrow{P} A)}$. Another application of the cohomology theory is to classify abelian extensions. More precisely, we consider abelian extensions of a relative averaging algebra $M \xrightarrow{P} A$ by a bimodule and show that isomorphism classes of such abelian extensions are classified by the second cohomology group of the relative averaging algebra $M \xrightarrow{P} A$ with coefficients in the bimodule.

\medskip

In the final part of the paper, we define homotopy relative averaging operators. Like relative averaging operators are defined on bimodules over an associative algebra, a homotopy relative averaging operator is defined on a representation space over an $A_\infty$-algebra. Given an $A_\infty$-algebra and a representation of it, we construct a suitable $L_\infty$-algebra $(\mathfrak{a} , \{ l_k \}_{k=1}^\infty)$. This $L_\infty$-algebra is a generalization of the graded Lie algebra that characterizes relative averaging operators as Maurer-Cartan elements. Motivated by this, we define a homotopy relative averaging operator as a Maurer-Cartan element of the $L_\infty$-algebra $(\mathfrak{a} , \{ l_k \}_{k=1}^\infty)$. A triple consisting of an $A_\infty$-algebra, a representation and a homotopy relative averaging operator is called a homotopy relative averaging algebra. We show that a homotopy relative averaging algebra induces a $Diass_\infty$-algebra (strongly homotopy diassociative algebra) structure. This generalizes our previous result that a relative averaging algebra induces a diassociative algebra. Finally, we show that a $Diass_\infty$-algebra gives rise to a homotopy relative averaging algebra.

\medskip

\noindent {\bf Note.} It is important to mention that Wang and Zhou \cite{wang-av} recently considered the cohomology and deformation theory of an averaging algebra. In their approach, they only considered the fact that an averaging algebra induces two new associative algebra structures. However, they have not used the fact that the induced two associative structures form a diassociative algebra. In the present paper, we first point out the intimate relationships between averaging algebras and diassociative algebras (see Proposition \ref{l-adj-fun}). In our knowledge, diassociative algebra is the key object for the study of (relative) averaging algebras. In Proposition \ref{n-mult}, we show that the cohomology of a relative averaging operator can be seen as the cohomology of the induced diassociative algebra with coefficients in a suitable representation. In Theorem \ref{n-func}, we find a cohomological relation between relative averaging operators and induced diassociative algebras. Since Wang and Zhou didn't consider the full diassociative algebra (as the induced structure), the above important results no longer exist in their approach. In our paper, we showed many other results which indicate that diassociative algebras are required to study relative averaging operators (cf. Proposition \ref{grph}, Theorem \ref{mc-thm-opp}, Theorem \ref{ll-iinnff}, Theorem \ref{mc-ravg-thm}). Thus, we believe that our approach and the constructions (including all the graded Lie algebras and $L_\infty$-algebras) are appropriate to deal with relative averaging algebras.

\medskip

\noindent {\bf Organization of our paper.} In Section \ref{sec-2}, we recall some basic preliminaries on the cohomology of diassociative algebras. In Section \ref{sec-3}, we consider relative averaging operators and relative averaging algebras which are the main objects of our study. We also define and study bimodules over a relative averaging algebra. The Maurer-Cartan characterizations and cohomology of relative averaging operators and relative averaging algebras are respectively considered in Sections \ref{sec-4} and \ref{sec-5}. Applications of cohomology include deformations and abelian extensions of relative averaging algebras which are respectively studied in Sections \ref{sec-6} and \ref{sec-7}. Finally, in Section \ref{sec-8}, we define homotopy relative averaging algebras and find relations with $Diass_\infty$-algebras.

\medskip

All vector spaces, linear maps and tensor products are over a field ${\bf k}$ of characteristic $0$.

\medskip






\medskip

\section{Background on diassociative algebras}\label{sec-2}
In this section, we recall some necessary background on diassociative algebras. In particular, we describe the cohomology theory of diassociative algebras. Our main references are \cite{loday-di,frab,maj-muk}.

 The notion of diassociative algebras was introduced by Loday in the study of Leibniz algebras \cite{loday-di}. The cohomology theory of diassociative algebras was developed by Frabetti \cite{frab} and further studied by Majumdar and Mukherjee \cite{maj-muk}.

\begin{defn}
A {\bf diassociative algebra} is a vector space $D$ equipped with two bilinear operations $\dashv, \vdash ~: D \times D \rightarrow D$ that satisfy the following five identities
\begin{align}\label{di-id}
\begin{cases}
&(a \dashv b ) \dashv c = a \dashv (b  \dashv c) =  a \dashv (b \vdash c),\\
&(a \vdash b) \dashv c = a \vdash (b \dashv c),\\
&(a \dashv b) \vdash c = (a \vdash b) \vdash c = a \vdash (b \vdash c),
\end{cases}
\end{align}
for $a,b, c \in D$. A diassociative algebra as above may be denoted by the triple $(D, \dashv, \vdash)$ or simply by $D$.
\end{defn}

It follows from the above definition that both the bilinear operations $\dashv$ and $\vdash$ in a diassociative algebra are associative products. Moreover, these two associative products additionally satisfy three associative-like identities. Thus, any associative algebra $A$ can be regarded as a diassociative algebra in which both the operations $\dashv$ and $\vdash$ are the given associative multiplication on $A$.

\medskip

Let $(D, \dashv, \vdash)$ be a diassociative algebra. A {\bf representation} of the diassociative algebra is a vector space $M$ equipped with four bilinear operations (called action maps)
\begin{align*}
\dashv ~: D \times M \rightarrow M, \quad \vdash ~: D \times M \rightarrow M, \quad \dashv ~: M \times D \rightarrow M ~ \text{ and } ~ \vdash ~: M \times D \rightarrow M
\end{align*}
that satisfy fifteen identities, where each set of five identities corresponds to the identities in (\ref{di-id}) with exactly one of $x, y, z$ replaced by an element of $M$. It follows that any diassociative algebra $(D, \dashv, \vdash)$ is a representation of itself, called the adjoint representation.

\medskip

Before we recall the cohomology of diassociative algebras, we need some notations about planar binary trees. A planar binary tree with $n$-vertices (often called an $n$-tree) is a planar tree with $(n+1)$ leaves, one root and each vertex trivalent. Let $Y_n$ be the set of all $n$-trees (for $n \geq 1$) and $Y_0$ be the set consisting of a root only, i.e.

\medskip

\medskip

\medskip

\begin{center}
$Y_0$ = \Bigg\{ \begin{tikzpicture}[scale=0.2]    
\draw (0,0) -- (0,4);
\end{tikzpicture} \Bigg\}, \qquad $Y_1$ = \Bigg\{  
\begin{tikzpicture}[scale=0.2]
\draw (-4,-2)-- (-2,0); \draw (-4,-2) -- (-4,-4); \draw (-6,0) -- (-4,-2);
\end{tikzpicture}
\Bigg\}, \qquad 
$Y_2$ = \Bigg\{ 
\begin{tikzpicture}[scale=0.2]
\draw (6,0) -- (8,-2);    \draw (8,-2) -- (10,0);     \draw (8,-2) -- (8,-4);     \draw (9,-1) -- ( 8,0);
\end{tikzpicture} ~,~ 
\begin{tikzpicture}[scale=0.2]
\draw (0,0) -- (2,-2); \draw (2,-2) -- (4,0); \draw (2,-2) -- (2,-4); \draw (1,-1) -- (2,0); 
\end{tikzpicture}
\Bigg\}, \\ \medskip \medskip $Y_3$ = \Bigg\{
\begin{tikzpicture}[scale=0.2]  
   \draw (36,0) -- (38, -2) ; \draw (38, -2) -- (38, -4) ; \draw (38, -2) -- (40, 0); \draw (37.33, 0) -- ( 38.67, - 1.33) ; \draw (38.66, 0) -- (39.34, -0.66);	
\end{tikzpicture} ~,~
\begin{tikzpicture}[scale=0.2]  
  \draw (30,0) -- (32,-2); \draw (32, -2) -- (32, -4); \draw (32,-2) -- (34, 0); \draw (31.33, 0) -- (32.67 , -1.33) ; \draw (32.66, 0) -- (32, -0.66); 
\end{tikzpicture} ~,~
\begin{tikzpicture}[scale=0.2] 
  \draw (24,0)-- (26,-2); \draw (26,-2)-- (26,-4); \draw (26,-2) -- (28,0); \draw (25.33, 0) -- (24.66, -0.66); \draw (26.66, 0) -- (27.34, -0.66); 
\end{tikzpicture} ~,~
\begin{tikzpicture}[scale=0.2]
 \draw (18,0) -- (20,-2); \draw (20, -2) -- (20, -4); \draw (20, -2) -- (22,0); \draw (19.33, -1.33) -- (20.66, 0); \draw (19.33, 0) -- (20, -0.66);
\end{tikzpicture} ~,~
 \begin{tikzpicture}[scale=0.2]
	     \draw (12,0)-- (14,-2); \draw (14,-2) -- (14,-4); \draw (14, -2) -- (16,0); \draw (12.7, - 0.7) -- (13.3333, 0) ; \draw (13.33, -1.33) -- (14.66, 0);
	     \end{tikzpicture} \Bigg\}, \quad \text{etc.}
\end{center}

\medskip

\medskip

\noindent Note that the cardinality of $Y_n$ is given by the $n$-th Catalan number $\frac{(2n) !}{(n+1)!~ n!}$. The grafting of a $m$-tree $y_1 \in Y_m$ and a $n$-tree $y_2 \in Y_n$ is a $(m+n+1)$-tree $y_1 \vee y_2 \in Y_{m+n+1}$ obtained by joining the roots of $y_1$ and $y_2$ to a new vertex and creating a new root from that vertex. Given an $n$-tree $y \in Y_n$, we label the $n+1$ leaves of $y$ by $\{ 0, 1, 2, \ldots, n \}$ from left to right. For each $n \geq 1$ and $0 \leq i \leq n$, there is a map (called the $i$-th face map) $d_i : Y_n \rightarrow Y_{n-1}$, $y \mapsto d_i y$ which is obtained by removing the $i$-th leaf of $y$. Also there are maps $\star_i : Y_n \rightarrow \{ \dashv, \vdash \}$, $y \mapsto \star_i^y$ (for $0 \leq i \leq n$) which are defined as follows:

\medskip

\medskip

\medskip

\begin{center}
\begin{tabular}{|c|c|}
\hline
$\star_0^y =$  & $\dashv ~~~~ \text{ if } y= | \vee y_1 \text{ for some } (n-1)\text{-tree } y_1$, \\
 & $\vdash ~~~~ \text{ otherwise},$ \\ \hline 
$\star_i^y =$ & $ ~\dashv ~~~~ \text{ if the } i\text{-th leaf of } y \text{ is oriented like } `\setminus',$ \\
(for $1 \leq i \leq n-1$) & $~ \vdash ~~~~ \text{ if the } i\text{-th leaf of } y \text{ is oriented like } `/',$ \\ \hline
$\star_n^y =$ & $ \vdash ~~~~ \text{ if } y= y_1 \vee | \text{ for some } (n-1)\text{-tree } y_1,$ \\
 & $\dashv ~~~~ \text{ otherwise}.$\\
\hline
\end{tabular}
\end{center}

\medskip

\medskip

\medskip

Let $(D, \dashv, \vdash)$ be a diassociative algebra and $M$ be a representation of it. For each $n \geq 0$, we define the space $CY^n (D,M)$ of all $n$-cochains by $CY^n (D,M) := \mathrm{Hom}({\bf k}[Y_n] \otimes D^{\otimes n}, M)$. Then there is a map $\delta_\mathrm{Diass} : CY^n(D,M) \rightarrow CY^{n+1}(D,M)$ given by
\begin{align*}
\delta_\mathrm{Diass}(f)& (y; a_1, \ldots, a_{n+1}) = a_1 \star_0^y f (d_0 y; a_2, \ldots, a_{n+1}) \\
&+ \sum_{i=1}^n (-1)^i ~ f (d_i y; a_1, \ldots, a_i \star_i^y a_{i+1}, \ldots, a_{n+1}) 
~+~ (-1)^{n+1} ~ f(d_{n+1} y; a_1, \ldots, a_n) \star_{n+1}^y a_{n+1},
\end{align*}
for $f \in CY^n (D,M)$, $y \in Y_{n+1}$ and $a_1, \ldots, a_{n+1} \in D.$ Then it has been shown by Frabetti \cite{frab} that $(\delta_\mathrm{Diass})^2 = 0$. In other words, $\{ CY^\bullet(D, M), \delta_\mathrm{Diass} \}$ is a cochain complex. The corresponding cohomology is called the {\bf cohomology} of the diassociative algebra $(D, \dashv, \vdash)$ with coefficients in the representation $M$. We denote the $n$-th cohomology group by $H^n_\mathrm{Diass} (D,M)$.

\medskip

In \cite{maj-muk} Majumdar and Mukherjee constructed a graded Lie algebra whose Maurer-Cartan elements correspond to diassociative algebra structures on a given vector space (see also \cite{yau}). To describe their graded Lie bracket in a more simple form, we define maps $R_0^{m; i, n} : Y_{m+n-1} \rightarrow Y_m$ and $R_i^{m; i, n} : Y_{m+n-1} \rightarrow Y_n$ (for $m, n \geq 1$ and $1 \leq i \leq m$) by
\begin{align}
R_0^{m; i, n} (y) = \widehat{d_0} \circ \widehat{d_1} \circ \cdots \circ \widehat{d_{i-1}} \circ d_i \circ \cdots \circ d_{i+n-2} \circ \widehat{d_{i + n-1}} \circ \cdots \circ \widehat{d_{m+n-1}} (y), \label{r0-maps}\\
R_i^{m; i, n} (y) = d_0 \circ d_1 \circ \cdots \circ d_{i-2} \circ \widehat{d_{i-1}} \circ \cdots \circ \widehat{d_{i+n-1}} \circ d_{i+n} \circ \cdots \circ d_{m+n-1} (y), \label{ri-maps}
\end{align}
where ~$~\widehat{~}~$~ means that the term is missing from the expression. Let $D$ be a vector space (not necessarily a diassociative algebra). They showed that the graded vector space $CY^\bullet (D, D) = \oplus_{n = 1}^\infty CY^n (D,D) = \oplus_{n = 1}^\infty \mathrm{Hom}({\bf k}[Y_n] \otimes D^{\otimes n}, D)$ inherits a degree $-1$ graded Lie bracket (which we call the Majumdar-Mukherjee bracket) given by
\begin{align}\label{mm-circ}
[f, g]_\mathsf{MM} := \big( \sum_{i=1}^m (-1)^{(i-1)(n-1)} f \circ_i g \big) - (-1)^{(m-1)(n-1)} \big( \sum_{i=1}^n (-1)^{(i-1)(m-1)} g \circ_i f \big),
\end{align}
where
\begin{align}\label{partial-comp}
(f \circ_i g) (y; a_1, \ldots, a_{m+n-1}) = f \big(  R_0^{m;i,n} (y) ; a_1, \ldots, a_{i-1}, g (R_i^{m;i,n} (y); a_i, \ldots, a_{i+n-1}), a_{i+n}, \ldots, a_{m+n-1} \big),
\end{align}
for $f \in CY^m(D,D)$, $g \in CY^n(D, D)$, $y \in Y_{m+n-1}$ and $a_1, \ldots, a_{m+n-1} \in D$ (see also \cite{yau,das-jpaa} for more details). In other words, $\big( CY^{\bullet +1} (D,D) = \oplus_{n = 0}^\infty CY^{n+1} (D,D) , [~,~]_\mathsf{MM}  \big)$ is a graded Lie algebra. An element $\pi \in CY^2 (D,D)$ determines (and determined by) two bilinear maps $\dashv , \vdash : D \times D \rightarrow D$ given by 
\begin{align}
a \dashv b = \pi \big( \begin{tikzpicture}[scale=0.1]
\draw (6,0) -- (8,-2);  \draw (8,-2) -- (10,0); \draw (8,-2) -- (8,-4);   \draw (9,-1) -- ( 8,0);
\end{tikzpicture} ; a, b \big) \quad \text{ and } \quad a \vdash b = \pi \big( \begin{tikzpicture}[scale=0.1]
\draw (0,0) -- (2,-2);  \draw (2,-2) -- (4,0);  \draw (2,-2) -- (2,-4);    \draw (1,-1) -- (2,0);   
\end{tikzpicture} ; a, b \big), \text{ for } a, b \in D. 
\end{align}

\medskip

\noindent Then it has been shown in \cite{maj-muk} that $\pi$ defines a Maurer-Cartan element in the above-graded Lie algebra if and only if $(\dashv, \vdash)$ defines a diassociative algebra structure on $D$.


\begin{remark}
Let $(D, \dashv, \vdash)$ be a diassociative algebra. Consider the corresponding Maurer-Cartan element $\pi$ in the graded Lie algebra $ \big( CY^{\bullet +1} (D,D) = \oplus_{n = 0}^\infty CY^{n+1} (D,D) , [~,~]_\mathsf{MM}  \big)$. Then the coboundary map $\delta_\mathrm{Diass} : CY^n (D,D) \rightarrow CY^{n+1}(D,D)$ of the diassociative algebra $D$ with coefficients in the adjoint representation is simply given by 
\begin{align*}
\delta_\mathrm{Diass} (f) = (-1)^{n-1} [\pi, f]_\mathsf{MM}, \text{ for } f \in CY^n (D,D).
\end{align*}
\end{remark}

\medskip

\section{Relative averaging operators and relative averaging algebras}\label{sec-3}

In this section, we first introduce relative averaging operators, relative averaging algebras and provide various examples. Next, we consider the close relationship between relative averaging algebras and diassociative algebras. Finally, we define and study bimodules over relative averaging algebras.

\begin{defn}\label{first-defn}
(i) Let $A$ be an associative algebra and $M$ be an $A$-bimodule. A {\bf relative averaging operator} on $M$ over the algebra $A$ is a linear map $P: M \rightarrow A$ that satisfies
\begin{align}\label{avg-iden}
P(u) \cdot P(v) = P \big( P(u) \cdot_M v  \big) = P (u \cdot_M P(v)), \text{ for } u, v \in M.
\end{align}
Here $\cdot$ denotes the associative multiplication on $A$ and $\cdot_M$ denotes both the left and right $A$-actions on $M$.

\medskip

(ii) A {\bf relative averaging algebra} is a triple $(A, M, P)$ consisting of an associative algebra $A$, an $A$-bimodule $M$ and a relative averaging operator $P: M \rightarrow A$.
\end{defn}

For our convenience, we denote a relative averaging algebra $(A, M, P)$ by the notation $M \xrightarrow{P} A$. Hence $(A, M, P)$ and $M \xrightarrow{P} A$ represent the same mathematical structure.

\begin{defn}\label{defn-morp}
Let $M \xrightarrow{P} A$ and $M' \xrightarrow{P'} A'$ be two relative averaging algebras. A {\bf morphism} of relative averaging algebras from $M \xrightarrow{P} A$ to $M' \xrightarrow{P'} A'$ consists of a pair $(\varphi, \psi)$ of an algebra morphism $\varphi : A \rightarrow A'$ and a linear map $\psi : M \rightarrow M'$ satisfying
\begin{align*}
\psi (a \cdot_M u) = \varphi(a) \cdot_{M'}^{A'} \psi (u), \quad \psi (u \cdot_M a) = \psi (u) \cdot_{M'}^{A'} \varphi(a) ~~~~ \text{ and } ~~~~ \varphi \circ P = P' \circ \psi, ~ \text{ for all } a \in A, u \in M.
\end{align*}
Here $\cdot_{M'}^{A'}$ denotes both the left and right $A'$-actions on $M'$. It is said to be an {\bf isomorphism} if both $\varphi$ and $\psi$ are linear isomorphisms.
\end{defn}

\begin{exam}
Any averaging algebra $(A, P)$ can be regarded as a relative averaging algebra $A \xrightarrow{P} A$, where $A$ is equipped with the adjoint $A$-bimodule structure. Thus, a relative averaging algebra is a generalization of an averaging algebra.
\end{exam}

\begin{exam}
Let $A$ be an associative algebra. Then the tensor product $A \otimes A$ can be equipped with an $A$-bimodule structure with the left and right $A$-actions respectively given by
\begin{align*}
c \cdot_{A \otimes A} (a \otimes b) = c \cdot a \otimes b ~~~ \text{ and } ~~~ (a \otimes b) \cdot_{A \otimes A} c = a \otimes b \cdot c, \text{ for }  a \otimes b \in A \otimes A,  c \in A.
\end{align*}
Consider the map $P: A \otimes A \rightarrow A$ given by $P(a \otimes b) = a \cdot b$, for $a \otimes b \in A \otimes A$. For any $a \otimes b,~ a' \otimes b' \in A \otimes A$, we have
\begin{align*}
P (a \otimes b) \cdot P (a' \otimes b') = a \cdot b \cdot a' \cdot b' = \begin{cases}
= P \big(  a \cdot b \cdot a' \otimes b' \big) = P \big( P (a \otimes b) \cdot_{A \otimes A} (a' \otimes b') \big),\\
= P \big( a \otimes b \cdot a' \cdot b'  \big) = P \big( (a \otimes b) \cdot_{A \otimes A} P (a' \otimes b')  \big).
\end{cases}
\end{align*}
This shows that $P : A \otimes A \rightarrow A$ is a relative averaging operator. Thus, $A \otimes A \xrightarrow{P} A$ is a relative averaging algebra.
\end{exam}

\begin{exam}
Let $A$ be an associative algebra. Then the space $\underbrace{A \oplus \cdots \oplus A}_{n \text{ copies}}$ is an $A$-bimodule where the left (resp. right) $A$-action is given by componentwise left (resp. right) multiplication map. Then it is easy to see that the map
\begin{align*}
P : A \oplus \cdots \oplus A \rightarrow A, ~~ P\big((a_1, \ldots, a_n)\big) = a_1 + \cdots + a_n, \text{ for } (a_1, \ldots, a_n) \in A \oplus \cdots \oplus A
\end{align*}
is a relative averaging operator. In other words, $A \oplus \cdots \oplus A \xrightarrow{P} A$ is a relative averaging algebra.
\end{exam}

\begin{exam}
Let $A$ be an associative algebra. Then for any $1 \leq i \leq n$, the $i$-th projection map $P_i : A \oplus \cdots \oplus A \rightarrow A$, $(a_1, \ldots, a_n) \mapsto a_i$ is a relative averaging operator. That is, $A \oplus \cdots \oplus A \xrightarrow{P_i} A$ is a relative averaging algebra.
\end{exam}

\begin{exam}
Let $A$ be an associative algebra and $M$ be an $A$-bimodule. Suppose $G$ is a finite group and there are maps $G \times A \rightarrow A$, $(g, a) \mapsto {}^g a$ and $G \times M \rightarrow A$, $(g, u) \mapsto {}^g u$ 
that satisfy
\begin{align*}
{}^g(a \cdot_M u) = {}^g a \cdot {}^g u, \quad {}^g (u \cdot_M a) = {}^g u \cdot {}^g a ~~~~~ \text{ and } ~~~~~ {}^g ({}^h u) = {}^{gh} u,
\end{align*}
for $a \in A$, $u \in M$ and $g, h \in G$. We define a map $P: M \rightarrow A$ by $P(u) = \sum_{g \in G} {}^g u$. For any $u,v \in M$, we observe that
\begin{align*}
P \big( P(u) \cdot_M v \big) = \sum_{h \in G} {}^h \big( (\sum_{g \in G} {}^g u) \cdot_{M} v \big) = \sum_{h \in G} \big(    \sum_{g \in G} {}^{hg}u \big) \cdot {}^h v = \big(  \sum_{g \in G} {}^g u \big)  \cdot \big( \sum_{h \in G} {}^h v  \big) = P(u) \cdot P(v),\\
P \big( u \cdot_M P(v) \big) = \sum_{h \in G} {}^h \big( u \cdot_M (\sum_{g \in G} {}^g v)  \big) = \sum_{h \in G} {}^h u \cdot \big( \sum_{h \in G} {}^{hg} v \big) = \big( \sum_{h \in G} {}^h u  \big) \cdot \big( \sum_{g \in G} {}^g v  \big) = P(u) \cdot P(v).
\end{align*}
This shows that $P: M \rightarrow A$ is a relative averaging operator, equivalently, $M \xrightarrow{P} A$ is a relative averaging algebra.
\end{exam}

\begin{exam}\label{mod-map}
Let $A$ be an associative algebra and $M$ be an $A$-bimodule. Let $P:M \rightarrow A$ be an $A$-bimodule map, i.e.
\begin{align*}
P (a \cdot_M u) = a \cdot P(u) ~~~ \text{ and } ~~~ P (u \cdot_M a) = P(u) \cdot a, ~ \text{ for } a \in A,~ u \in M.
\end{align*}
Then it is easy to see that $M \xrightarrow{P} A$ is a relative averaging algebra.
\end{exam}

\begin{exam}
In \cite{loday-lm} Loday and Pirashvili introduced the category $\mathcal{LM}$ whose objects are linear maps between vector spaces. In other words, an object in $\mathcal{LM}$ is of the form $V \xrightarrow{f} W$, where $V, W$ are vector spaces and $f$ is a linear map. They equip $\mathcal{LM}$ with a tensor product which makes it a tensor category. It has been observed that an associative object in $\mathcal{LM}$ is given by a datum $M \xrightarrow{f} A$, where $A$ is an associative algebra, $M$ is an $A$-bimodule and $f$ is an $A$-bimodule map. Thus, it turns out that an associative object in $\mathcal{LM}$ is a relative averaging algebra.
\end{exam}

\begin{exam}(Crossed modules of associative algebras \cite{wagemann}) A crossed module of associative algebras is a quadruple $(A, M, \cdot_M, d)$ in which $A, M$ are both associative algebras and $M$ is also equipped with an $A$-bimodule structure (with both the left and right $A$-actions on $M$ being denoted by $\cdot_M$) and $d: M \rightarrow A$ is an algebra morphism that satisfy
\begin{align*}
d (a \cdot_M u) = a \cdot du, \quad d(u \cdot_M a)= du \cdot a, \quad (du) \cdot_M v = u \cdot_M (dv) = u \diamond v, \text{ for } a \in A, u, v \in M.
\end{align*}
Here $\diamond$ denotes the associative multiplication on $M$.
Thus, it follows from Example \ref{mod-map} that  $M \xrightarrow{d} A$ is a relative averaging algebra.
\end{exam}
It has been observed in \cite{wagemann,baez} that crossed modules of associative algebras are equivalent to `strict' associative $2$-algebras. Hence by following the previous example, one can construct relative averaging algebras from strict associative $2$-algebras.


\medskip

In the following, we give some characterizations of relative averaging operators. We start with the following useful result.

\begin{prop}\label{am-diass}
Let $A$ be an associative algebra and $M$ be an $A$-bimodule. Then the direct sum $A \oplus M$ inherits a diassociative algebra structure with the operations
\begin{align*}
(a, u) \dashv (b, v) = (a \cdot b, u \cdot_M b) ~~~~ \text{ and } ~~~~ (a, u) \vdash (b, v) = (a \cdot b, a \cdot_M v), \text{ for } (a, u), (b, v) \in A \oplus M.
\end{align*}
We denote this diassociative algebra simply by $A \oplus_\mathrm{Diass} M.$
\end{prop}

\begin{proof}
For any $(a, u), (b, v), (c, w) \in A \oplus M$, we have
\begin{align*}
\big( (a, u) \dashv (b, v) \big) \dashv (c, w) =~& \big(  a \cdot b, u \cdot_M b \big) \dashv (c, w) = \big(  (a \cdot b ) \cdot c, (u \cdot_M b) \cdot_M c  \big),\\
 (a, u) \dashv \big( (b, v) \dashv (c, w)\big) =~& (a, u) \dashv \big(  b \cdot c, v \cdot_M c \big) = \big(  a \cdot (b  \cdot c), u \cdot_M (b \cdot c)  \big), \\
  (a, u) \dashv \big( (b, v) \vdash (c, w)\big) =~& (a, u) \dashv \big(  b \cdot c, b \cdot_M w \big) = \big(  a \cdot (b  \cdot c), u \cdot_M (b \cdot c)  \big).
\end{align*} 
Thus, it follows that
\begin{align*}
\big( (a, u) \dashv (b, v) \big) \dashv (c, w) =  (a, u) \dashv \big( (b, v) \dashv (c, w)\big) =   (a, u) \dashv \big( (b, v) \vdash (c, w)\big).
\end{align*}
Similarly, one can show that
\begin{align*}
  \big( (a, u) \vdash  (b, v) \big) \dashv (c, w) =   (a, u) \vdash \big( (b, v) \dashv (c, w)\big),
\end{align*}
\begin{align*}
 \big( (a, u) \dashv  (b, v) \big) \vdash (c, w) =  \big( (a, u) \vdash  (b, v) \big) \vdash (c, w) = (a, u) \vdash \big( (b, v) \vdash (c, w)\big).
\end{align*}
This completes the proof.
\end{proof}


\begin{prop}\label{grph}
Let $A$ be an associative algebra and $M$ be an $A$-bimodule. A linear map $P: M \rightarrow A$ is a relative averaging operator (on $M$ over the algebra $A$) if and only if the graph $\mathrm{Gr}(P) = \{ (P(u), u) | u \in M \}$ is a subalgebra of the diassociative algebra $A \oplus_\mathrm{Diass} M$.
\end{prop}

\begin{proof}
For any $u, v \in M$, we have $(P(u), u) \dashv (P(v), v) = (P(u) \cdot P(v), u \cdot_M P(v))$. This is in $\mathrm{Gr}(P)$ if and only if $P(u) \cdot P(v) = P \big( u \cdot_M P(v) \big)$. Similarly, the product $(P(u), u) \vdash (P(v), v) = (P(u) \cdot P(v), P(u) \cdot_M v)$ is in $\mathrm{Gr}(P)$ if and only if $P(u) \cdot P(v) = P \big( P(u) \cdot_M v \big)$. In other words, $\mathrm{Gr}(P)$ is a subalgebra of the diassociative algebra $A \oplus_\mathrm{Diass} M$ if and only if $P$ is a relative averaging operator.
\end{proof}

Let $(D, \dashv, \vdash)$ be a diassociative algebra. A linear map $\mathcal{N}: D \rightarrow D$ is said to be a {\bf Nijenhuis operator} on the diassociative algebra $D$ if for all $a, b \in D,$
\begin{align*}
\mathcal{N}(a) \dashv \mathcal{N}(b) := \mathcal{N} \big( \mathcal{N}(a) \dashv b ~+~ a \dashv \mathcal{N}(b) - \mathcal{N} (a \dashv b) \big),\\
 \mathcal{N}(a) \vdash \mathcal{N}(b) := \mathcal{N} \big( \mathcal{N}(a) \vdash b ~+~ a \vdash \mathcal{N}(b) - \mathcal{N} (a \vdash b) \big).
\end{align*}

\begin{prop}\label{avg-nij}
Let $A$ be an associative algebra and $M$ be an $A$-bimodule. A linear map $P: M \rightarrow A$ is a relative averaging operator (on $M$ over the algebra $A$) if and only if the map $\mathcal{N}_P : A \oplus M \rightarrow A \oplus M$, $\mathcal{N}_P ((a, u)) = (P(u), 0)$ for $(a, u) \in A \oplus M$, is a Nijenhuis operator on the diassociative algebra $A \oplus_\mathrm{Diass} M$.
\end{prop}

\begin{proof}
For any $(a, u), (b, v) \in A \oplus M$, we have
\begin{align*}
\mathcal{N}_P (a, u) \dashv \mathcal{N}_P (b, v) = (P(u), 0) \dashv (P(v), 0) = \big( P(u) \cdot P(v), 0 \big).
\end{align*}
On the other hand, we have
\begin{align*}
&\mathcal{N}_P  \big( \mathcal{N}_P (a, u) \dashv (b, v) ~+~ (a, u) \dashv \mathcal{N}_P  (b, v) - \mathcal{N}_P  ((a, u) \dashv (b, v)) \big) \\
&= \mathcal{N}_P  \big(  (P(u) \cdot b, 0) ~+~ (a \cdot P(v), u \cdot_M P(v))  - (P( u \cdot_M b), 0)  \big) = \big( P(u \cdot_M P(v)), 0  \big).
\end{align*}
This shows that $\mathcal{N}_P (a, u) \dashv \mathcal{N}_P (b, v) = \mathcal{N}_P  \big( \mathcal{N}_P (a, u) \dashv (b, v) ~+~ (a, u) \dashv \mathcal{N}_P  (b, v) - \mathcal{N}_P  ((a, u) \dashv (b, v))  \big)$ holds if and only if $P(u) \cdot P(v) = P(u \cdot_M P(v))$. Similarly, one can show that 
\begin{align*}
\mathcal{N}_P (a, u) \vdash \mathcal{N}_P (b, v) = \mathcal{N}_P  \big(  \mathcal{N}_P (a, u) \vdash (b, v) ~+~ (a, u) \vdash \mathcal{N}_P  (b, v) - \mathcal{N}_P  ((a, u) \vdash (b, v))  \big)
\end{align*}
holds if and only if $P(u) \cdot P(v) = P(P(u) \cdot_M v)$, for $u, v \in M$. Combining these, we get that $\mathcal{N}_P$ is a Nijenhuis operator on the diassociative algebra $A \oplus_\mathrm{Diass} M$ if and only if $P$ is a relative averaging operator.
\end{proof}

\subsection*{Free relative averaging algebra}

Let $V \xrightarrow{f} W$ be a $2$-term chain complex. Consider the tensor algebra of $W$,
\begin{align*}
T(W) = {\bf k} \oplus W \oplus W^{\otimes 2} \oplus \cdots \oplus W^{\otimes n} \oplus \cdots
\end{align*}
with the concatenation product. Then the space $T(W) \otimes V \otimes T(W)$ can be given a $T(W)$-bimodule structure with the left and right actions given by
\begin{align*}
(w'_1 \cdots w'_p) \cdot (w_{-m} \cdots w_{-1} \otimes v_0 \otimes w_1 \cdots w_n) = w'_1 \cdots w'_p w_{-m} \cdots w_{-1} \otimes v_0 \otimes w_1 \cdots w_n,\\
 (w_{-m} \cdots w_{-1} \otimes v_0 \otimes w_1 \cdots w_n) \cdot (w'_1 \cdots w'_p) = w_{-m} \cdots w_{-1} \otimes v_0 \otimes w_1 \cdots w_n w'_1 \cdots w'_p,
\end{align*}
for $w'_1 \cdots w'_p \in T(W)$ and $w_{-m} \cdots w_{-1} \otimes v_0 \otimes w_1 \cdots w_n \in T(W) \otimes V \otimes T(W).$
Then it is easy to see that $T(W) \otimes V \otimes T(W) \xrightarrow{\mathcal{P}(f)} T(W)$ is a relative averaging algebra, where
\begin{align*}
\mathcal{P}(f) ( w_{-m} \cdots w_{-1} \otimes v_0 \otimes w_1 \cdots w_n  ) = w_{-m} \cdots w_{-1} f(v_0) w_1 \cdots w_n.
\end{align*}

\begin{remark}
Let $V$ be any vector space. Consider the $2$-term chain complex $V \xrightarrow{\mathrm{id}_V} V$. Then it follows that $T(V) \otimes V \otimes T(V) \xrightarrow{\mathcal{P}(\mathrm{id}_V)} T(V)$ is a relative averaging algebra.
\end{remark}

\begin{defn}
Let $V \xrightarrow{f} W$ be a $2$-term chain complex. The {\bf free relative averaging algebra} over $V \xrightarrow{f} W$ is a relative averaging algebra $\mathcal{M}(V) \xrightarrow{ \mathcal{P}(f) } \mathcal{A}(W)$ equipped with a morphism $(i,j)$ of complexes from $V \xrightarrow{f} W$ to $\mathcal{M}(V) \xrightarrow{ \mathcal{P}(f) } \mathcal{A}(W)$ that satisfy the following universal condition:
\begin{center}
for any relative averaging algebra $M \xrightarrow{P} A$ and a morphism    $(\varphi, \psi)$ of complexes from $V \xrightarrow{f} W$ to $M \xrightarrow{P} A$, there exists a morphism $(\widetilde{\varphi}, \widetilde{\psi})$ of relative averaging algebras from $\mathcal{M}(V) \xrightarrow{ \mathcal{P}(f) } \mathcal{A}(W)$ to $M \xrightarrow{P} A$ that makes the following diagram commutative
\[
\xymatrix{
M \ar[rrr]^P & & & A \\
 & V \ar[lu]^\psi \ar[ld]_j \ar[r]^f & W \ar[ru]_\varphi \ar[rd]^i & \\
\mathcal{M}(V) \ar@{..>}[uu]^{\widetilde{\psi}} \ar[rrr]_{\mathcal{P}(f)} & & & \mathcal{A}(W). \ar@{..>}[uu]_{\widetilde{\varphi}}
}
\]
\end{center}
\end{defn}

\medskip

\begin{prop}
Let $V \xrightarrow{f} W$ be a $2$-term chain complex. Then the relative averaging algebra \\
\begin{center}
$T(W) \otimes V \otimes T(W) \xrightarrow{\mathcal{P}(f)} T(W)$ is free over the chain complex $V \xrightarrow{f} W$.
\end{center}
\end{prop}

\begin{proof}
We define maps $i : W \rightarrow T(W)$ and $j: V \rightarrow T(W) \otimes V \otimes T(W)$ by
\begin{align*}
i(w) = w  ~~~ \text{ and } ~~~ j (v) = 1 \otimes v \otimes 1, \text{ for } w \in W, v \in V.
\end{align*}
Let $M \xrightarrow{P} A$ be any relative averaging algebra and $(\varphi, \psi)$ be a morphism of complexes from $V \xrightarrow{f} W$ to $M \xrightarrow{P} A$. We define maps $\widetilde{\varphi} : T(W) \rightarrow A$ and $\widetilde{\psi} : T(W) \otimes V \otimes T(W) \rightarrow M$ by
\begin{align*}
\widetilde{\varphi} (w_1 \cdots w_n) =~& \varphi(w_1) \cdots \varphi(w_n),\\
\widetilde{\psi} ( w_{-m} \cdots w_{-1} \otimes v_0 \otimes w_1 \cdots w_n) =~& \big(  \varphi( w_{-m}) \cdots \varphi(w_{-1}) \big) \cdot_M \psi (v_0) \cdot_M \big( \varphi( w_1) \cdots \varphi(w_n) \big),
\end{align*}
for $w_1 \cdots w_n \in T(W)$ and $w_{-m} \cdots w_{-1} \otimes v_0 \otimes w_1 \cdots w_n \in T(W) \otimes V \otimes T(W)$.
Then it is easy to see that the pair $(\widetilde{\varphi}, \widetilde{\psi})$ is a morphism of relative averaging algebras from $T(W) \otimes V \otimes T(W) \xrightarrow{\mathcal{P}(f)} T(W)$ to $M \xrightarrow{P} A$ and satisfies the universal condition.
\end{proof}

\subsection*{Functorial relations with diassociative algebras}

\begin{prop}
(i) Let $M \xrightarrow{P} A$ be a relative averaging algebra. Then the vector space $M$ carries a diassociative algebra structure with the bilinear operations
\begin{align}\label{mp}
u \dashv_P v := u \cdot_M P(v) ~~~~~ \text{ and } ~~~~~ u \vdash_P v := P(u) \cdot v, \text{ for } u, v \in M.
\end{align}
We denote this diassociative algebra simply by $M_P$.

(ii) Let $M \xrightarrow{P} A$ and $M' \xrightarrow{P'} A'$ be two relative averaging algebras and $(\varphi, \psi)$ be a morphism between them. Then $\psi : M \rightarrow M'$ is a morphism between induced diassociative algebras (from $M_P$ to $M'_{P'}$).
\end{prop}

\begin{proof}
(i) Since $P: M \rightarrow A$ is a relative averaging operator, it follows from Proposition \ref{grph} that $\mathrm{Gr}(P)$ is a subalgebra of the diassociative algebra $A \oplus_\mathrm{Diass} M$. The inherited diassociative structure on $\mathrm{Gr}(P)$ is given by
\begin{align*}
(P(u), u) \dashv (P(v), v) = \big( \equalto{P(u) \cdot P(v)}{P ( u \cdot_M P(v) )} , u \cdot_M P(v)   \big) ~ \text{ and } ~ (P(u), u) \vdash (P(v), v) = \big( \equalto{P(u) \cdot P(v)}{P( P(u) \cdot_M v ) } , P(u) \cdot_M v   \big),
\end{align*}
for $u, v \in M$. As the vector space $M$ is isomorphic to $\mathrm{Gr}(P)$ via $u \leftrightsquigarrow (P(u), u)$, for $u \in M$, we have a diassociative algebra structure on $M$ which is precisely given by (\ref{mp}).


(ii) For any $u, v \in M$, we have
\begin{align*}
\psi (u \dashv_P v) = \psi (u \cdot_M P(v)) = \psi (u) \cdot^{A'}_{M'} \varphi P(v) = \psi (u) \cdot^{A'}_{M'} P' (\psi(v)) = \psi (u) \dashv_{P'} \psi(v),\\
\psi (u \vdash_P v) = \psi (P(u) \cdot_M v) = \varphi P(u) \cdot_{M'}^{A'} \psi (v) = P' (\psi (u)) \cdot_{M'}^{A'}  \psi (v) = \psi (u) \vdash_{P'} \psi (v).
\end{align*}
This proves that $\psi : M_P \rightarrow M'_{P'}$ is a morphism of diassociative algebras.
\end{proof}

The above proposition shows that there is a functor $\mathcal{F} : {\bf  rAvg} \rightarrow {\bf Diass}$ from the category of relative averaging algebras to the category of diassociative algebras. In the following, we will construct a functor in the other direction.

Let $(D, \dashv, \vdash)$ be a diassociative algebra. Let $D_{\mathrm{Ass}}$ be the quotient of $D$ by the ideal generated by the elements $a \dashv b - a \vdash b$, for $a , b \in D$. Then $D_\mathrm{Ass}$ is an associative algebra, where the product is given by $[a] \cdot [b] := [a \dashv b] = [a \vdash b] $, for $[a], [b] \in D_\mathrm{Ass}$. Moreover, the vector space $D$ is a $D_\mathrm{Ass}$-bimodule, where the left and right $D_\mathrm{Ass}$-actions on $D$ are respectively given by
\begin{align*}
[a] \cdot_D b = a \vdash b ~~~~ \text{ and } ~~~~ b \cdot_D [a] = b \dashv a, \text{ for } [a] \in D_\mathrm{Ass}, b \in D.
\end{align*}
With these notations, the quotient map $q: D \rightarrow D_\mathrm{Ass}$ is a relative averaging operator as
\begin{align*}
q(a) \cdot q(b) =~& [a] \cdot [b]
= \begin{cases}
= [a \vdash b] = [[ a] \cdot_D b] = q \big( q(a) \cdot_D b \big), \\
= [a \dashv b] = [a \cdot_D [b]] = q \big( a \cdot_D q(b) \big),
\end{cases}
\end{align*}
for $a, b \in D$. Thus, $D \xrightarrow{q} D_\mathrm{Ass}$ is a relative averaging algebra. Moreover, the induced diassociative algebra structure on $D$ coincides with the given one, as
\begin{align*}
a \dashv_q b = a \cdot_D q(b) = a \dashv b ~~~~ \text{ and } ~~~~ a \vdash_q b = q(a) \cdot_D b = a \vdash b, \text{ for } a,b \in D.
\end{align*}

Let $(D, \dashv, \vdash)$ and $(D', \dashv', \vdash')$ be two diassociative algebras and $\psi : D \rightarrow D'$ be a morphism between them. Then it is easy to verify that the pair $(\varphi, \psi)$ is a morphism of relative averaging algebras from $D \xrightarrow{q} D_\mathrm{Ass}$ to $D' \xrightarrow{q'} D'_\mathrm{Ass}$, where $\varphi : D_\mathrm{Ass} \rightarrow D'_\mathrm{Ass}$ is given by $\varphi ([a]) = [\psi (a)]$, for $[a] \in D_\mathrm{Ass}$. This construction yields a functor $\mathcal{G} : {\bf Diass} \rightarrow {\bf rAvg}$ from the category of diassociative algebras to the category of relative averaging algebras.

\begin{prop}\label{l-adj-fun}
The functor $\mathcal{G}: {\bf Diass} \rightarrow {\bf rAvg}$ is left adjoint to the functor $\mathcal{F} : {\bf rAvg} \rightarrow {\bf Diass}$. More precisely, for any diassociative algebra $(D, \dashv, \vdash)$ and a relative averaging algebra $M \xrightarrow{P} A$, we have
\begin{align*}
\mathrm{Hom}_\mathrm{\bf Diass} (D, M_P) ~ \cong ~ \mathrm{Hom}_{\bf rAvg} (D \xrightarrow{q} D_\mathrm{Ass} , M \xrightarrow{P} A).
\end{align*}
\end{prop}

\begin{proof}
Let $\psi \in \mathrm{Hom}_\mathrm{\bf Diass} (D, M_P).$ We define a map $\varphi^\psi : D_\mathrm{Ass} \rightarrow A$ by $\varphi^\psi ([a]) = P (\psi (a))$, for $[a] \in D_\mathrm{Ass}$. Then it is easy to see that $\varphi^\psi$ is an algebra morphism. Moreover, we have
\begin{align*}
\psi ([a] \cdot_D b) = \psi (a \vdash b) = \psi (a) \vdash_P \psi (b) = P \psi (a) \cdot_M \psi(b) = \varphi^\psi ([a]) \cdot_M \psi (b),\\
\psi (b \cdot_D [a]) = \psi (b \dashv a) = \psi (b) \dashv_P \psi (a) = \psi (b) \cdot_M P\psi (a) = \psi (b) \cdot_M \varphi^\psi ([a]),
\end{align*}
for $[a] \in D_\mathrm{Ass}$, $b \in D$. Further, $\varphi^\psi \circ q = P \circ \psi$. Thus, $(\varphi^\psi, \psi) \in \mathrm{Hom}_{\bf rAvg} (D \xrightarrow{q} D_\mathrm{Ass} , M \xrightarrow{P} A).$ 

On the other hand, if $(\varphi, \psi) \in \mathrm{Hom}_{\bf rAvg} (D \xrightarrow{q} D_\mathrm{Ass} , M \xrightarrow{P} A)$, then $\psi \in \mathrm{Hom}_\mathrm{\bf Diass} (D, M_P)$. Finally, the above two correspondences are inverses to each other. 
\end{proof}

\subsection*{Bimodules over relative averaging algebras} Here we introduce bimodules over relative averaging algebras. We show that a bimodule over a relative averaging algebra gives two representations of the induced diassociative algebra.

\begin{defn}
Let $M \xrightarrow{P} A$ be a relative averaging algebra. A {\bf bimodule} over it consists of a tuple $(N \xrightarrow{Q} B, l, r)$ in which $N \xrightarrow{Q} B$ is a $2$-term chain complex with both $B$ and $N$ are $A$-bimodules, and  $l: M \times B \rightarrow N$ and $r: B \times M \rightarrow N$ are bilinear maps (called the pairing maps) satisfying
\begin{align}
&l (a \cdot_M u, b) = a \cdot_N l(u, b), \quad l (u \cdot_M a, b) = l (u, a \cdot_B b), \quad l(u, b \cdot_B a)= l(u, b) \cdot_N a,\label{dasm1}\\
&r(a \cdot_B b, u) = a \cdot_N r(b, u), \quad r (b \cdot_B a, u) = r (b, a \cdot_M u), \quad r (b, u \cdot_M a) = r(b, u) \cdot_N a, \label{dasm2}
\end{align}
and
\begin{align}
&P(u) \cdot_B Q(n) = Q \big( P(u) \cdot_N n  \big) = Q \big( l (u, Q(n)) \big), \label{dasm3}\\
&Q(n) \cdot_B P(u) = Q \big(  r (Q(n), u) \big) = Q \big( n \cdot_N P(u)  \big), \label{dasm4}
\end{align}
for $a \in A$, $b \in B$, $u \in M$ and $n \in N$. Sometimes we denote a bimodule as above by the complex $N \xrightarrow{Q} B$ when the bilinear maps $l$ and $r$ are clear from the context.
\end{defn}

\begin{exam}\label{adj-bimod}
(Adjoint bimodule) Let $M \xrightarrow{P} A$ be a relative averaging algebra. Then it is easy to see that the tuple $(M \xrightarrow{P} A, l_\mathrm{ad}, r_\mathrm{ad})$ is a bimodule over the relative averaging algebra $M \xrightarrow{P} A$, where the pairing maps $l_\mathrm{ad} : M \times A \rightarrow M$ and $r_\mathrm{ad} : A \times M \rightarrow M$ are respectively the given right and left $A$-actions on $M$. This is called the adjoint bimodule.
\end{exam}

\begin{exam}\label{avg-alg-bimod}
(Bimodule over an averaging algebra \cite{wang-av}) Let $(A, P)$ be an averaging algebra. A bimodule over it consists of a pair $(M, Q)$ in which $M$ is an $A$-bimodule and $Q: M \rightarrow M$ is a linear map satisfying for $a \in A$, $u \in M$,
\begin{align*}
P(a) \cdot_M Q(u) =~& Q (P(a) \cdot_M u ) = Q (a \cdot_M Q(u)),\\
Q(u) \cdot_M P(a) =~& Q (Q(u) \cdot_M a ) = Q (u \cdot_M P(a)).
\end{align*}
This is equivalent to the fact that the tuple $(M \xrightarrow{Q} M, \cdot_M, \cdot_M)$ is a bimodule over the relative averaging algebra $A \xrightarrow{P} A.$
\end{exam}

Let $A$ be an associative algebra. Given an element ${\bf r} = \sum r_{(1)} \otimes r_{(2)} \in A \otimes A$, we consider the following three elements
\begin{align*}
{\bf r}_{13} {\bf r}_{12} =& \sum r_{(1)} \cdot \widetilde{r}_{(1)} \otimes \widetilde{r}_{(2)} \otimes r_{(2)}, \qquad {\bf r}_{12} {\bf r}_{23} = \sum r_{(1)} \otimes r_{(2)} \cdot \widetilde{r}_{(1)} \otimes \widetilde{r}_{(2)} \\
&\text{ and } ~~~ {\bf r}_{23} {\bf r}_{13} = \sum r_{(1)} \otimes \widetilde{r}_{(1)} \otimes \widetilde{r}_{(2)} \cdot r_{(2)} ~~ \text{ of }  ~A \otimes A \otimes A.
\end{align*}
Here $\sum \widetilde{r}_{(1)} \otimes \widetilde{r}_{(2)}$ is another copy of ${\bf r}$. An element ${\bf r} \in A \otimes A$ is called an {\bf averaging element} if it satisfies
\begin{align}
{\bf r}_{13} {\bf r}_{12} = {\bf r}_{12} {\bf r}_{23} = {\bf r}_{23} {\bf r}_{13}.
\end{align}
Let $r = \sum r_{(1)} \otimes r_{(2)} \in A \otimes A$ be an averaging element. Then the map $P : A \rightarrow A$ defined by $P(a) = \sum r_{(1)} \cdot a \cdot r_{(2)}$, for $a \in A$, is an averaging operator on $A$. To see this, we observe that
\begin{align*}
P(a) \cdot P(a') =~& \sum r_{(1)} \cdot a \cdot r_{(2)} \cdot \widetilde{r}_{(1)} \cdot a' \cdot \widetilde{r}_{(2)} \\
=~& \begin{cases}  
 = \sum r_{(1)} \cdot \widetilde{r}_{(1)} \cdot a \cdot \widetilde{r}_{(2)} \cdot a' \cdot r_{(2)}  \quad (\text{since } {\bf r}_{13} {\bf r}_{12} = {\bf r}_{12} {\bf r}_{23}) ~~= P (P(a) \cdot a'),\\
 = \sum r_{(1)} \cdot a \cdot \widetilde{r}_{(1)} \cdot a' \cdot \widetilde{r}_{(2)} \cdot r_{(2)} \quad (\text{since } {\bf r}_{12} {\bf r}_{23} = {\bf r}_{23} {\bf r}_{13}) ~~= P(a \cdot P(a')),
\end{cases}
\end{align*}
for $a, a' \in A$. In other words, $(A,P)$ is an averaging algebra. If $M$ is any $A$-bimodule, we define a linear map $Q : M \rightarrow M$ by $Q(u) = \sum r_{(1)} \cdot_M u \cdot_M r_{(2)}$, for $u \in M$. Then it is easy to verify that $(M, Q)$ is a bimodule over the averaging algebra $(A,P)$.

\begin{exam}
Let $M \xrightarrow{P} A$ and $M' \xrightarrow{P'} A'$ be two relative averaging algebras, and let $(\varphi, \psi)$ be a morphism between them (see Definition \ref{defn-morp}). Then the tuple $(M' \xrightarrow{P'} A', l, r)$ is a bimodule over the relative averaging algebra $M \xrightarrow{P} A$, where the $A$-bimodule structure on $A'$ is induced by the algebra morphism $\varphi : A \rightarrow A'$, and the $A$-bimodule structure on $M'$ is given by $a \cdot_{M'} m' = \varphi (a) \cdot_{M'}^{A'} m'$ and $m' \cdot_{M'} a = m' \cdot_{M'}^{A'} \varphi (a)$, for $a \in A$, $m' \in M'$. Moreover, the pairing maps $l: M \times A' \rightarrow M'$ and $r: A' \times M \rightarrow M'$ are respectively given by
\begin{align*}
l (u, a') = \psi (u) \cdot_{M'}^{A'} a' \quad \text{ and } \quad r (a', u) = a' \cdot_{M'}^{A'} \psi (u), \text{ for } u \in M, a' \in A'.
\end{align*}
\end{exam}

\medskip

Note that any bimodule over an associative algebra can be dualized. More generally, if $A$ is an associative algebra and $M$ is an $A$-bimodule then the dual space $M^*$ can be equipped with an $A$-bimodule structure with left and right $A$-actions given by
\begin{align*}
(a \cdot_{M^*} f) (u) = f(u \cdot_M a) ~~~~ \text{ and } ~~~~ (f \cdot_{M^*} a) (u) = f(a \cdot_M u), \text{ for } a \in A, ~f \in M^*,~ u \in M.
\end{align*}
In the following result, we give the dual construction of a bimodule over a relative averaging algebra.

\begin{prop}
Let $M \xrightarrow{P} A$ be a relative averaging algebra and $(N \xrightarrow{Q} B, l, r)$ be a bimodule over it. Then $(B^* \xrightarrow{Q^*} N^*, l^*, r^*)$ is also a bimodule, where $B^*, N^*$ are equipped with dual $A$-bimodule structures and the pairings $l^* : M \times N^* \rightarrow B^*$ and $r^*: N^* \times M \rightarrow B^*$ are respectively given by
\begin{align*}
l^* (u, f_N)(b) = f_N (r (b, u)) ~~~ \text{ and } ~~~ r^* (f_N, u) (b) = f_N (l(u, b)), ~\text{ for } u \in M,~f_N \in N^*,~ b \in B.
\end{align*}
\end{prop}

\begin{proof}
For any $a \in A$, $u \in M$, $f_N \in N^*$ and $b \in B$, we first observe that
\begin{align*}
l^* (a \cdot_M u, f_N) (b) = f_N (r (b, a \cdot_M u)) \stackrel{(\ref{dasm2})}{=} f_N (r (b \cdot_B a, u)) = l^* (u, f_N) (b \cdot_B a)
= (a \cdot_{B^*} l^*(u, f_N)) (b),
\end{align*}
\begin{align*}
l^* (u \cdot_M a, f_N) (b) = f_N (r (b, u \cdot_M a)) \stackrel{(\ref{dasm2})}{=} f_N (r (b ,u) \cdot_N a) = (a \cdot_{N^*} f_N) (r (b,u)) 
= l^* (u, a \cdot_{N^*} f_N)(b),
\end{align*}
\begin{align*}
l^* (u, f_N \cdot_{N^*} a)(b) = f_N (a \cdot_N r (b,u)) \stackrel{(\ref{dasm2})}{=} f_N (r (a \cdot_B b, u)) 
= l^* (u, f_N) (a \cdot_B b)  = (l^* (u, f_N) \cdot_{B^*} a)(b).
\end{align*}
This shows that the identities in (\ref{dasm1}) hold for the dual structure. Similarly, one can verify the identities in (\ref{dasm2}) for the dual structure. Finally, for any $u \in M$, $f_B \in B^*$ and $n \in N$, we have
\begin{align*}
\big(  P(u) \cdot_{N^*} Q^*(f_B) \big) (n) =~& Q^*(f_B) \big(n \cdot_N P(u)  \big) \\
=~& f_B \big( Q (n \cdot_N P(u))  \big) \\
=~& \begin{cases}
= f_B \big(  Q(n) \cdot_B P(u) \big) = \big(  P(u) \cdot_{B^*}f_B \big) (Q(n)) = Q^* \big( P(u) \cdot_{B^*} f_B \big)(n), \\
= f_B \big(  Q \circ r (Q(n), u) \big) = l^* \big( u, Q^* (f_B)  \big) (Q(n)) = Q^* \big( l^* (u, Q^* (f_B))   \big) (n).
\end{cases}
\end{align*}
Similarly, we have
\begin{align*}
\big(   Q^* (f_B) \cdot_{N^*} P(u) \big) (n) =~& Q^* (f_B) \big(P(u) \cdot_N n \big) \\
=~& f_B \big( Q(P(u) \cdot_N n)  \big) \\
=~& \begin{cases}
= f_B \big(  Q \circ l (u, Q(n))  \big)  = r^* \big( Q^* (f_B), u  \big) (Q(n)) = Q^* \big( r^* (Q^* (f_B), u)   \big) (n),\\
= f_B \big( P(u) \cdot_B Q(n) \big) = \big( f_B \cdot_{B^*} P(u)  \big) (Q(n)) = Q^* \big( f_B \cdot_{B^*} P(u)  \big) (n).
\end{cases}
\end{align*}
This shows that the identities in (\ref{dasm3}) and (\ref{dasm4}) also hold for the dual structure. Hence $(B^* \xrightarrow{Q^*} N^*, l^*, r^*)$ is a bimodule over the relative averaging algebra $M \xrightarrow{P} A$.
\end{proof}

\medskip

Let $M \xrightarrow{P} A$ be a relative averaging algebra. Then $(A^* \xrightarrow{P^*} M^*, l^*, r^*)$ is a bimodule, where the pairings $l^* : M \times M^* \rightarrow A^*$ and $r^* : M^* \times M \rightarrow A^*$ are respectively given by
\begin{align*}
l^* (u, f_M)(a) = f_M (a \cdot_M u) ~~~ \text{ and } ~~~ r^* (f_M, u)(a) = f_M (u \cdot_M a), ~\text{ for } u \in M, ~f_M \in M^*, ~ a \in A.
\end{align*}
Note that this bimodule is dual to the adjoint bimodule given in Example \ref{adj-bimod}.

\medskip

Let $M \xrightarrow{P} A$ be a relative averaging algebra and $(N \xrightarrow{Q} B, l, r)$ be a bimodule over it. Since $B$ is an $A$-bimodule, one can consider the semidirect product algebra $A \oplus B$ with the product
\begin{align*}
(a, b) \cdot_\ltimes (a', b') = \big( a \cdot a', a \cdot_B b' + b \cdot_B a'  \big), \text{ for } (a, b), (a', b') \in A \oplus B.
\end{align*}
It has been shown in \cite{das-art} that the vector space $M \oplus N$ carries a bimodule structure over the semidirect product algebra $A \oplus B$ with left and right $(A \oplus B)$-actions are respectively given by
\begin{align}\label{apb}
(a, b) \triangleright (u, n) = (a \cdot_M u, a \cdot_N n + r(b,u)) ~~\text{ and } ~~ (u,n) \triangleleft (a, b) = (u \cdot_M a, l(u,b) + n \cdot_N a),
\end{align}
for $(a, b) \in A \oplus B$ and $(u,n) \in M \oplus N$. With these notations, we have the following result.

\begin{thm}\label{thm-semid}
(Semidirect product) Let $M \xrightarrow{P} A$ be a relative averaging algebra and $(N \xrightarrow{Q} B, l, r)$ be a bimodule over it. Then $M \oplus N \xrightarrow{P \oplus Q} A \oplus B$ is a relative averaging algebra.
\end{thm}

\begin{proof}
We have already seen that $A \oplus B$ is an associative algebra (with the semidirect product structure) and $M \oplus N$ is an $(A \oplus B)$-bimodule with left and right actions given by (\ref{apb}). Next, for any $(u,n), (u',n') \in M \oplus N$, we observe that 
\begin{align*}
(P \oplus Q) (u,n) \cdot_\ltimes (P \oplus Q) (u', n') 
&= (P(u), Q(n)) \cdot_\ltimes (P(u'), Q(n')) \\
&= \big(  P(u) \cdot P(u'), P(u) \cdot_B Q(n') + Q (n) \cdot_B P(u')  \big) \\
&= \big(  P (P(u) \cdot_M u'), Q (P(u) \cdot_N n') + Q (r (Q(n), u'))  \big) \\
&= (P \oplus Q) \big( P(u) \cdot_M u', P(u) \cdot_N n' + r (Q(n), u')   \big) \\
&= (P \oplus Q) \big(  ((P \oplus Q) (u,n)) \triangleright (u',n')  \big).
\end{align*}
Also, we have
\begin{align*}
(P \oplus Q) (u,n) \cdot_\ltimes (P \oplus Q) (u', n') 
&= \big(  P(u) \cdot P(u'), P(u) \cdot_B Q(n') + Q (n) \cdot_B P(u')  \big) \\
&= \big( P (u \cdot_M P(u')), Q (l (u, Q(n'))) + Q (n \cdot_N P(u')) \big) \\
&= (P \oplus Q) \big(  u \cdot_M P(u'), l(u, Q(n')) + n \cdot_N P(u')  \big) \\
&= (P \oplus Q) \big(  (u,n) \triangleleft ((P \oplus Q) (u', n'))  \big).
\end{align*}
This proves that $P \oplus Q : M \oplus N \rightarrow A \oplus B$ is a relative averaging operator. In other words, $M \oplus N \xrightarrow{P \oplus Q} A \oplus B$ is a relative averaging algebra.
\end{proof}

\medskip

\begin{prop}
Let $M \xrightarrow{P} A$ be a relative averaging algebra and $(N \xrightarrow{Q} B, l, r)$ be a bimodule over it. Then the vector space $N$ carries a representation of the induced diassociative algebra $M_P$ with the action maps given by
\begin{align}\label{bim-str-maps}
\begin{cases}
\dashv : M_P \otimes N \rightarrow N, \quad u \dashv n = l (u, Q(n)), \\
\vdash : M_P \otimes N \rightarrow N, \quad u \vdash n = P(u) \cdot_N n, \\
\dashv : N \otimes M_P \rightarrow N, \quad n \dashv u = n \cdot_N P(u), \\
\vdash : N \otimes M_P \rightarrow N, \quad n \vdash u = r (Q(n), u).
\end{cases}
\end{align}
\end{prop}

\begin{proof}
To prove the result, we consider the semidirect product relative averaging algebra $M \oplus N \xrightarrow{P \oplus Q} A \oplus B$ given in Theorem \ref{thm-semid}. Then it follows that the vector space $M \oplus N$ carries a diassociative algebra structure (denoted by $(M \oplus N)_{P \oplus Q}$) with the operations
\begin{align*}
(u,n) \dashv_{P \oplus Q} (u', n') = (u , n) \triangleleft \big( P(u'), Q(n') \big) = \big(  u \dashv_P u' , l (u, Q(n')) + n \cdot_N P(u') \big), \\
(u,n) \vdash_{P \oplus Q} (u', n') =  \big(  P(u), Q(n) \big) \triangleright (u', n') = \big( u \vdash_P u', P(u) \cdot_N n' + r (Q(n), u')  \big),
\end{align*}
for $(u,n), (u', n') \in M \oplus N$. This shows that the diassociative algebra $M_P$ has a representation on $N$ with the structure maps (\ref{bim-str-maps}), and the diassociative algebra $(M \oplus N)_{P \oplus Q}$ is nothing but the semidirect product of the diassociative algebra $M_P$ with the representation $N$.
\end{proof}

\begin{prop}\label{aa-b}
Let $M \xrightarrow{P} A$ be a relative averaging algebra and $(N \xrightarrow{Q} B, l, r)$ be a bimodule over it. Then the vector space $B$ can be given a representation of the induced diassociative algebra $M_P$ with the action maps given by
\begin{align}\label{nij-act-di}
\begin{cases}
 \dashv : M_P \otimes B \rightarrow B, \quad u \dashv b = P(u) \cdot_B b - Q (l (u, b)), \\
 \vdash : M_P \otimes B \rightarrow B, \quad u \vdash b = P(u) \cdot_B b, \\
 \dashv : B \otimes M_P \rightarrow B, \quad b \dashv u = b \cdot_B P(u),\\
 \vdash : B \otimes M_P \rightarrow B, \quad b \vdash u = b \cdot_B P(u) - Q (r (b, u)).
\end{cases}
\end{align}
\end{prop}

\begin{proof}
We first consider the semidirect product relative averaging algebra $M \oplus N \xrightarrow{P \oplus Q} A \oplus B$. Then it follows from Proposition \ref{avg-nij} that the map 
\begin{align*}
\mathcal{N}_{P \oplus Q} : (A \oplus B) \oplus (M \oplus N) \rightarrow (A \oplus B) \oplus (M \oplus N),~(a,b, u, n) \mapsto (P(u), Q(v), 0 ,0)
\end{align*}
is a Nijenhuis operator on the diassociative algebra $(A \oplus B) \oplus_\mathrm{Diass} (M \oplus N)$. Hence there is a new diassociative structure on the vector space $A \oplus B \oplus M \oplus N$ (deformed by the Nijenhuis operator $\mathcal{N}_{P \oplus Q}$). It can be checked that this diassociative structure on $A \oplus B \oplus M \oplus N$ restricts a diassociative algebra structure on the vector space $B \oplus M$. The structure is precisely given by
\begin{align*}
(b,u) \dashv (b', v) =  ( \underbrace{b \cdot_B P(v)}_{:=~ b \dashv v} + \underbrace{P(u) \cdot_B b' - Q (l (u, b'))}_{:= ~u \dashv b'}, u \dashv_P v), \\
(b,u) \vdash (b', v) =  ( \underbrace{b \cdot_B P(v) - Q (r(b, v))}_{:=~ b \vdash v} + \underbrace{P(u) \cdot_B b'}_{:=~u \vdash b'}, u \vdash_P v),
\end{align*}
for $(b, u), (b', v) \in B \oplus M.$ This expression shows that the diassociative algebra $M_P$ has a representation on $B$ with the action maps (\ref{nij-act-di}).
\end{proof}

\section{The controlling algebra and cohomology for relative averaging operators}\label{sec-4}
Given an associative algebra $A$ and an $A$-bimodule $M$, here we construct a graded Lie algebra whose Maurer-Cartan elements are precisely relative averaging operators. This characterization allows us to define the cohomology of a relative averaging operator. Subsequently, we show that this cohomology of a relative averaging operator can be seen as the cohomology of the induced diassociative algebra with coefficients in a suitable representation.

Let $A$ be an associative algebra and $M$ be an $A$-bimodule.
Consider the vector space $D= A \oplus M$ and the Majumdar-Mukherjee graded Lie algebra
\begin{align*}
\mathfrak{g} = \big(  \oplus_{n = 0}^\infty CY^{n+1} (D,D) = \oplus_{n = 0}^\infty \mathrm{Hom}({\bf k}[Y_{n+1}] \otimes D^{\otimes n+1}, D) , [~,~]_\mathsf{MM} \big)
\end{align*}
associated with the vector space $D$. Note that the diassociative algebra structure on $D$ (see Proposition \ref{am-diass}) gives rise to a Maurer-Cartan element $\Delta \in \mathfrak{g}_1$ in the graded Lie algebra $\mathfrak{g}$ (i.e. we have $[\Delta, \Delta]_\mathsf{MM} = 0$), where $\Delta$ is given by
\begin{align}\label{del-de}
\begin{cases}
\Delta \big( \begin{tikzpicture}[scale=0.1]
\draw (6,0) -- (8,-2);    \draw (8,-2) -- (10,0);     \draw (8,-2) -- (8,-4);     \draw (9,-1) -- ( 8,0);
\end{tikzpicture}  ; (a, u), (b, v) \big) = (a, u) \dashv (b, v) = (a \cdot b, u \cdot_M b), \\
\Delta \big(  \begin{tikzpicture}[scale=0.1]
\draw (0,0) -- (2,-2); \draw (2,-2) -- (4,0); \draw (2,-2) -- (2,-4); \draw (1,-1) -- (2,0); 
\end{tikzpicture}   ; (a, u), (b, v) \big) = (a, u) \vdash (b, v) = (a \cdot b, a \cdot_M v),
\end{cases}
\end{align}
for $(a, u), (b, v) \in D$.
Moreover, it is easy to see that the graded subspace
\begin{align*}
\mathfrak{a} = \oplus_{n = 0}^\infty CY^{n+1} (M,A) = \oplus_{n = 0}^\infty \mathrm{Hom} ({\bf k}[Y_{n+1}] \otimes M^{\otimes n+1}, A)
\end{align*}
is an abelian Lie subalgebra. Let $p: \mathfrak{g} \rightarrow \mathfrak{g}$ be the projection onto the subspace $\mathfrak{a}$. Then $\mathrm{ker}(p) \subset \mathfrak{g}$ is a graded Lie subalgebra and $\mathrm{im}(p) = \mathfrak{a}$. Moreover, $\Delta \in \mathrm{ker}(p)_1$. Hence by Kosmann-Schwarzbach's derived bracket construction \cite{yks}, 
the shifted graded vector space $s \mathfrak{a} = \oplus_{n = 1}^\infty CY^n (M,A) = \oplus_{n = 1}^\infty \mathrm{Hom} ({\bf k} [Y_n] \otimes M^{\otimes n}, A)$ carries a graded Lie algebra structure with the bracket
\begin{align*}
\llbracket f, g \rrbracket := (-1)^m ~[[ \Delta, f]_\mathsf{MM}, g ]_\mathsf{MM},
\end{align*}
for $f \in CY^m (M, A)$ and $g \in CY^n (M, A)$. In terms of $\circ_i$ operations (see (\ref{partial-comp})), the above bracket is
\begin{align*}
\llbracket f, g \rrbracket 
&= \sum_{i=1}^m (-1)^{(i-1)n} f \circ_i (\Delta \circ_1 g) - \sum_{i=1}^m (-1)^{in}  f \circ_i (\Delta \circ_2 g) \\
& \quad - (-1)^{mn} \big\{  \sum_{i=1}^n (-1)^{(i-1)m} g \circ_i (\Delta \circ_1 f) - \sum_{i=1}^n (-1)^{im}  g \circ_i (\Delta \circ_2 f)   \big\} \\
& \quad + (-1)^{mn} ( \Delta \circ_1 f) \circ_{m+1} g - ( \Delta \circ_1 g) \circ_{n+1} f.
\end{align*}
Explicitly, the bracket is given by
\begin{align}\label{tr-tr}
&\llbracket f, g \rrbracket (y; u_1, \ldots, u_{m+n}) \\
&= \sum_{i=1}^m (-1)^{(i-1)n} f \bigg( R_0^{m; i, n+1} (y); u_1, \ldots, u_{i-1}, \nonumber \\
& \qquad \qquad \qquad \qquad \Delta \big(  R_0^{2;  1, n} R_i^{m;i, n+1} (y); g \big( R_1^{2; 1, n} R_i^{m;i, n+1} (y); u_i, \ldots, u_{i+n-1} \big), u_{i+n}   \big), u_{i+n+1}, \ldots, u_{m+n} \bigg)  \nonumber \\
& - \sum_{i=1}^m (-1)^{in} f \bigg( R_0^{m; i, n+1} (y); u_1, \ldots, u_{i-1},  \nonumber  \\
&  \qquad \qquad \qquad \qquad \Delta \big(  R_0^{2; 2,n} R_i^{m;i, n+1} (y); u_i, g \big( R_2^{2; 2,n} R_i^{m;i, n+1} (y); u_{i+1}, \ldots, u_{i+n} \big)  \big), u_{i+n+1}, \ldots, u_{m+n} \bigg)  \nonumber \\
& - (-1)^{mn} \sum_{i=1}^n (-1)^{(i-1)m} g \bigg( R_0^{n; i, m+1} (y); u_1, \ldots, u_{i-1},  \nonumber \\
& \qquad \qquad \qquad \quad \Delta \big(  R_0^{2;  1, m} R_i^{n;i, m+1} (y); f \big( R_1^{2; 1, m} R_i^{n;i, m+1} (y); u_i, \ldots, u_{i+m-1} \big), u_{i+m}   \big), u_{i+m+1}, \ldots, u_{m+n} \bigg)  \nonumber \\
& + (-1)^{mn}  \sum_{i=1}^n (-1)^{im} g \bigg( R_0^{n; i, m+1} (y); u_1, \ldots, u_{i-1},  \nonumber \\
& \qquad \qquad \qquad \quad \Delta \big(  R_0^{2; 2,m} R_i^{n;i, m+1} (y); u_i, f \big( R_1^{2; 2,m} R_i^{n;i, m+1} (y); u_{i+1}, \ldots, u_{i+m} \big)  \big), u_{i+m+1}, \ldots, u_{m+n} \bigg)  \nonumber \\
& + (-1)^{mn} \Delta \bigg( R_0^{2;1,m} R_0^{m+1;m+1, n}(y); f \big( R_1^{2;1,m} R_0^{m+1; m+1,n} (y); u_1, \ldots, u_m   \big), g \big(   R_{m+1}^{m+1; m+1, n} (y); u_{m+1}, \ldots, u_{m+n} \big)   \bigg)  \nonumber \\
& - \Delta \bigg( R_0^{2;1,n} R_0^{n+1;n+1, m}(y); g \big( R_1^{2;1,n} R_0^{n+1; n+1,m} (y); u_1, \ldots, u_n   \big), f \big(   R_{n+1}^{n+1; n+1, m} (y); u_{n+1}, \ldots, u_{m+n} \big)   \bigg),  \nonumber 
\end{align}
for $y \in Y_{m+n}$ and $u_1, \ldots, u_{m+n} \in M.$ This graded Lie bracket can be extended to the graded space $CY^\bullet (M, A) = \oplus_{n = 0}^\infty CY^n (M, A) = \oplus_{n = 0}^\infty \mathrm{Hom} ({\bf k} [Y_n] \otimes M^{\otimes n}, A)$ by the following rules
\begin{align*}
\llbracket f, a \rrbracket (y; u_1, \ldots, u_m) =~& \sum_{i=1}^m f \big( y; u_1, \ldots, u_{i-1}, a \cdot_M u_i - u_i \cdot_M a, u_{i+1}, \ldots, u_m \big)    \\
&+ f (y; u_1, \ldots, u_m) \cdot a - a \cdot f (y; u_1, \ldots, u_m), \\
\llbracket a , b \rrbracket =~& a \cdot b - b \cdot a, \text{ for } f \in CY^m (M,A),~ y \in Y_m \text{ and } a, b \in A = CY^0 (M,A).
\end{align*}

With all the above notations, we have the following interesting result.

\begin{thm}\label{mc-thm-opp}
Let $A$ be an associative algebra and $M$ be an $A$-bimodule. Then the pair $\big( CY^\bullet (M, A), \llbracket ~, ~ \rrbracket \big)$ is a graded Lie algebra. Moreover, a linear map $P: M \rightarrow A$ is a relative averaging operator if and only if $P \in CY^1(M, A)$ is a Maurer-Cartan element in the graded Lie algebra $\big( CY^\bullet (M, A), \llbracket ~, ~ \rrbracket \big)$.
\end{thm}

\begin{proof}
The first part follows from the previous discussion. To prove the second part, we first observe that any linear map $P: M \rightarrow A$ can be identified with an element (denoted by the same notation) $P \in CY^1(M,A)$, where $P ( \begin{tikzpicture}[scale=0.08]
\draw (-4,-2)-- (-2,0); \draw (-4,-2) -- (-4,-4); \draw (-6,0) -- (-4,-2);
\end{tikzpicture}; u)= P(u)$, for $u \in M$ and the unique tree \begin{tikzpicture}[scale=0.08]
\draw (-4,-2)-- (-2,0); \draw (-4,-2) -- (-4,-4); \draw (-6,0) -- (-4,-2);
\end{tikzpicture} $\in Y_1$. With this identification, it follows from (\ref{tr-tr}) that
\begin{align*}
\llbracket P, P \rrbracket \big( \begin{tikzpicture}[scale=0.1]
\draw (6,0) -- (8,-2);    \draw (8,-2) -- (10,0);     \draw (8,-2) -- (8,-4);     \draw (9,-1) -- ( 8,0);
\end{tikzpicture} ; u, v \big) = 2 \bigg(   P \big( \underbrace{\Delta (    \begin{tikzpicture}[scale=0.1]
\draw (6,0) -- (8,-2);    \draw (8,-2) -- (10,0);     \draw (8,-2) -- (8,-4);     \draw (9,-1) -- ( 8,0);
\end{tikzpicture} ; P(u), v)}_{= 0}  \big) +   P \big( \underbrace{ \Delta (    \begin{tikzpicture}[scale=0.1]
\draw (6,0) -- (8,-2);    \draw (8,-2) -- (10,0);     \draw (8,-2) -- (8,-4);     \draw (9,-1) -- ( 8,0);
\end{tikzpicture} ; u, P(v)) }_{= u \cdot_M P(v)} \big)  - \underbrace{\Delta \big(  \begin{tikzpicture}[scale=0.1]
\draw (6,0) -- (8,-2);    \draw (8,-2) -- (10,0);     \draw (8,-2) -- (8,-4);     \draw (9,-1) -- ( 8,0);
\end{tikzpicture} ; P(u), P(v)  \big) }_{P(u) \cdot P(v)} \bigg),\\
\llbracket P, P \rrbracket \big( \begin{tikzpicture}[scale=0.1]
\draw (6,0) -- (8,-2);    \draw (8,-2) -- (10,0);     \draw (8,-2) -- (8,-4);     \draw (7,-1) -- ( 8,0);
\end{tikzpicture} ; u, v \big) = 2 \bigg(   P \big( \underbrace{\Delta (    \begin{tikzpicture}[scale=0.1]
\draw (6,0) -- (8,-2);    \draw (8,-2) -- (10,0);     \draw (8,-2) -- (8,-4);     \draw (7,-1) -- ( 8,0);
\end{tikzpicture} ; P(u), v)}_{= P(u) \cdot_M v}  \big) +   P \big( \underbrace{ \Delta (    \begin{tikzpicture}[scale=0.1]
\draw (6,0) -- (8,-2);    \draw (8,-2) -- (10,0);     \draw (8,-2) -- (8,-4);     \draw (7,-1) -- ( 8,0);
\end{tikzpicture} ; u, P(v)) }_{= 0} \big)  - \underbrace{\Delta \big(  \begin{tikzpicture}[scale=0.1]
\draw (6,0) -- (8,-2);    \draw (8,-2) -- (10,0);     \draw (8,-2) -- (8,-4);     \draw (7,-1) -- ( 8,0);
\end{tikzpicture} ; P(u), P(v)  \big) }_{P(u) \cdot P(v)} \bigg),
\end{align*}
for $u, v \in M$. Hence $P$ is a Maurer-Cartan element (i.e. $\llbracket P, P \rrbracket = 0$) if and only if $P$ is a relative averaging operator.
\end{proof}

Let $A$ be an associative algebra and $M$ be an $A$-bimodule. In the previous theorem, we have seen that any relative averaging operator $P: M \rightarrow A$ can be considered as a Maurer-Cartan element in the graded Lie algebra $\big( CY^\bullet (M, A), \llbracket ~, ~ \rrbracket \big)$. Hence a relative averaging operator $P$ induces a differential
\begin{align*}
d_P := \llbracket P, - \rrbracket : CY^n (M, A) \rightarrow CY^{n+1}(M,A), \text{ for } n \geq 0,
\end{align*}
which makes $\{ CY^\bullet (M, A), d_P \}$ into a cochain complex. 
The corresponding cohomology is called the {\bf cohomology} of the relative averaging operator $P$, and the $n$-th cohomology group is denoted by $H^n_P (M,A)$.
Moreover, the differential $d_P$ makes the triple $\big( CY^\bullet(M,A), d_P, \llbracket ~, ~ \rrbracket \big)$ into a differential graded Lie algebra. This differential graded Lie algebra controls the deformations of the relative averaging operator $P$ (see the theorem below). For this reason, we call the differential graded Lie algebra $\{ CY^\bullet (M, A), d_P \}$ as the {\bf controlling algebra} for the operator $P$.

\begin{thm}
Let $P: M \rightarrow A$ be a relative averaging operator. For any linear map $P' : M \rightarrow A$, the sum $P + P'$ is also a relative averaging operator if and only if $P'$ is a Maurer-Cartan element in the differential graded Lie algebra $( CY^\bullet (M, A), d_P , \llbracket ~, ~ \rrbracket )$. 
\end{thm}

\begin{proof}
Note that the sum $P+P'$ is a relative averaging operator if and only if $\llbracket P+ P', P+ P' \rrbracket = 0$, equivalently,
\begin{align*}
\llbracket P,P' \rrbracket + \llbracket P',P \rrbracket + \llbracket P',P' \rrbracket = 0.
\end{align*}
This holds if and only if $d_P (P') + \frac{1}{2} \llbracket P', P' \rrbracket = 0$, which is equivalent to the fact that $P'$ is a Maurer-Cartan element in the differential graded Lie algebra $( CY^\bullet (M, A), d_P , \llbracket ~, ~ \rrbracket )$. 
\end{proof}

In the next, we show that the cohomology of a relative averaging operator $P: M \rightarrow A$ can be seen as the cohomology of the induced diassociative algebra $M_P$ with coefficients in a suitable representation on $A$. We start with the following result which is a particular case of Proposition \ref{aa-b}.

\begin{prop}
Let $P: M \rightarrow A$ be a relative averaging operator. Then there is a representation of the induced diassociative algebra $M_P$ on the vector space $A$ with the action maps
\begin{align*}
&\dashv : M_P \times A \rightarrow A, \quad u \dashv a = P(u) \cdot a - P (u \cdot_M a), \\
&\vdash : M_P \times A \rightarrow A, \quad u \vdash a = P(u) \cdot a, \\
&\dashv : A \times M_P \rightarrow A, \quad a \dashv u = a \cdot P(u), \\
&\vdash : A \times M_P \rightarrow A, \quad a \vdash u = a \cdot P(u) - P (a \cdot_M u).
\end{align*}
\end{prop}

It follows from the above proposition that one may define the cohomology of the induced diassociative algebra $M_P$ with coefficients in the above representation on $A$. More precisely, we consider the cochain complex $\{ CY^\bullet (M_P, A), \delta_\mathrm{Diass}^P \}$, where $CY^n (M_P, A) := \mathrm{Hom} ({\bf k}[Y_n] \otimes M^{\otimes n}, A)$ for $n \geq 0$, and the coboundary map $\delta_\mathrm{Diass}^P : CY^n (M_P, A) \rightarrow CY^{n+1} (M_P, A)$ given by
\begin{align*}
\delta_\mathrm{Diass}^P (f) & (y; u_1, \ldots, u_{n+1}) = u_1 \star_0^y f(d_0 y ; u_2, \ldots, u_{n+1}) \\
& + \sum_{i=1}^n (-1)^i ~ f (d_i y ; u_1, \ldots, u_i (\star_i^y)_P u_{i+1}, \ldots, u_{n+1}) + (-1)^{n+1}~ f(d_{n+1} y; u_1, \ldots u_n) \star_{n+1}^y u_{n+1},
\end{align*}
for $f \in CY^n (M_P, A)$, $y \in Y_{n+1}$ and $u_1, \ldots, u_{n+1} \in M$. Here $(\star_i^y)_P$ represents the product $\dashv_P$ or $\vdash_P$ accordingly when $\star_i^y$ is given by $\dashv$ or $\vdash$. We denote the $n$-th cohomology group of cochain complex $\{ CY^\bullet (M_P, A), \delta_\mathrm{Diass}^P \}$ by $H^n_\mathrm{Diass}(M_P, A)$.

\begin{prop}\label{n-mult}
Let $P: M \rightarrow A$ be a relative averaging operator. Then the coboundary operators $d_P$ and $\delta_\mathrm{Diass}^P$ are related by
\begin{align*}
d_P (f) = (-1)^n ~\delta_\mathrm{Diass}^P (f) , \text{ for } f \in CY^n (M, A).
\end{align*}
\end{prop}

\begin{proof}
For any $y \in Y_{n+1}$ and $u_1, \ldots, u_{n+1} \in M$, we have
\begin{align*}
    &\big(  d_P (f) \big) (y; u_1, \ldots, u_{n+1}) \\
    &= \llbracket P, f \rrbracket (y; u_1, \ldots, u_{n+1}) \\
   & = P \big(     \Delta \big(  R_0^{2;1,n}(y); f \big(  R_1^{2;1,n} (y); u_1, \ldots, u_n   \big), u_{n+1}  \big)    \big) \\
   & \qquad - (-1)^n P \big( \Delta \big(  R_0^{2;2,n} (y) ;  u_1, f \big( R_2^{2;2,n} (y); u_2, \ldots, u_{n+1} \big)
    \big)   \big) \\
   & \qquad - (-1)^n \sum_{i=1}^n (-1)^{i-1} f \big(  R_0^{n;i, 2} (y) ; u_1, \ldots, u_{i-1}, \Delta \big(  R_{i}^{n;i, 2} (y); P(u_i), u_{i+1}  \big), u_{i+2}, \ldots, u_{n+1}  \big) \\
  & \qquad  + (-1)^n \sum_{i=1}^n (-1)^i f \big(  R_0^{n;i, 2} (y) ; u_1, \ldots, u_{i-1}, \Delta \big(  R_i^{n;i,2} (y); u_i, P(u_{i+1}) \big), u_{i+2}, \ldots, u_{n+1} \big) \\
  & \qquad  + (-1)^n \Delta \big(  R_0^{2;2,n} (y) ; P(u_1), f \big(  R_2^{2;2,n} (y), u_2, \ldots, u_{n+1} \big)   \big) \\
 & \qquad - \Delta \big(  R_0^{2;1,n} (y); f \big(  R_1^{2;1,n} (y); u_1, \ldots, u_n \big), P(u_{n+1})  \big)  \\
   & = (-1)^n \bigg\{ \Delta \big(  R_0^{2;2,n} (y) ; P(u_1), f \big(  R_2^{2;2,n} (y), u_2, \ldots, u_{n+1} \big)   \big) \\
   & \qquad \qquad -  P \big( \Delta \big(  R_0^{2;2,n} (y) ;  u_1, f \big( R_2^{2;2,n} (y); u_2, \ldots, u_{n+1} \big)
    \big)   \big)  \\
&  +  \sum_{i=1}^n (-1)^i f \big(  R_0^{n;i, 2} (y) ; u_1, \ldots, u_{i-1}, \underbrace{\Delta \big(  R_{i}^{n;i, 2} (y); P(u_i), u_{i+1} \big) +  \Delta \big(  R_i^{n;i,2} (y); u_i, P(u_{i+1}) \big)}_{ =  u_i (\star_i^y)_P u_{i+1}}    , u_{i+2}, \ldots, u_{n+1}  \big) \\
 &  \qquad \qquad  +(-1)^{n+1} \Delta \big(  R_0^{2;1,n} (y); f \big(  R_1^{2;1,n} (y); u_1, \ldots, u_n \big), P(u_{n+1})  \big) \\
 & \qquad \qquad - (-1)^{n+1} P \big(     \Delta \big(  R_0^{2;1,n}(y); f \big(  R_1^{2;1,n} (y); u_1, \ldots, u_n   \big), u_{n+1}  \big)    \big)  \bigg\} \\
 &   = (-1)^n \bigg\{  u_1 \star_0^y f (d_0(y); u_2, \ldots, u_{n+1}) + \sum_{i=1}^n (-1)^i f (R_0^{n;i,2} (y); u_1, \ldots, u_{i-1} , u_i (\star_i^y)_P u_{i+1}, \ldots, u_{n+1}) \\
  & \qquad \quad + (-1)^{n+1}  f(d_{n+1}(y); u_1, \ldots, u_n) \star_{n+1}^y u_{n+1}
\bigg\}  \quad  (\text{as } R_2^{2;2,n} (y) = d_0 (y) \text{ and } R_1^{2;1,n} (y) = d_{n+1}(y)) \\
& =(-1)^{n} (\delta_\mathrm{Diass}^P (f) ) (u_1, \ldots, u_{n+1}).
\end{align*}
This completes the proof.
\end{proof}

It follows from the above proposition that the cohomology of a relative averaging operator $P$ is isomorphic to the cohomology of the diassociative algebra $M_P$ with coefficients in the representation $A.$ That is,
\begin{align*}
H^\bullet_P (M,A) \cong H^\bullet_{\mathrm{Diass}} (M_P, A).
\end{align*}

\begin{remark}
    The cohomology of a relative averaging operator $P$ is useful to study deformations of the operator $P$ by keeping the underlying algebra and bimodule intact. See \cite{das-rota} for the similar deformation theory of relative Rota-Baxter operators.
\end{remark}

\medskip

\subsection*{Cohomological relation with diassociative algebras}

Let $P: M \rightarrow A$ be a relative averaging operator. In the following, we find the relation between the cohomology of the relative averaging operator $P$ and the cohomology of the induced diassociative algebra $M_P$ with coefficients in the adjoint representation. To do this, we define a collection $\{ \Theta_n \}_{n=0}^\infty$ of maps
\begin{align*}
\Theta_n : CY^n (M, A) \rightarrow CY^{n+1}(M_P, M_p)
\end{align*}
by
\begin{align*}
\Theta_n (f) (y; u_1, \ldots, u_{n+1}) = \begin{cases}
(-1)^{n+1} ~ u_1 \cdot_M f (y_1; u_2, \ldots, u_{n+1}) ~~~ & \text{ if } y = | \vee y_1 \text{ for some } n\text{-tree } y_1, \\
f (y_1; u_1, \ldots, u_n) \cdot_M u_{n+1} ~~~ & \text{ if } y =  y_1 \vee | \text{ for some } n\text{-tree } y_1, \\
0 ~~~ & \text{ otherwise},
\end{cases}
\end{align*}
for $u_1, \ldots, u_{n+1} \in M.$ Then we have the following.

\begin{lemma}
For $f \in CY^m (M, A)$ and $g \in CY^n (M, A)$, we have
\begin{align*}
[\Theta_m (f), \Theta_n (g)]_\mathsf{MM} = \Theta_{m+n} \big( \llbracket f,g \rrbracket \big). 
\end{align*}
In other words, the collection $\{ \Theta_n \}_{n=0}^\infty$ defines a morphism of graded Lie algebras from $(CY^\bullet (M,A), \llbracket ~, ~ \rrbracket)$ to $(CY^{\bullet +1} (M_P, M_P), [~,~]_\mathsf{MM})$.
\end{lemma}

\begin{proof}
Let $y \in Y_{m+n+1}$ be an $(m+n+1)$-tree and $u_0 , u_1, \ldots, \ldots, u_{m+n}$ be elements of $M$. If $y = | \vee y_1$, for some $(m+n)$-tree $y_1$, then
\begin{align*}
&[\Theta_m (P), \Theta_n (Q)]_\mathsf{MM} \big(  y; u_0, u_1, \ldots, u_{m+n} \big) \\
&= \bigg( \sum_{i=1}^{m+1} (-1)^{(i-1)n}~ \Theta_m (P) \circ_i \Theta_n (Q) - (-1)^{mn}  \sum_{i=1}^{n+1} (-1)^{(i-1)m}~ \Theta_n (Q) \circ_i \Theta_m (P) \bigg) \big(  y; u_0, u_1, \ldots, u_{m+n} \big) \\
&= \Theta_m (P) \big(  R_0^{m+1; 1, n+1} (y); \Theta_n (Q) \big(      R_1^{m+1;1,n+1} (y); u_0, \ldots, u_n \big), u_{n+1}, \ldots, u_{m+n}   \big) \\ 
& \Scale[0.95]{  + \sum_{i=1}^m (-1)^{in} ~ \Theta_m (P) \big(  R_0^{m+1; i+1, n+1} (y); u_0, \ldots, u_{i-1}, \Theta_n (Q) \big(  R_{i+1}^{m+1; i+1, n+1} (y); u_i, \ldots, u_{i+n}  \big), u_{i+n+1}, \ldots, u_{m+n}   \big)} \\
&- (-1)^{mn} \bigg\{   \Theta_n (Q) \big(  R_0^{n+1; 1, m+1} (y); \Theta_m (P) \big(      R_1^{n+1;1,m+1} (y); u_0, \ldots, u_m \big), u_{m+1}, \ldots, u_{m+n}   \big) \\
& \Scale[0.92]{ + \sum_{i=1}^n (-1)^{im} ~ \Theta_n (Q) \big(  R_0^{n+1; i+1, m+1} (y); u_o, \ldots, u_{i-1}, \Theta_m (P) \big(  R_{i+1}^{n+1; i+1, m+1} (y); u_i, \ldots, u_{i+m}  \big), u_{i+m+1}, \ldots, u_{m+n}   \big)  \bigg\} } \\
&= (-1)^{m+n+1} ~ u_0 \cdot_M \big( \llbracket P, Q \rrbracket (y_1; u_1, \ldots, u_{m+n}) \big) = \big( \Theta_{m+n} \llbracket P, Q \rrbracket  \big) (y; u_0, u_1, \ldots, u_{m+n}).
\end{align*}
On the other hand, if $y = y_1 \vee |$, for some $(m+n)$-tree $y_1$, and $u_1, \ldots, u_{m+n+1}$ are elements of $M$, then
\begin{align*}
&[ \Theta_m (P), \Theta_n (Q) ]_\mathsf{MM} (y; u_1, u_2, \ldots, u_{m+n+1}) \\
&= \bigg( \sum_{i=1}^{m+1} (-1)^{(i-1)n}~ \Theta_m (P) \circ_i \Theta_n (Q) - (-1)^{mn}  \sum_{i=1}^{n+1} (-1)^{(i-1)m}~ \Theta_n (Q) \circ_i \Theta_m (P) \bigg) \big(  y; u_0, u_1, \ldots, u_{m+n} \big) \\
&= \sum_{i=1}^m (-1)^{(i-1) n}~ \Theta_m (P) \big( R_0^{m+1; i, n+1} (y) ;  u_1, \ldots, u_{i-1} , \Theta_n (Q) \big( R_i^{m+1; i, n+1} (y) ; u_i, \ldots, u_{i+n} \big) , \ldots, u_{m+n+1} \big) \\
&+ (-1)^{mn} \Theta_m (P) \big(  R_0^{m+1; m+1, n+1} (y); u_1, \ldots, u_{m}, \Theta_n (Q) \big( R_{m+1}^{m+1; m+1, n+1} (y); u_{m+1}, \ldots, u_{m+n+1} \big) \big) \\
& \Scale[0.9]{ - (-1)^{mn} \bigg\{    
\sum_{i=1}^n (-1)^{(i-1) m}~ \Theta_n (Q) \big( R_0^{n+1; i, m+1} (y) ;  u_1, \ldots, u_{i-1} , \Theta_m (P) \big( R_i^{n+1; i, m+1} (y) ; u_i, \ldots, u_{i+m} \big) , \ldots, u_{m+n+1} \big) } \\
&+ (-1)^{mn} \Theta_n (Q) \big(  R_0^{n+1; n+1, m+1} (y); u_1, \ldots, u_{n}, \Theta_m (P) \big( R_{n+1}^{n+1; n+1, m+1} (y); u_{n+1}, \ldots, u_{m+n+1} \big) \big) \bigg\} \\
&= \big( \llbracket P, Q \rrbracket (y; u_1, \ldots, u_{m+n})  \big) \cdot_M u_{m+n+1} = \big( \Theta_{m+n} \llbracket P, Q \rrbracket  \big) (y; u_1, \ldots, u_{m+n+1}).
\end{align*}
Finally, for any other $y$'s in $Y_{m+n+1}$ (that are not of the form $| \vee y_1$ or $y_1 \vee |$), one can easily verify from the partial compositions (\ref{partial-comp}) that
\begin{align*}
[\Theta_m (P) , \Theta_n (Q) ] (y; u_0, u_1, \ldots, u_{m+n}) = 0 = (\Theta_{m+n} \llbracket P, Q \rrbracket) (y; u_0, u_1, \ldots, u_{m+n}).
\end{align*}
This concludes the proof.
\end{proof}

Let $\pi_P \in CY^2 (M_P, M_P)$ be the Maurer-Cartan element corresponding to the induced diassociative algebra $M_P$. In other words, $\pi_P$ is given by

\medskip

\begin{center}
$\pi_P \big($ \begin{tikzpicture}[scale=0.15]
\draw (6,0) -- (8,-2);  \draw (8,-2) -- (10,0); \draw (8,-2) -- (8,-4);   \draw (9,-1) -- ( 8,0);
\end{tikzpicture} $;~ u, v \big) = u \dashv_P v $ \quad and \quad $ \pi_P \big($ \begin{tikzpicture}[scale=0.15]
\draw (0,0) -- (2,-2);  \draw (2,-2) -- (4,0);  \draw (2,-2) -- (2,-4);    \draw (1,-1) -- (2,0);   
\end{tikzpicture} $;~ u, v \big) = u \vdash_P v,$ for $u, v \in M.$ 
\end{center}

\medskip

\medskip

\noindent Then it follows from the above lemma that the following diagram commutes
\begin{align*}
\xymatrix{
 CY^n(M, A) \ar[r]^{\Theta_n} \ar[d]_{\delta^P_\mathrm{Diass} = (-1)^n ~ \! d_P} & CY^{n+1} (M_P, M_P) \ar[d]^{\delta_\mathrm{Diass} = (-1)^n~ \! [\pi_P, ~]_\mathsf{MM}} \\
 CY^{n+1} (M, A) \ar[r]_{\Theta_{n+1}} & CY^{n+2} (M_P, M_P).
 }
\end{align*}

As a consequence, we get the following result.

\begin{thm}\label{n-func}
Let $P: M \rightarrow A$ be a relative averaging operator. Then there is a morphism 
\begin{align*}
\Theta_\bullet : {H^\bullet_P (M, A)} \rightarrow H^{\bullet +1}_\mathrm{Diass} (M_P, M_P)
\end{align*}
from the cohomology of the relative averaging operator $P$ to the cohomology of the induced diassociative algebra $M_P$ with coefficients in the adjoint representation.
\end{thm}
\medskip

\section{The controlling algebra and cohomology for relative averaging algebras}\label{sec-5}
In this section, we first construct a $L_\infty$-algebra whose Maurer-Cartan elements are precisely relative averaging algebras. Next, given a relative averaging algebra, we construct the corresponding controlling $L_\infty$-algebra. Finally, we define the cohomology of a relative averaging algebra with coefficients in a given bimodule. 

\subsection*{$L_\infty$-algebras} The notion of $L_\infty$-algebras (also known as strongly homotopy Lie algebras) first appeared in the work of Lada and Stasheff \cite{lada-stasheff}. In this paper, we follow the equivalent definition by a degree shift \cite{lada-markl}.

\begin{defn}
A {\bf $L_\infty$-algebra} is a pair $(L,\{l_k\}^\infty_{k=1} )$ consisting of a graded vector space  $L =\oplus_{i \in \mathbb{Z}} L_i$   equipped with a collection $\{ l_k : L^{\otimes k} \rightarrow L \}_{k=1}^\infty$ of degree $1$ graded linear maps that are graded symmetric in the sense that
\begin{align*}
l_k (x_{\sigma (1)}, \ldots, x_{\sigma (k)}) = \epsilon (\sigma) l_k (x_1, \ldots, x_k), \text{ for } k \geq 1 \text{ and } \sigma \in \mathbb{S}_k,
\end{align*}
and satisfy the following higher Jacobi identities: 
\begin{align*}
\sum_{i+j= n+1} \sum_{\sigma \in \mathbb{S}_{(i, n-i)}} \epsilon (\sigma) ~l_j \big( l_i ( x_{\sigma(1)}, \ldots, x_{\sigma(i)} ), x_{\sigma (i+1)}, \ldots, x_{\sigma(n)}  \big) = 0,
\end{align*}
for all $n \geq 1$ and homogeneous elements $x_1, \ldots, x_n \in {L}$. Here $\epsilon (\sigma)$ is the Koszul sign that appears in the graded context.
\end{defn}

Throughout the paper, we assume that all $L_\infty$-algebras are weakly filtered \cite{getzler} (see also \cite{laza-rota}). Thus, certain infinite sums in $L$ are always convergent.

\begin{defn}
Let $({L},\{l_k\}^\infty_{k=1} )$ be a $L_\infty$-algebra. An element $\alpha \in {L}_0$ is said to be a {\bf Maurer-Cartan element} of the $L_\infty$-algebra is $\alpha$ satisfies
\begin{align*}
l_1 (\alpha) + \frac{1}{2!} l_2 (\alpha, \alpha) + \cdots + \frac{1}{n!} ~ l_n (\alpha, \ldots, \alpha) + \cdots \cdot = 0 \quad \big(\text{i.e. } \sum_{k=1}^\infty \frac{1}{k!} l_k (\alpha, \ldots, \alpha) = 0\big).
\end{align*}
\end{defn}

\medskip

If $({L},\{l_k\}^\infty_{k=1} )$ is a $L_\infty$-algebra and $\alpha \in {L}_0$ is a Maurer-Cartan element of it, then one can construct a new $L_\infty$-algebra $({L},\{l_k^\alpha \}^\infty_{k=1} )$ on the same graded vector space ${L}$ with the structure maps given by
\begin{align*}
l_k^\alpha (x_1, \ldots, x_k) = l_k (x_1, \ldots, x_k) + l_{1+k} (\alpha, x_1, \ldots, x_k) + \cdots + \frac{1}{n!}~ l_{n+k} (\underbrace{\alpha, \ldots, \alpha}_{n \text{ copies}}, x_1, \ldots, x_k) + \cdots \cdot, \text{ for } k \geq 1.
\end{align*}
This is called the $L_\infty$-algebra obtained from $({L},\{l_k\}^\infty_{k=1} )$ twisted by the Maurer-Cartan element $\alpha$. 

\begin{remark}\label{rem-get}(\cite{getzler})
Let $\alpha$ be a Maurer-Cartan element of the $L_\infty$-algebra $({L},\{l_k\}^\infty_{k=1} )$. 
Then for any $\alpha' \in L_0$, the sum $\alpha + \alpha'$ is a Maurer-Cartan element of the $L_\infty$-algebra $({L},\{l_k\}^\infty_{k=1} )$ if and only if $\alpha'$ is a Maurer-Cartan element of the twisted $L_\infty$-algebra $({L},\{l_k^\alpha \}^\infty_{k=1} )$.
\end{remark}

\medskip

There is a well-known construction of a $L_\infty$-algebra given by Voronov \cite{voro}.
Let $(\mathfrak{g}, \mathfrak{a}, p, \Delta)$ be a quadruple consists of a graded Lie algebra $\mathfrak{g}$ (with the graded Lie bracket $[~,~]$), an abelian graded Lie subalgebra $\mathfrak{a} \subset \mathfrak{g}$, a projection map $p: \mathfrak{g} \rightarrow \mathfrak{g}$ with $\mathrm{im}(p) = \mathfrak{a}$ and $\mathrm{ker}(p) \subset \mathfrak{g}$ a graded Lie subalgebra, and an element $\Delta \in \mathrm{ker}(p)_1$ that satisfies $[\Delta, \Delta] = 0$. Such a quadruple is called a {\bf $V$-data}.

\begin{thm}\label{v-data-l}
Let $(\mathfrak{g}, \mathfrak{a}, p, \Delta)$ be a $V$-data.

(i) Then the graded vector space $\mathfrak{a}$ can be equipped with a $L_\infty$-algebra with the structure maps
\begin{align*}
l_k (a_1, \ldots, a_k) = p [ \cdots [[ \Delta, a_1], a_2], \ldots, a_k], \text{ for } k \geq 1.
\end{align*}

(ii) Let $\mathfrak{h} \subset \mathfrak{g}$ be a graded Lie subalgebra that satisfies $[\Delta, \mathfrak{h}] \subset \mathfrak{h}$. Then the graded vector space $s^{-1} \mathfrak{h} \oplus \mathfrak{a}$ can be given a $L_\infty$-algebra with the structure maps
\begin{align*}
l_1 \big( (s^{-1}x, a) \big) =~& \big(  -s^{-1}[\Delta, x], p (x + [\Delta, a])  \big),\\
l_2 \big( (s^{-1}x,0),  (s^{-1}y, 0) \big) =~& \big( (-1)^{|x|} s^{-1}[x, y], 0 \big), \\
l_k \big( (s^{-1}x,0),(0, a_1), \ldots, (0,a_{k-1}) \big) =~& \big( 0, p [ \cdots [[x, a_1], a_2], \ldots, a_{k-1}] \big), ~ k \geq 2,\\
l_k \big( (0,a_1), \ldots, (0,a_k) \big) =~& \big( 0, p [ \cdots [[ \Delta, a_1], a_2], \ldots, a_k] \big), ~ k \geq 2,
\end{align*}
for homogeneous elements $x, y \in \mathfrak{h}$ (which are considered as elements $s^{-1}x, s^{-1} y \in s^{-1} \mathfrak{h}$ by a degree shift) and homogeneous elements $a_1, \ldots, a_k \in \mathfrak{a}$. Up to permutations of the above inputs, all other maps vanish.
\end{thm}

\medskip

\subsection*{Maurer-Cartan characterization of relative averaging algebras}
Let $A$ and $M$ be two vector spaces. Consider the graded Lie algebra 
\begin{align*}
\mathfrak{g} = \big(  \oplus_{n = 0}^\infty CY^{n+1} (A \oplus M, A \oplus M), [~,~]_\mathsf{MM}  \big)
\end{align*} 
associated to the vector space $A \oplus M$. For any $k, l \geq 0$, let $\mathcal{A}^{k, l}$ be the direct sum of all possible $(k+l)$ tensor powers of $A$ and $M$ in which $A$ appears $k$ times and $M$ appears $l$ times. For instance,
\begin{align*}
\mathcal{A}^{2,0} =  A \otimes A, \quad \mathcal{A}^{0,2} = M \otimes M ~~~ \text{ and } ~~~ \mathcal{A}^{1,1} = (A \otimes M) \oplus (M \otimes A).
\end{align*}
Then for any $n \geq 1$, there is an isomorphism $(A \oplus M)^{\otimes n} \cong \oplus_{k+l=n} \mathcal{A}^{k,l}$ of vector spaces.
A linear map $f \in CY^{n+1} (A \oplus M, A \oplus M)$ is said to have {\bf bidegree} $k|l$ with $k+l = n$ if 
\begin{align*}
f ({\bf k}[Y_{n+1}] \otimes \mathcal{A}^{k+1, l}) \subset A, \quad f ({\bf k}[Y_{n+1}] \otimes \mathcal{A}^{k, l+1}) \subset M ~~~~ \text{ and } ~~~~ f = 0 \text{ otherwise}.
\end{align*}
We denote the set of all linear maps of bidegree $k|l$ by $CY^{k|l} (A \oplus M, A \oplus M)$. Note that there are natural isomorphisms
\begin{align*}
CY^{k|0} (A \oplus M, A \oplus M) \cong~&  \mathrm{Hom} ({\bf k}[Y_{k+1}] \otimes A^{\otimes k+1}, A) \oplus \mathrm{Hom} ({\bf k}[Y_{k+1}] \otimes \mathcal{A}^{k, 1}, M), \\
CY^{-1 | l} (A \oplus M, A \oplus M) \cong~& \mathrm{Hom} ({\bf k } [Y_l] \otimes M^{\otimes l}, A).
\end{align*}

Moreover, we have the following interesting result.

\begin{prop}\label{deg-pre}
For $f \in CY^{k_f | l_f} (A \oplus M, A \oplus M)$ and $g \in CY^{k_g | l_g} (A \oplus M, A \oplus M)$, we have
\begin{align*}
[f, g ]_\mathsf{MM} \in  CY^{k_f + k_g | l_f + l_g} (A \oplus M, A \oplus M).
\end{align*}
\end{prop}

\begin{proof}
Let $f \in CY^{m+1} (A \oplus M, A \oplus M)$ and $g \in CY^{n+1} (A \oplus M, A \oplus M).$ Then we have $k_f + l_f = m$ and $k_g + l_g = n$. For any $y \in Y_{m+n+1}$, $1 \leq i \leq m+1$ and $x_1 \otimes \cdots \otimes x_{m+n+1} \in \mathcal{A}^{k_f + k_g + 1, l_f + l_g }$, we have
\begin{align}\label{fg-term}
(f \circ_i g ) \big( y ; x_1, \ldots, x_{m+n+1}   \big) = f \big(  R_0^{m+1; i, n+1} (y) ; x_1, \ldots, x_{i-1}, g \big(  R_{i}^{m+1; i, n+1} (y); x_i, \ldots, x_{i+n} \big), \ldots, x_{m+n+1}  \big).
\end{align}
Note that the term  $g \big(  R_{i}^{m+1; i, n+1} (y); x_i, \ldots, x_{i+n} \big)$ is nonvanishing only when the tensor product $x_i \otimes x_{i+1} \otimes \cdots \otimes x_{i+n}$ lies in $\mathcal{A}^{k_g +1 , l_g}$ or lies in $\mathcal{A}^{k_g , l_g + 1}$.

\medskip

 {Case 1.} (Let $x_i \otimes x_{i+1} \otimes \cdots \otimes x_{i+n} \in \mathcal{A}^{k_g +1 , l_g}$.) In this case, $g \big(  R_{i}^{m+1; i, n+1} (y); x_i, \ldots, x_{i+n} \big) \in A$. Hence the tensor product
 \begin{align*}
 x_1 \otimes \cdots \otimes x_{i-1} \otimes g \big(  R_{i}^{m+1; i, n+1} (y); x_i, \ldots, x_{i+n} \big) \otimes x_{i+n+1} \otimes \cdots \otimes x_{m+n+1} ~\in \mathcal{A}^{k_f +1, l_f}.
 \end{align*}
 As a consequence, the term (\ref{fg-term}) lies in $A$.
 
 \medskip
 
  {Case 2.} (Let $x_i \otimes x_{i+1} \otimes \cdots \otimes x_{i+n} \in \mathcal{A}^{k_g , l_g + 1}$.) In this case, $g \big(  R_{i}^{m+1; i, n+1} (y); x_i, \ldots, x_{i+n} \big) \in M$. Hence the tensor product
 \begin{align*}
 x_1 \otimes \cdots \otimes x_{i-1} \otimes g \big(  R_{i}^{m+1; i, n+1} (y); x_i, \ldots, x_{i+n} \big) \otimes x_{i+n+1} \otimes \cdots \otimes x_{m+n+1} ~\in \mathcal{A}^{k_f +1, l_f}.
 \end{align*}
 As a consequence, the term (\ref{fg-term}) also lies in $A$. Therefore, we always have $(f \circ_i g) (\mathcal{A}^{k_f + k_g +1, l_f + l_g}) \subset A$. Similarly, we can show that
 \begin{align*}
 (f \circ_i g)  (\mathcal{A}^{k_f + k_g, l_f + l_g + 1}) \subset M \quad \text{ and } \quad f \circ_i g = 0 ~~~ \text{ otherwise.}
 \end{align*}
By interchanging the roles of $f$ and $g$, we get similar results for $g \circ_i f$. Therefore, it follows from (\ref{mm-circ}) that
\begin{align*}
[f, g]_\mathsf{MM} (\mathcal{A}^{k_f + k_g +1, l_f + l_g}) \subset A, \quad [f, g]_\mathsf{MM} (\mathcal{A}^{k_f + k_g, l_f + l_g +1}) \subset M ~~~~~ ~~ \text{ and } ~~~~~ ~~ [f, g]_\mathsf{MM} = 0 ~~\text{ otherwise}.
\end{align*}
Hence we get that $[f, g]_\mathsf{MM} \in CY^{k_f + k_g | l_f + l_g} (A \oplus M, A \oplus M)$.
\end{proof}

As a consequence of the previous proposition, we get the following.

\begin{prop}\label{sublie}
Let $A$ and $M$ be two vector spaces. Then

 (i) $\mathfrak{h} = CY^{\bullet | 0} (A \oplus M, A \oplus M) =  \oplus_{n = 0}^\infty CY^{n|0} (A \oplus M, A \oplus M) \subset \mathfrak{g}$ is a graded Lie subalgebra;
 
 (ii) $\mathfrak{a} = CY^{-1 | \bullet +1} (A \oplus M, A \oplus M) =  \oplus_{n = 0}^\infty CY^{-1| n +1} (A \oplus M, A \oplus M) \subset \mathfrak{g}$ is an abelian subalgebra.  
\end{prop}

Next, we construct a $V$-data as follows. Let $\mathfrak{g} = \big( \oplus_{n = 0}^\infty CY^{n+1} (A \oplus M, A \oplus M), [~, ~]_\mathsf{MM} \big)$ be the graded Lie algebra associated to the vector space $A \oplus M$. Consider the abelian Lie subalgebra $\mathfrak{a} = \oplus_{n = 0}^\infty CY^{-1| n +1} (A \oplus M, A \oplus M)$, and let $p: \mathfrak{g} \rightarrow \mathfrak{g}$ be the projection onto the subspace $\mathfrak{a}$. Then the quadruple $(\mathfrak{g}, \mathfrak{a}, p, \overline{\Delta} = 0)$ is a $V$-data. Moreover, it follows from Proposition \ref{sublie} that $\mathfrak{h} = \oplus_{n = 0}^\infty CY^{n |0} (A \oplus M , A \oplus M)$ is a graded Lie subalgebra of $\mathfrak{g}$ that obviously satisfies $[ \overline{\Delta}, \mathfrak{h} ]_\mathsf{MM} \subset \mathfrak{h}$. Hence by applying Theorem \ref{v-data-l}, we obtain the following.

\begin{thm}\label{ll-iinnff}
Let $A$ and $M$ be two vector spaces. Then there is a $L_\infty$-algebra structure on the graded vector space $s^{-1} \mathfrak{h} \oplus \mathfrak{a}$ with the structure maps $\{ l_k \}_{k = 1}^\infty$ are given by
\begin{align*}
l_2 ( (s^{-1} f,0),( s^{-1} g,0) ) =~& ((-1)^{|f|} ~s^{-1} [f,g]_\mathsf{MM}, 0),\\
l_k ((s^{-1}f,0),(0, h_1), \ldots, (0,h_{k-1})) =~& (0,p [ \cdots [[ f, h_1]_\mathsf{MM}, h_2 ]_\mathsf{MM}, \ldots, h_{k-1}]_\mathsf{MM}), ~ k \geq 2,
\end{align*}
for homogeneous elements $f, g \in \mathfrak{h}$ (considered as elements $s^{-1} f, s^{-1} g \in s^{-1} \mathfrak{h}$) and homogeneous elements $h_1, \ldots, h_{k-1} \in \mathfrak{a}$. Up to permutations of the above entries, all other maps vanish.
\end{thm}

\medskip



Let $A$ and $M$ be two vector spaces. Suppose there are maps
\begin{align*}
\mu \in \mathrm{Hom} (A^{\otimes 2}, A), ~ ~ l_M \in \mathrm{Hom} (A \otimes M, M),~ ~ r_M \in \mathrm{Hom} (M \otimes A, M) ~\text{ and }~ P \in \mathrm{Hom} (M, A).
\end{align*}
We define an element $\Delta \in \mathfrak{h}_1 = CY^{1|0} (A \oplus M, A \oplus M) = \mathrm{Hom} ({\bf k}[Y_2] \otimes A^{\otimes 2}, A) \oplus \mathrm{Hom}({\bf k}[Y_2] \otimes \mathcal{A}^{1,1}, M)$ by
\begin{align}\label{del-del}
 \Delta \big( \begin{tikzpicture}[scale=0.1]
\draw (6,0) -- (8,-2);    \draw (8,-2) -- (10,0);     \draw (8,-2) -- (8,-4);     \draw (9,-1) -- ( 8,0);
\end{tikzpicture}  ; (a, u), (b, v) \big)= ( \mu (a , b) , r_M(u , b) ) ~~ \text{ and } ~~
\Delta \big(  \begin{tikzpicture}[scale=0.1]
\draw (0,0) -- (2,-2); \draw (2,-2) -- (4,0); \draw (2,-2) -- (2,-4); \draw (1,-1) -- (2,0); 
\end{tikzpicture}   ; (a, u), (b, v) \big) =  ( \mu( a , b), l_M(a , v)),
\end{align}
for $(a, u), (b, v) \in A \oplus M$. Note that $\Delta$ can be regarded as an element $s^{-1} \Delta \in (s^{-1} \mathfrak{h})_0$.

\begin{thm}\label{mc-ravg-thm}
With the above notations, $A_\mu := (A, \mu)$ is an associative algebra, $M_{l_M, r_M} := (M, l_M, r_M)$ is an $A_\mu$-bimodule and $P: M \rightarrow A$ is a relative averaging operator (in short, $M_{l_M,r_M} \xrightarrow{P} A_\mu$ is a relative averaging algebra) if and only if $\alpha = (s^{-1} \Delta, P) \in (s^{-1} \mathfrak{h} \oplus \mathfrak{a})_0$ is a Maurer-Cartan element of the $L_\infty$-algebra $(s^{-1} \mathfrak{h} \oplus \mathfrak{a}, \{ l_k \}_{k=1}^\infty)$.
\end{thm}





\begin{proof}
First observe that $l_1 ((s^{-1} \Delta, P)) = 0$. Moreover, it follows from Proposition \ref{deg-pre} that
\begin{align*}
[\Delta, P]_\mathsf{MM} &\in CY^{0|1} (A \oplus M, A \oplus M), \quad [ [\Delta, P]_\mathsf{MM}, P]_\mathsf{MM} \in CY^{-1|2} (A \oplus M, A \oplus M) \\ & ~~~~~ \text{ and } ~~~~~ [ [ [\Delta, P]_\mathsf{MM}, P]_\mathsf{MM}, P]_\mathsf{MM} \in CY^{-2|3} (A \oplus M, A \oplus M).
\end{align*}
Since the space $CY^{-2|3} (A \oplus M, A \oplus M)$ is trivial, we have $ [ [ [\Delta, P]_\mathsf{MM}, P]_\mathsf{MM}, P]_\mathsf{MM} =0.$ As a consequence, we have $l_k \big( (s^{-1} \Delta, P), \ldots, (s^{-1} \Delta, P)   \big) = 0$ for $k \geq 4.$ Hence
\begin{align}\label{some-ee}
&\sum_{k=1}^\infty \frac{1}{k!} ~ l_k \big( (s^{-1} \Delta, P), \ldots, (s^{-1} \Delta, P)   \big) \nonumber \\
&= \frac{1}{2!} l_2 \big( (s^{-1} \Delta, P), (s^{-1} \Delta, P)   \big) ~+~ \frac{1}{3!} l_3 \big( (s^{-1} \Delta, P), (s^{-1} \Delta, P)  , (s^{-1} \Delta, P)   \big)  \nonumber \\
&= \big(   - \frac{1}{2} s^{-1} [\Delta, \Delta]_\mathsf{MM}, ~\frac{1}{2} [[\Delta, P]_\mathsf{MM}, P]_\mathsf{MM} \big).
\end{align}
Observe that
\begin{align*}
[\Delta, \Delta]_\mathsf{MM} =~& 0 ~~\text{ if and only if } A_\mu \text{ is an associative algebra and } M_{l_M, r_M} \text{ is an } A_\mu\text{-bimodule},  \\
[[\Delta, P]_\mathsf{MM}, P]_\mathsf{MM}  =~& 0 ~~ \text{ if and only if } P \text{ is a relative averaging operator}~ (\mathrm{cf.~ Theorem ~} \ref{mc-thm-opp}).
\end{align*}
Thus, it follows from (\ref{some-ee}) that $\alpha = (s^{-1} \Delta, P)$ is a Maurer-Cartan element of the $L_\infty$-algebra $(s^{-1} \mathfrak{h} \oplus \mathfrak{a}, \{ l_k \}_{k=1}^\infty)$ if and only if $M_{l_M,r_M} \xrightarrow{P} A_\mu$ is a relative averaging algebra.
\end{proof}

\medskip

Let $M_{l_M,r_M} \xrightarrow{P} A_\mu$ be a given relative averaging algebra. Here $\mu$ denotes the associative multiplication on $A$, and $l_M, r_M$ respectively denote the left and right $A$-actions on $M$. We have seen in the previous theorem that $\alpha = (s^{-1} \Delta, P) \in (s^{-1} \mathfrak{h} \oplus \mathfrak{a})_0$ is a Maurer-Cartan element of the $L_\infty$-algebra $(s^{-1} \mathfrak{h} \oplus \mathfrak{a}, \{ l_k \}_{k=1}^\infty)$, where $\Delta$ is given by (\ref{del-del}) or (\ref{del-de}). Therefore, we can consider the $L_\infty$-algebra $\big( s^{-1} \mathfrak{h} \oplus \mathfrak{a}, \{ l_k^{(s^{-1} \Delta, P)} \}_{k=1}^\infty \big)$ twisted by the Maurer-Cartan element $\alpha = (s^{-1} \Delta, P)$. Then by following Remark \ref{rem-get}, we get the next result.

\begin{thm}
Let $M_{ l_M, r_M } \xrightarrow{P} A_\mu$ be a given relative averaging algebra with the corresponding Maurer-Cartan element $\alpha = (s^{-1} \Delta, P) \in (s^{-1} \mathfrak{h} \oplus \mathfrak{a})_0$. Suppose there are maps
\begin{align*}
\mu' \in \mathrm{Hom} (A^{\otimes 2}, A), ~ ~ l'_M \in \mathrm{Hom} (A \otimes M, M),~ ~ r'_M \in \mathrm{Hom} (M \otimes A, M) ~\text{ and }~ P' \in \mathrm{Hom} (M, A).
\end{align*}
Then $M_{ l_M + l'_M , r_M+ r'_M} \xrightarrow{P+P'} A_{ \mu+ \mu'}$
is a relative averaging algebra if and only if $\alpha' = (s^{-1} \Delta', P')$ is a Maurer-Cartan element of the  $L_\infty$-algebra $\big( s^{-1} \mathfrak{h} \oplus \mathfrak{a}, \{ l_k^{(s^{-1} \Delta, P)} \}_{k=1}^\infty \big)$, where $\Delta'$ is defined in similar to (\ref{del-del}).
\end{thm}

The above theorem shows that the $L_\infty$-algebra $\big( s^{-1} \mathfrak{h} \oplus \mathfrak{a}, \{ l_k^{(s^{-1} \Delta, P)} \}_{k=1}^\infty \big)$ controlls the deformations of the relative averaging algebra $M_{ l_M, r_M } \xrightarrow{P} A_\mu$. For this reason, the $L_\infty$-algebra $\big( s^{-1} \mathfrak{h} \oplus \mathfrak{a}, \{ l_k^{(s^{-1} \Delta, P)} \}_{k=1}^\infty \big)$ is called the {\bf controlling algebra} for the given relative averaging algebra $M_{ l_M, r_M } \xrightarrow{P} A_\mu$.

\medskip

\begin{remark}\label{last-rem}
    Let $M \xrightarrow{P} A$ be a relative averaging algebra. Since the corresponding controlling algebra $\big( s^{-1} \mathfrak{h} \oplus \mathfrak{a}, \{ l_k^{(s^{-1} \Delta, P)} \}_{k=1}^\infty \big)$ is a $L_\infty$-algebra, it follows that $(l_1^{(s^{-1} \Delta, P)})^2 = 0$. We will use this fact in the construction of the cochain complex of the relative averaging algebra $M \xrightarrow{P} A.$
\end{remark}



\medskip

\subsection*{Cohomology of relative averaging algebras (with adjoint bimodule)} Here we will define the cohomology of a relative averaging algebra $M \xrightarrow{P} A$ (with coefficients in the adjoint bimodule). For each $n \geq 0$, we define an abelian group $C^n_\mathrm{rAvg} (M \xrightarrow{P} A)$ by
\begin{align*}
C^n_\mathrm{rAvg} (M \xrightarrow{P} A) = \begin{cases}
0 & \text{ if } n=0, \\ \medskip 
\mathrm{Hom}(A,A) \oplus \mathrm{Hom}(M, M) & \text{ if } n=1, \\ \medskip 
\mathrm{Hom}( A^{\otimes n}, A) \oplus \mathrm{Hom}( \mathcal{A}^{n-1, 1}, M) \oplus \mathrm{Hom}({\bf k}[Y_{n-1}] \otimes M^{\otimes n-1}, A) 
& \text{ if } n \geq 2.
\end{cases}
\end{align*}
Before we define the coboundary map $\delta_\mathrm{rAvg} : C^n_\mathrm{rAvg} (M \xrightarrow{P} A) \rightarrow C^{n+1}_\mathrm{rAvg} (M \xrightarrow{P} A)$, we observe the following. First, there is an embedding
$\mathrm{Hom} ((A \oplus M)^{\otimes n}, A \oplus M) \hookrightarrow \mathrm{Hom} ({\bf k}[Y_n] \otimes (A \oplus M)^{\otimes n}, A \oplus M),~ f \mapsto \widetilde{f},$ where $\widetilde{f}$ is given by
\begin{align*}
\widetilde{f} (y; x_1, \ldots, x_n) =  f(x_1, \ldots, x_n), \text{ for all } y \in Y_n \text{ and } x_1, \ldots, x_n \in A \oplus M.
\end{align*}
With this, the classical Gerstenhaber bracket $[~,~]_\mathsf{G}$ on the graded space $\oplus_{n=1}^\infty \mathrm{Hom} ((A \oplus M)^{\otimes n}, A \oplus M)$ embedds into the Majumdar-Mukherjee bracket $[~,~]_\mathsf{MM}.$ When we restrict the above embedding, we obtain embeddings
\begin{align*}
\mathrm{Hom}(A^{\otimes n}, A) & \hookrightarrow \mathrm{Hom}({\bf k} [Y_n] \otimes A^{\otimes n}, A), ~ f \mapsto \widetilde{f}, \\ \mathrm{Hom}( \mathcal{A}^{n-1, 1}, M)& \hookrightarrow \mathrm{Hom}( {\bf k}[Y_n] \otimes \mathcal{A}^{n-1, 1}, M), ~g \mapsto \widetilde{ g}. 
\end{align*}
Note that an element $(f,g) \in C^1_\mathrm{rAvg} (M \xrightarrow{P} A) = \mathrm{Hom} (A,A) \oplus \mathrm{Hom}(A,A)$ can be identified with the element $(s^{-1} (\widetilde{f} + \widetilde{g}), 0) \in (s^{-1} \mathfrak{h} \oplus \mathfrak{a})_{-1}$. Here we assume that $\mathfrak{a}_{-1} = 0$. Similarly, an element $(f,g, \gamma) \in C^{n \geq 2}_\mathrm{rAvg} (M \xrightarrow{P} A)$ can be identified with the element $(s^{-1} (\widetilde{f} + \widetilde{g}), \gamma) \in (s^{-1} \mathfrak{h} \oplus \mathfrak{a})_{n-2}$.
%
Using the above identifications, we now define a map $\delta_\mathrm{rAvg} : C^n_\mathrm{rAvg} (M \xrightarrow{P} A) \rightarrow C^{n+1}_\mathrm{rAvg} (M \xrightarrow{P} A)$ by
\begin{align*}
    \delta_\mathrm{rAvg} ((f,g)) =~& - l_1^{(s^{-1} \Delta, P)} (    s^{-1} (\widetilde{f}+\widetilde{g}), 0), \text{ for } (f,g) \in C^1_\mathrm{rAvg} (M \xrightarrow{P} A), \\
     \delta_\mathrm{rAvg} ((f,g, \gamma)) =~& (-1)^{n-2} l_1^{(s^{-1} \Delta, P)} (    s^{-1} (\widetilde{f}+\widetilde{g}), \gamma), \text{ for } (f,g,h) \in C^{n \geq 2}_\mathrm{rAvg} (M \xrightarrow{P} A).
\end{align*}
It follows from Remark \ref{last-rem} that $(\delta_\mathrm{rAvg})^2 = 0$. In other words,  $\{ C^\bullet_\mathrm{rAvg} (M \xrightarrow{P} A), \delta_\mathrm{rAvg} \}$ is a cochain complex. The corresponding cohomology is called the {\bf cohomology} of the relative averaging algebra $M \xrightarrow{P} A$. We denote the corresponding $n$-th cohomology group by $H^n_\mathrm{rAvg} (M \xrightarrow{P} A)$.

Note that
\begin{align}\label{ident-diff}
    &\delta_\mathrm{rAvg} ((f, g, \gamma)) \nonumber \\
    &= (-1)^{n-2} l_1^{(s^{-1} \Delta, P)} (s^{-1} (\widetilde{f} + \widetilde{g}), \gamma )  \nonumber \\
    &= (-1)^{n-2} \sum_{k=0}^\infty \frac{1}{k!} l_{k+1} \big( \underbrace{  (s^{-1} \Delta, P), \ldots, (s^{-1} \Delta, P)}_{k \text{ times}},  (s^{-1} (\widetilde{f} + \widetilde{g}), \gamma )    \big)  \nonumber \\
    &= (-1)^{n-2} \bigg\{  l_2 \big(   (s^{-1} \Delta, 0) , (s^{-1} (\widetilde{f} + \widetilde{g}), 0) \big)   + l_3 \big(  (s^{-1} \Delta, 0), (0,P), (0, \gamma)  \big)  \nonumber  \\
    & \qquad \qquad \quad + \frac{1}{n !} l_{n+1} (  (s^{-1} (\widetilde{f} + \widetilde{g}), 0 ), \underbrace{(0, P), \ldots, (0, P)}_{n \text{ times}}   )   \bigg\} \quad (\text{as the other terms get vanished})  \nonumber \\
    &= (-1)^{n-2} \bigg(  -s^{-1} [\Delta, \widetilde{f} + \widetilde{g}]_\mathsf{MM} ~,~ [[\Delta, P]_\mathsf{MM}, \gamma ]_\mathsf{MM} + \frac{1}{n!} \underbrace{[ \cdots [[}_{n \text{ times}}  \widetilde{f} + \widetilde{g}, P  ]_\mathsf{MM}, P]_\mathsf{MM}, \ldots, P]_\mathsf{MM}  \bigg)  \nonumber \\
    &= \bigg(  (-1)^{n-1} s^{-1} [\Delta, \widetilde{f} + \widetilde{g}]_\mathsf{MM} ~,~ \underbrace{(-1)^{n}[[\Delta, P]_\mathsf{MM}, \gamma ]_\mathsf{MM}}_{= \delta_\mathrm{Diass}^P (\gamma)} + \underbrace{\frac{(-1)^n}{n!} [ \cdots [[  \widetilde{f} + \widetilde{g}, P  ]_\mathsf{MM}, P]_\mathsf{MM}, \ldots, P]_\mathsf{MM}}_{= h_P (f,g) \text{~ (say)}}  \bigg).
\end{align}
Using the above identifications, the term (\ref{ident-diff}) (which lies is $(s^{-1} \mathfrak{h} \oplus \mathfrak{a})_{n-1}$) can be identified with the element
\begin{align*}
    \big(  (-1)^{n-1} [\mu, f ]_\mathsf{G}~,~ (-1)^{n-1} [\mu + l_M + r_M , f + g ]_\mathsf{G}~, ~\delta_\mathrm{Diass}^P (\gamma) + h_P (f, g)   \big) \in C^{n+1}_\mathrm{rAvg} (M \xrightarrow{P} A).
\end{align*}
Here the first component $(-1)^{n-1} [\mu , f]_\mathsf{G}$ is nothing but $\delta_\mathrm{Hoch}(f)$, where $\delta_\mathrm{Hoch}$ is the Hochschild coboundary operator of the associative algebra $A$ with coefficients in the adjoint $A$-bimodule. We denote the second component $(-1)^{n-1} [\mu + l_M + r_M , f + g ]_\mathsf{G} \in \mathrm{Hom}(\mathcal{A}^{n,1}, M)$ by the notation $\delta_\mathrm{Hoch}^f (g)$ and it is given by
\begin{align*}
    \big( \delta_\mathrm{Hoch}^f (g) \big)&(a_1, \ldots, a_{n+1}) = a_1 \cdot_M (f+g)(a_2 , \ldots , a_{n+1}) \\
   & + \sum_{i=1}^n (-1)^i g (a_1, \ldots, a_{i-1}, (\mu + \cdot_M) (a_i, a_{i+1}), \ldots, a_{n+1})  + (-1)^{n+1} (f+g)(a_1, \ldots, a_{n}) \cdot_M a_{n+1},
\end{align*}
for $a_1 \otimes \cdots \otimes a_{n+1} \in \mathcal{A}^{n,1}$ (i.e. all $a_i$'s are from $A$ except one, which is from $M$).
Finally, to better understand the term $h_P (f,g)$, we first realize an element of $\mathrm{Hom} ({\bf k}[Y_l] \otimes (A \oplus M)^{\otimes l}, A \oplus M)$ as a degree $(l-1)$ coderivation on the free dendriform coalgebra $\oplus_{n=1}^\infty {\bf k}[Y_n] \otimes (s^{-1} A \oplus s^{-1} M)^{\otimes n}$. See \cite{val-manin} for details. With this identification, the Majumdar-Mukherjee bracket can be seen as the commutator bracket of coderivations on the dendriform coalgebra $\oplus_{n=1}^\infty {\bf k}[Y_n] \otimes (s^{-1} A \oplus s^{-1} M)^{\otimes n}$. Hence, for any $y \in Y_n$ (say $y = y_1 \vee y_2$ for some unique $(i-1)$-tree $y_1 \in Y_{i-1}$ and $(n-i)$-tree $y_2 \in Y_{n-i}$) and $u_1, \ldots, u_n \in M$,
\begin{align*}
   & (h_P (f,g)) (y; u_1, \ldots, u_n) \\
   &= \frac{(-1)^n}{n!} \sum_{j=0}^n (-1)^j \binom{n}{j} \big(  \underbrace{  P \circ \cdots \circ P}_{j \text{ times}} \circ (\widetilde{f} + \widetilde{g}) \circ  \underbrace{  P \circ \cdots \circ P}_{(n-j) \text{ times}} \big) (y; u_1, \ldots, u_n) \\
   & = \frac{(-1)^n}{n!} \bigg\{ n! f \big(  P(u_1), \ldots, P(u_n)  \big)  - n (n-1)! P g \big( P(u_1), \ldots, u_i, \ldots, P(u_n)   \big)  \bigg\} \\
   & =  (-1)^n \bigg\{ f \big(  P(u_1), \ldots, P(u_n)  \big)  - P g \big( P(u_1), \ldots, P(u_{i-1}), u_i,  P(u_{i+1}), \ldots, P(u_n)   \big) \bigg\}.
\end{align*}
Hence the coboundary map $\delta_\mathrm{rAvg}$ is given by $\delta_\mathrm{rAvg} ((f,g, \gamma)) = \big( \delta_\mathrm{Hoch} (f), \delta_\mathrm{Hoch}^f (g) , \delta_\mathrm{Diass}^P (\gamma) + h_P (f,g)  
 \big)$, for $(f, g, \gamma) \in C^n_\mathrm{rAvg} (M \xrightarrow{P} A)$.

\medskip

Let $M \xrightarrow{P} A$ be a relative averaging algebra. In the following, we construct a long exact sequence that connects the cohomology of the operator $P$ and the cohomology of the full relative averaging algebra $M \xrightarrow{P} A.$ We first consider a new cochain complex $\{ C^\bullet_\mathrm{AssBimod} ( {}^AM^A,  {}^AM^A) , \delta_\mathrm{AssBimod} \}$, where
\begin{align*}
    C^0_\mathrm{AssBimod} ( {}^AM^A,  {}^AM^A) = 0 ~~~ \text{ and } ~~~ C^{n \geq 1}_\mathrm{AssBimod} ( {}^AM^A,  {}^AM^A) = \mathrm{Hom}(A^{\otimes n}, A) \oplus \mathrm{Hom}(\mathcal{A}^{n-1,1}, M).
\end{align*}
The coboundary map $\delta_\mathrm{AssBimod}$ is given by 
\begin{align*}
\delta_\mathrm{AssBimod} ((f,g)) = (\delta_\mathrm{Hoch} (f), \delta_\mathrm{Hoch}^f (g)), \text{ for } (f,g) \in C^{n \geq 1}_\mathrm{AssBimod} ( {}^AM^A,  {}^AM^A). 
\end{align*}
We denote the $n$-th cohomology of this complex by $H^{n}_\mathrm{AssBimod} ( {}^AM^A,  {}^AM^A)$. Since this cohomology captures precisely the information of the associative algebra $A$ and the $A$-bimodule $M$, we call this cohomology the cohomology of the associative bimodule ${}^A M^{A}$ (i.e. associative algebra $A$ together with the $A$-bimodule $M$).

\begin{thm}\label{long-thm}
    Let $M \xrightarrow{P} A$ be a relative averaging algebra. Then there is a long exact sequence
    \begin{align}\label{long-long}
        \cdots \rightarrow H^{n-1}_P (M, A) \rightarrow H^n_\mathrm{rAvg} (M \xrightarrow{P} A) \rightarrow    H^n_\mathrm{AssBimod} ( {}^AM^A,  {}^AM^A) \rightarrow  H^{n}_P (M, A) \rightarrow \cdots 
    \end{align}
\end{thm}

\begin{proof}
    Note that there is a short exact sequence of cochain complexes
    \begin{align*}
        0 \rightarrow \{ CY^{\bullet -1} (M_P, A), \delta^P_\mathrm{Diass} \} \xrightarrow{} \{ C^\bullet_\mathrm{rAvg} (M \xrightarrow{P} A), \delta_\mathrm{rAvg} \} \xrightarrow{} \{ C^\bullet_\mathrm{AssBimod} ( {}^AM^A, {}^AM^A), \delta_\mathrm{AssBimod} \} \rightarrow 0
    \end{align*}
    with obvious maps between complexes.
    This short exact sequence induces the long exact sequence (\ref{long-long}) on the cohomology groups.
\end{proof}

\medskip

\medskip

\noindent {\bf Cohomology of an averaging algebra (with coefficients in the adjoint bimodule).} Let $A \xrightarrow{P} A$ be an averaging algebra. For each $n \geq 0$, we define the space $C^n_\mathrm{Avg} (A \xrightarrow{P} A)$ of $n$-cochains by 
\begin{align*}
C^n_\mathrm{Avg} (A \xrightarrow{P} A) = \begin{cases}
0 & \text{ if } n = 0,\\
\mathrm{Hom} (A, A) & \text{ if } n=1,\\
\mathrm{Hom} (A^{\otimes n}, A) \oplus \mathrm{Hom} ({\bf k}[Y_{n-1}] \otimes A^{\otimes n-1}, A) & \text{ if } n \geq 2.
\end{cases}
\end{align*}
Then there is an embedding ${i} : C^n_\mathrm{Avg} (A \xrightarrow{P} A) \hookrightarrow C^n_\mathrm{rAvg} (A \xrightarrow{P} A)$ given by
\begin{align*}
i (f) =~& (f,f), \text{ for } f \in C^1_\mathrm{Avg} (A \xrightarrow{P} A),\\
{i} (f, \gamma) =~& (f, f, \gamma), \text{ for } (f, \gamma) \in C^{n \geq 2}_\mathrm{Avg} (A \xrightarrow{P} A).
\end{align*}
Let $(f, \gamma) \in C^{n}_\mathrm{Avg} (A \xrightarrow{P} A)$. Here we assume that $\gamma = 0$ when $n =1$. Then
\begin{align*}
    \delta_\mathrm{rAvg} \big(  i(f, \gamma) \big) = \delta_\mathrm{rAvg} \big( (f,f,\gamma) \big) = \big(  \delta_\mathrm{Hoch} (f) , \underbrace{\delta_\mathrm{Hoch}^f (f)}_{= \delta_\mathrm{Hoch} (f)} , \delta_\mathrm{Diass}^P (\gamma) + h_P (f,f)  \big) \in \mathrm{im} (i).
\end{align*}
This shows that the map $\delta_\mathrm{rAvg} : C^n_\mathrm{rAvg} (A \xrightarrow{P} A) \rightarrow C^{n+1}_\mathrm{rAvg} (A \xrightarrow{P} A)$ restricts to a map 
\begin{align*}
\delta_\mathrm{Avg} : C^n_\mathrm{Avg} (A \xrightarrow{P} A) \rightarrow C^{n+1}_\mathrm{Avg} (A \xrightarrow{P} A)
\end{align*}
that satisfies $\delta_\mathrm{rAvg} \circ i = i \circ \delta_\mathrm{Avg}$. Explicity, the map $\delta_\mathrm{Avg}$ is given by
\begin{align*}
    \delta_\mathrm{Avg} ((f, \gamma)) = (\delta_\mathrm{Hoch} (f) , \delta_\mathrm{Diass}^P (\gamma) + h_P (f,f) ), \text{ for } (f, \gamma) \in C^n_\mathrm{Avg} (A \xrightarrow{P} A).
\end{align*}
It follows from the condition $(\delta_\mathrm{rAvg})^2 = 0$ that the map $\delta_\mathrm{Avg}$ is also a differential (i.e. $(\delta_\mathrm{Avg})^2 = 0$). Hence $\{ C^\bullet_\mathrm{Avg} (A \xrightarrow{P} A), \delta_\mathrm{Avg} \}$ is a cochain complex. The corresponding cohomology is called the {\bf cohomology} of the averaging algebra $A \xrightarrow{P} A$. We denote the $n$-th cohomology group by $H^n_\mathrm{Avg} (A \xrightarrow{P} A)$.

\medskip





The next result shows that the cohomology of an averaging algebra fits into a long exact sequence. This is a particular case of Theorem \ref{long-thm}.

\begin{thm}
    Let $A \xrightarrow{P} A$ be an averaging algebra. Then there is a long exact sequence
    \begin{align*}
        \ldots \rightarrow H^{n-1}_P (A, A) \rightarrow H^n_\mathrm{Avg} (A \xrightarrow{P} A) \rightarrow H^n_\mathrm{Hoch} (A, A) \rightarrow H^n_P (A,A) \rightarrow \cdots.
    \end{align*}
    Here $H^n_P (A,A) $ is the $n$-th cohomology group of the averaging operator $P$ and $H^n_\mathrm{Hoch}(A,A)$ is the $n$-th Hochschild cohomology group of the associative algebra $A$.
\end{thm}

\medskip

\subsection*{Cohomology of relative averaging algebras (with arbitrary bimodule)} Here we will introduce the cohomology of a relative averaging algebra with coefficients in a bimodule. We will use this cohomology in Section \ref{sec-7} to study abelian extensions.


Let $M \xrightarrow{P} A$ be a relative averaging algebra and $(N \xrightarrow{Q} B, l, r)$ be a bimodule over it. For each $n \geq 0$, we define the space of $n$-cochains $C^n_\mathrm{rAvg} (M \xrightarrow{P} A; N \xrightarrow{Q} B )$ by
\begin{align*}
C^n_\mathrm{rAvg} (M \xrightarrow{P} A; N \xrightarrow{Q} B) = \begin{cases}
0 & \text{if } n=0, \\ \medskip 
\mathrm{Hom}(A,B) \oplus \mathrm{Hom}(M, N) & \text{if } n=1, \\ \medskip 
\mathrm{Hom}(A^{\otimes n}, B) \oplus \mathrm{Hom}(\mathcal{A}^{n-1, 1}, N) \oplus \mathrm{Hom}({\bf k}[Y_{n-1}] \otimes M^{\otimes n-1}, B) & \text{if } n \geq 2.
\end{cases}
\end{align*}
To define the coboundary map, we first consider the cochain complex $\{  C^\bullet_\mathrm{rAvg} (M \oplus N \xrightarrow{P \oplus Q} A \oplus B), \delta_\mathrm{rAvg} \}$ of the semidirect product relative averaging algebra $M \oplus N \xrightarrow{P \oplus Q} A \oplus B$ (given in Theorem \ref{thm-semid}) with coefficients in the adjoint bimodule. Then for each $n \geq 0$, there is an obvious inclusion
\begin{align*}
C^n_\mathrm{rAvg} (M \xrightarrow{P} A; N \xrightarrow{Q} B) \hookrightarrow C^n_\mathrm{rAvg} (M \oplus N \xrightarrow{P \oplus Q} A \oplus B).
\end{align*}
Moreover, the map $\delta_\mathrm{rAvg} : C^n_\mathrm{rAvg} (M \oplus N \xrightarrow{P \oplus Q} A \oplus B) \rightarrow C^{n+1}_\mathrm{rAvg} (M \oplus N \xrightarrow{P \oplus Q} A \oplus B)$ restricts to a map (denoted by the same notation) $\delta_{\mathrm{rAvg}} : C^n_\mathrm{rAvg} (M \xrightarrow{P} A; N \xrightarrow{Q} B) \rightarrow C^{n+1}_\mathrm{rAvg} (M \xrightarrow{P} A; N \xrightarrow{Q} B)$. Hence $\{ C^\bullet_\mathrm{rAvg} (M \xrightarrow{P} A; N \xrightarrow{Q} B), \delta_\mathrm{rAvg} \}$ becomes a cochain complex. Note that the restricted map $\delta_\mathrm{rAvg}$ is explicitly given by
\begin{align*}
\delta_\mathrm{rAvg} ((f, g, \gamma )) = \big(  \delta_\mathrm{Hoch} (f) , \delta_\mathrm{Hoch}^f (g) , \delta_\mathrm{Diass}^P (\gamma) + h_{P, Q} (f, g)  \big),
\end{align*}
for $(f,g, \gamma) \in C^n_\mathrm{rAvg} (M \xrightarrow{P} A; N \xrightarrow{Q} B)$. Here $\delta_\mathrm{Hoch}$ is the Hochschild coboundary operator of the associative algebra $A$ with coefficients in the $A$-bimodule $B$, and for any $f \in \mathrm{Hom}(A^{\otimes n}, B)$, the map $\delta_\mathrm{Hoch}^f : \mathrm{Hom}(\mathcal{A}^{n-1, 1}, N) \rightarrow \mathrm{Hom}(\mathcal{A}^{n, 1}, N)$ is given by
\begin{align*}
\big( \delta_\mathrm{Hoch}^f (g) \big) (a_1, \ldots, a_{n+1}) =~& (l + \cdot_N) (a_1, (f+g) (a_2, \ldots, a_{n+1})) \\~&+ \sum_{i=1}^n (-1)^i g \big(  a_1, \ldots, a_{i-1} , (\mu + l_M + r_M) (a_i, a_{i+1}), \ldots, a_{n+1}  \big) \\~&+ (-1)^{n+1} (r + \cdot_N) ( (f + g) (a_1, \ldots, a_n), a_{n+1}),
\end{align*}
for $g \in \mathrm{Hom}(\mathcal{A}^{n,1}, N)$ and $a_1 \otimes \cdots \otimes a_{n+1} \in \mathcal{A}^{n,1}$.
The map $\delta_\mathrm{Diass}^P$ is the coboundary operator of the induced diassociative algebra $M_P$ with coefficients in the representation $B$ (given in Proposition \ref{aa-b}). Finally, the map $h_{P, Q} (f,g)$ is given by
\begin{align*}
    (h_{P, Q} (f,g)) (y; u_1, \ldots, u_n ) = (-1)^n \big(  f (P(u_1), \ldots, P(u_n)) - Qg (  P(u_1), \ldots, u_i, \ldots, P(u_n) ) \big), 
\end{align*}
for $y \in Y_n$ (which can be uniquely  written as $y = y_1 \vee y_2$ for some $(i-1)$-tree $y_1$ and $(n-i)$-tree $y_2$) and $u_1, \ldots, u_n \in M$.



The cohomology of the complex $\{ C^\bullet_\mathrm{rAvg} (M \xrightarrow{P} A; N \xrightarrow{Q} B), \delta_\mathrm{rAvg} \}$ is called the {\bf cohomology} of the relative averaging algebra $M \xrightarrow{P} A$ with coefficients in the bimodule $(N \xrightarrow{Q} B, l, r).$ We denote the $n$-th cohomology group by $H^n_\mathrm{rAvg} (M \xrightarrow{P} A; N \xrightarrow{Q} B)$.

\medskip



\begin{remark}
    In Example \ref{avg-alg-bimod} we have seen that a bimodule over an averaging algebra can be seen as a bimodule over the corresponding relative averaging algebra. With this view, one can define the cohomology of an averaging algebra with coefficients in a bimodule over it.
\end{remark}

\medskip




\section{Deformations of relative averaging algebras}\label{sec-6}

In this section, we study formal and infinitesimal deformations of a relative averaging algebra in terms of the cohomology theory. In particular, we show that the set of all equivalence classes of infinitesimal deformations of a relative averaging algebra  $M \xrightarrow{P} A$ has a bijection with the second cohomology group $H^2_\mathrm{rAvg}(M \xrightarrow{P} A).$

Let $\mathsf{R}$ be a commutative unital ring with unity $1_\mathsf{R}$. An augmentation of $\mathsf{R}$ is a homomorphism $\varepsilon : \mathsf{R} \rightarrow {\bf k}$ satisfying $\varepsilon (1_\mathsf{R}) = 1_{\bf k}.$ Throughout this section, we assume that $\mathsf{R}$ is a commutative unital ring with an augmentation $\varepsilon$. Given such $\mathsf{R}$, one may always define the notion of $\mathsf{R}$-relative averaging algebra similar to Definition \ref{first-defn}(ii) by replacing the vector spaces and linear maps by $\mathsf{R}$-modules and $\mathsf{R}$-linear maps. In other words, a $\mathsf{R}$-relative averaging algebra is a relative averaging algebra in the category of $\mathsf{R}$-modules. Morphisms between $\mathsf{R}$-relative averaging algebras can be defined similarly. Note that any relative averaging algebra $M \xrightarrow{P} A$ can be regarded as a $\mathsf{R}$-relative averaging algebra, where the $\mathsf{R}$-module structures on $A$ and $M$ are respectively given by $r \cdot a = \varepsilon (r) a$ and $r \cdot u = \varepsilon (r) u$, for $r \in \mathsf{R}$, $a \in A$, $u \in M$.

\begin{defn}
A {\bf $\mathsf{R}$-deformation} of a relative averaging algebra $M \xrightarrow{P} A$ consists of a quadruple $(\mu_\mathsf{R}, l_\mathsf{R}, r_\mathsf{R}, P_\mathsf{R})$ of $\mathsf{R}$-bilinear  maps
\begin{center}
$\mu_\mathsf{R} : (\mathsf{R} \otimes_{\bf k} A) \times (\mathsf{R} \otimes_{\bf k} A) \rightarrow \mathsf{R} \times_{\bf k} A , \qquad l_\mathsf{R}: (\mathsf{R} \otimes_{\bf k} A) \times (\mathsf{R} \otimes_{\bf k} M) \rightarrow \mathsf{R} \otimes_{\bf k} M,$

\medskip

$r_\mathsf{R}: (\mathsf{R} \otimes_{\bf k} M) \times (\mathsf{R} \otimes_{\bf k} A) \rightarrow \mathsf{R} \otimes_{\bf k} M$ ~~ ~ and a $\mathsf{R}$-linear map $P_\mathsf{R} : \mathsf{R} \otimes_{\bf k} M \rightarrow \mathsf{R} \otimes_{\bf k} A$
\end{center}
such that the following conditions hold:

(i) $(\mathsf{R} \otimes_{\bf k} A, \mu_\mathsf{R})$ is an $\mathsf{R}$-associative algebra, $(\mathsf{R} \otimes_{\bf k} M, l_\mathsf{R}, r_\mathsf{R})$ a bimodule over it and the $\mathsf{R}$-linear map $P_R : \mathsf{R} \otimes_{\bf k} M \rightarrow \mathsf{R} \otimes_{\bf k} A$ is a relative averaging operator. In other words, $\mathsf{R} \otimes_{\bf k} M \xrightarrow{P_\mathsf{R}} \mathsf{R} \otimes_{\bf k} A$ is a $\mathsf{R}$-relative averaging algebra by considering the above structures on $\mathsf{R} \otimes_{\bf k} A$ and $\mathsf{R} \otimes_{\bf k} M$.

(ii) The pair $(\varepsilon \otimes_{\bf k} \mathrm{id}_A, \varepsilon \otimes_{\bf k} \mathrm{id}_M) :(\mathsf{R} \otimes_{\bf k} M \xrightarrow{P_\mathsf{R}} \mathsf{R} \otimes_{\bf k} A) \rightsquigarrow (M \xrightarrow{P} A)$ is a morphism of $\mathsf{R}$-relative averaging algebras. 
\end{defn}

\begin{defn}
Let $M \xrightarrow{P} A$ be a relative averaging algebra. Two $\mathsf{R}$-deformations $(\mu_\mathsf{R}, l_\mathsf{R}, r_\mathsf{R}, P_\mathsf{R})$ and $(\mu'_\mathsf{R}, l'_\mathsf{R}, r'_\mathsf{R}, P'_\mathsf{R})$ are said to be {\bf equivalent} if there exists an isomorphism of $\mathsf{R}$-relative averaging algebras
\begin{align*}
(\Phi, \Psi) : (\mathsf{R} \otimes_{\bf k} M \xrightarrow{P_\mathsf{R}} \mathsf{R} \otimes_{\bf k} A) \rightsquigarrow (\mathsf{R} \otimes_{\bf k} M \xrightarrow{P'_\mathsf{R}} \mathsf{R} \otimes_{\bf k} A)
\end{align*}
satisfying $(\varepsilon \otimes_{\bf k} \mathrm{id}_A) \circ \Phi = (\varepsilon \otimes_{\bf k} \mathrm{id}_A)$ and $(\varepsilon \otimes_{\bf k} \mathrm{id}_M) \circ \Psi = (\varepsilon \otimes_{\bf k} \mathrm{id}_M)$.
\end{defn}

We will now consider the cases, when $\mathsf{R} = {\bf k}[[t]]$ (the ring of formal power series) and $\mathsf{R} = {\bf k}[[t]] /(t^2)$ (the local Artinian ring of dual numbers). In the first case, a $\mathsf{R}$-deformation is called a formal deformation, and in the second case, a $\mathsf{R}$-deformation is called an infinitesimal deformation. A more precise description of formal deformation is given by the following.

\begin{defn}\label{formal-definition}
(i) Let $M_{l_M , r_M} \xrightarrow{P} A_\mu$ be a given relative averaging algebra. A {\bf formal deformation} of it consists of a quadruple $(\mu_t, l_t, r_t, P_t)$ of formal sums
\begin{align}\label{formal-component}
\mu_t = \sum_{i=0}^\infty t^i \mu_i , \qquad l_t = \sum_{i=0}^\infty t^i l_i , \qquad r_t = \sum_{i=0}^\infty t^i r_i   ~~~~ \text{ and } ~~~~ P_t = \sum_{i=0}^\infty t^i P_i
\end{align}
(where $\mu_i : A \times A \rightarrow A$, $l_i : A \times M \rightarrow M$, $r_i : M \times A \rightarrow M$ and $P_i : M \rightarrow A$ are bilinear/linear maps, for $i \geq 0$, with $\mu_0 = \mu$, $l_0 = l_M$, $r_0 = r_M$ and $P_0 = P$) such that $A[[t]] = (A[[t]], \mu_t)$ is an associative algebra over ${\bf k} [[t]]$, and $M[[t]] = (M[[t]], l_t, r_t)$ is a bimodule over the algebra $A[[t]]$, and the ${\bf k}[[t]]$-linear map $P_t : M [[t]] \rightarrow A[[t]]$ is a relative averaging operator. In other words, $M [[t]] \xrightarrow{P_t} A[[t]]$ is a relative averaging algebra over ${\bf k}[[t]]$.

(ii) Two formal deformations $(\mu_t, l_t, r_t, P_t)$ and $(\mu_t', l_t', r_t', P_t')$ are {\bf equivalent} if there exists a pair $(\varphi_t, \psi_t)$ of formal sums
\begin{align*}
\varphi_t = \sum_{i=0}^\infty t^i \varphi_i \quad \text{ and } \quad \psi_t = \sum_{i=0}^\infty t^i \psi_i
\end{align*}   
(where $\varphi_i : A \rightarrow A$ and $\psi_i : M \rightarrow M$ are linear maps, for $i \geq 0$, with $\varphi_0 = \mathrm{id}_A$ and $\psi_0 = \mathrm{id}_M$) such that
\begin{align*}
(\varphi_t, \psi_t) : (M[[t]] \xrightarrow{P_t} A[[t]]) \rightsquigarrow (M[[t]] \xrightarrow{P_t'} A[[t]])
\end{align*}
is an isomorphism of relative averaging algebras over ${\bf k} [[t]]$. Then we write $(\mu_t, l_t, r_t, P_t) \sim (\mu_t', l_t', r_t', P_t').$
\end{defn}

It follows from the above definition that a quadruple $(\mu_t, l_t, r_t, P_t)$ given by (\ref{formal-component}) is a formal deformation of the relative averaging algebra $M \xrightarrow{P} A$ if the following system of equations are hold:
\begin{align}
\sum_{i+j = n} \mu_i (\mu_j (a, b), c) =~& \sum_{i+j = n} \mu_i (a, \mu_j (b, c)), \label{defeqn-1}\\
\sum_{i+j = n} l_i (\mu_j (a, b), u) =~& \sum_{i+j = n} l_i (a, l_j (b, u)), \label{defeqn-2}\\
\sum_{i+j = n} r_i (l_j (a, u), b) =~& \sum_{i+j = n} l_i (a, r_j (u,b)), \label{defeqn-3}\\
\sum_{i+j = n} r_i (r_j (u, a), b) =~& \sum_{i+j = n} r_i (u, \mu_j (a,b)), \label{defeqn-4}
\end{align}
\begin{align}
\sum_{i+j+k = n} \mu_i \big(  P_j (u), P_k (v)  \big) = \sum_{i+j+k = n} P_i \big( l_j (P_k (u), v)  \big) =~& \sum_{i+j+k = n} P_i \big(  r_j (u, P_k (v))  \big), \label{defeqn-5}
\end{align}
for all $a, b, c \in A$, $u,v \in M$ and $n \geq 0$. These are called the deformation equations. Note that the deformation equations are held for $n=0$ as $M_{l_M, r_M} \xrightarrow{P} A_\mu$ is a relative averaging algebra.

\medskip

\noindent \underline{For $n=1$.} It follows from (\ref{defeqn-1}) that
\begin{align*}
\mu_1 (a \cdot b, c) + \mu_1 (a, b) \cdot c = \mu_1 (a, b \cdot c) + a \cdot \mu_1 (b, c), \text{ for } a, b, c \in A,
\end{align*}
which is equivalent to $\delta_\mathrm{Hoch}(\mu_1) = 0$. To summarize the identities (\ref{defeqn-2}), (\ref{defeqn-3}), (\ref{defeqn-4}) for $n=1$, we define an element $\beta_1 \in \mathrm{Hom}(\mathcal{A}^{1,1}, M)$ by
\begin{align}\label{beta-1}
 \beta_1 (a, u) = l_1 (a, u) ~~ \text{ and } ~~ \beta_1 (u, a) = r_1 (u, a), \text{ for } a \in A, u \in M.
\end{align}
Then we get that $\delta_\mathrm{Hoch}^{\mu_1} (\beta_1) = 0$. Finally, the identity (\ref{defeqn-5}) for $n=1$ is equivalent to
\begin{align*}
\big( \delta_\mathrm{Diass}^P (P_1) + h_P ( \mu_1 ,  \beta_1 ) \big) (y; u, v) = 0, \text{ for } y = \begin{tikzpicture}[scale=0.1]
\draw (6,0) -- (8,-2);    \draw (8,-2) -- (10,0);     \draw (8,-2) -- (8,-4);     \draw (9,-1) -- ( 8,0);
\end{tikzpicture} ~,~ 
\begin{tikzpicture}[scale=0.1]
\draw (0,0) -- (2,-2); \draw (2,-2) -- (4,0); \draw (2,-2) -- (2,-4); \draw (1,-1) -- (2,0); 
\end{tikzpicture} \text{ and } u, v \in M.
\end{align*}
Combining these, we get that $\delta_\mathrm{rAvg} (\mu_1, \beta_1, P_1) = 0.$ In other words, $(\mu_1, \beta_1, P_1) \in Z^2_\mathrm{rAvg} (M \xrightarrow{P} A)$ is a $2$-cocycle in the cochain complex of the relative averaging algebra $M \xrightarrow{P} A$ with coefficients in the adjoint bimodule.

\medskip

Next, let $(\mu_t, l_t, r_t, P_t)$ and $(\mu'_t, l'_t, r'_t, P'_t)$ be two equivalent formal deformations of a relative averaging algebra $M_{l_M, r_M} \xrightarrow{P} A_\mu$. Then it follows from Definition \ref{formal-definition}(ii) that there exists a pair $(\varphi_t, \psi_t)$ of formal sums such that the followings are hold:
\begin{align*}
\sum_{i+j=n} \varphi_i (\mu_j (a, b)) =~& \sum_{i+j+k =n} \mu_i' \big(  \varphi_j (a) , \varphi_k (b) \big), \\
\sum_{i+j=n} \psi_i (l_j (a, u)) =~& \sum_{i+j+k =n} l_i' \big(  \varphi_j (a) , \psi_k (u) \big),\\
\sum_{i+j=n} \psi_i (r_j (u,a)) =~& \sum_{i+j+k =n} r_i' \big(  \psi_j (u) , \varphi_k (a) \big),\\
\sum_{i+j = n} \varphi_i \circ P_j =~& \sum_{i+j=n} P_i' \circ \psi_j,
\end{align*}
for $a, b \in A$ and $u \in M$. As before, all these relations are hold for $n =0$ as $\varphi_0 = \mathrm{id}_A$ and $\psi_0 = \mathrm{id}_M$. However, for $n=1$, we get that
\begin{align*}
\quad \mu_1 (a, b) - \mu_1' (a, b) =~& a \cdot \varphi_1 (b) + \varphi_1(a) \cdot b  - \varphi_1 (a \cdot b), \\
\beta_1 (a, u) - \beta_1' (a, u) =~& a \cdot_M \psi_1 (u) + \varphi_1 (a) \cdot_M u - \psi_1 (a \cdot_M u), \\
\beta_1 (u, a) - \beta_1' (u, a) =~&  u \cdot_M \varphi_1 (a) + \psi_1(u) \cdot_M a - \psi_1 (u \cdot_M a),\\
\quad P_1 - P_1' =~& P \circ \psi_1 - \varphi_1 \circ P.
\end{align*}
Combining these relations, we get that
\begin{align*}
(\mu_1, \beta_1, P_1) - (\mu_1', \beta_1', P_1) = \big(  \delta_\mathrm{Hoch} (\varphi_1), \delta_\mathrm{Hoch}^{\varphi_1} (\psi_1), h_P (\varphi_1, \psi_1)  \big) = \delta_\mathrm{rAvg} ((\varphi_1, \psi_1)).
\end{align*}
Thus, equivalent formal deformations give rise to cohomologous $2$-cocycles.

\medskip

As mentioned earlier that an infinitesimal deformation of a relative averaging algebra is a $\mathsf{R}$-deformation, where $\mathsf{R} = {\bf k}[[t]]/(t^2)$. Thus, an infinitesimal deformation can be regarded as a truncated version (modulo $t^2$) of formal deformation. Equivalences between infinitesimal deformations can be defined similarly.

\begin{thm}
Let $M \xrightarrow{P} A$ be a relative averaging algebra. Then there is a bijection
\begin{align*}
(\text{infinitesimal deformations of } M \xrightarrow{P} A)/ \sim \quad \leftrightsquigarrow \quad H^2_\mathrm{rAvg} (M \xrightarrow{P} A).
\end{align*}
\end{thm}

\begin{proof}
Let $(\mu_t = \mu + t \mu_1, l_t = l + t l_1, r_t = r+ t r_1, P_t = P + t P_1)$ be an infinitesimal deformation of the relative averaging algebra $M \xrightarrow{P} A$. Then one can show that (similar to formal deformation) the triple $(\mu_1, \beta_1, P_1) \in Z^2_\mathrm{rAvg} (M \xrightarrow{P} A)$ is a $2$-cocycle, where $\beta_1 \in \mathrm{Hom} (\mathcal{A}^{1,1} , M)$ is given by (\ref{beta-1}). Moreover, equivalent infinitesimal deformations give rise to cohomologous $2$-cocycles. Hence, there is a well-defined map
\begin{align*}
\Gamma : (\text{infinitesimal deformations of } M \xrightarrow{P} A)/ \sim ~~ \rightarrow ~~ H^2_\mathrm{rAvg} (M \xrightarrow{P} A).
\end{align*}

To obtain a map in the other direction, we first consider a $2$-cocycle $(\mu_1, \beta_1, P_1) \in Z^2_\mathrm{rAvg} (M \xrightarrow{P} A)$. Then it is easy to see that the $2$-cocycle $(\mu_1, \beta_1, P_1)$ induces an infinitesimal deformation 
\begin{align*}
(\mu_t = \mu + t \mu_1, l_t = l + t l_1, r_t = r+ t r_1, P_t = P + t P_1)
\end{align*}
of the relative averaging algebra $M \xrightarrow{P} A$, where the maps $l_1, r_1$ are defined from $\beta_1$ by (\ref{beta-1}). Let $(\mu_1', \beta_1', P_1')$ be another $2$-cocycle cohomologous to $(\mu_1, \beta_1, P_1)$, i.e. $(\mu_1, \beta_1, P_1) - (\mu_1' , \beta_1', P_1') = \delta_\mathrm{rAvg} ((\varphi_1, \psi_1))$, for some $(\varphi_1, \psi_1) \in C^1_\mathrm{rAvg} (M \xrightarrow{P} A)$. Then it is easy to verify that the corresponding infinitesimal deformations $(\mu_t, l_t, r_t, P_t)$ and $(\mu'_t, l'_t, r'_t, P'_t)$ are equivalent via the pair $(\varphi_t = \mathrm{id}_A + t \varphi_1, \psi_t = \mathrm{id}_M + t \psi_1)$. As a consequence, we obtain a map
\begin{align*}
\Theta : H^2_\mathrm{rAvg} (M \xrightarrow{P} A) ~~ \rightarrow ~~ (\text{infinitesimal deformations of } M \xrightarrow{P} A)/ \sim.
\end{align*}
Finally, it is routine task to check that the maps $\Gamma$ and $\Theta$ are inverses to each other. This completes the proof.
\end{proof}

\medskip

\section{Abelian extensions of relative averaging algebras}\label{sec-7}
Our aim in this section is to study abelian extensions of a relative averaging algebra $M \xrightarrow{P} A$ by a bimodule $(N \xrightarrow{Q} B, l, r)$ of it. We show that the isomorphism classes of such abelian extensions are in bijective correspondence with the second cohomology group $H^2_\mathrm{rAvg} (M \xrightarrow{P} A; N \xrightarrow{Q} B)$.

Let $M \xrightarrow{P} A$ be a relative averaging algebra and $N \xrightarrow{Q} B$ be a $2$-term chain complex (not necessarily a bimodule). Note that $N \xrightarrow{Q} B$ can be regarded as a relative averaging algebra with the trivial associative multiplication on $B$ and the trivial $B$-bimodule structure on $N$. With this consideration, we have the following definition.

\begin{defn}
An {\bf abelian extension} of a relative averaging algebra $M \xrightarrow{P} A$ by a $2$-term chain complex $N \xrightarrow{Q} B$ is a relative averaging algebra $\widehat{M} \xrightarrow{\widehat{P}} \widehat{A}$ with a short exact sequence of relative averaging algebras of the form
\begin{align}\label{abl-fig}
\xymatrix{
0 \ar[r] & N \ar[r]^{\overline{i}} \ar[d]_Q & \widehat{M} \ar[r]^{\overline{p}} \ar[d]^{\widehat{P}} &  M \ar[r] \ar[d]^P & 0 \\
0 \ar[r] & B \ar[r]_i & \widehat{A}  \ar[r]_p & A \ar[r] & 0.
}
\end{align}
Sometimes, we denote an abelian extension as above by the relative averaging algebra $\widehat{M} \xrightarrow{\widehat{P}} \widehat{A}$ when the exact sequence is understood.
\end{defn}

A section of the abelian extension (\ref{abl-fig}) is given by a pair $(s, \overline{s})$ of linear maps $s : A \rightarrow \widehat{A}$ and $\overline{s} : M \rightarrow \widehat{M}$ satisfying $p \circ s = \mathrm{id}_A$ and $\overline{p} \circ  \overline{s} = \mathrm{id}_M$. Given any section $(s, \overline{s})$, we define two bilinear maps (both denoted by the same notation) $\cdot_B : A \times B \rightarrow B$ and $\cdot_B : B \times A \rightarrow B$ by
\begin{align*}
a \cdot_B b = s(a) \cdot_{\widehat{A}} i(b) ~~~~ \text{ and } ~~~~ b \cdot_B a = i(b) \cdot_{\widehat{A}} s(a), \text{ for } a \in A, b \in B.
\end{align*}
These two maps make $B$ into an $A$-bimodule. Similarly, there are two bilinear maps (both denoted by the same notation) $\cdot_N : A \times N \rightarrow N$ and $\cdot_N : N \times A \rightarrow N$ given by
\begin{align*}
a \cdot_N n = s(a) \cdot_{\widehat{M}} \overline{i}(n) ~~~~ \text{ and } ~~~~ n \cdot_N a = \overline{i}(n) \cdot_{\widehat{M}} s(a), \text{ for } a \in A, n \in N.
\end{align*}
Here $\cdot_{\widehat{M}}$ denotes both the left and right $\widehat{A}$-actions on $\widehat{M}$. These two maps make $N$ into an $A$-bimodule. Finally, we define bilinear maps $l: M \times B \rightarrow N$ and $r: B \times M \rightarrow N$ by
\begin{align*}
l(u, b) = \overline{s}(u) \cdot_{\widehat{M}} i(b) ~~~~ \text{ and } ~~~~ r (b, u) = i( b) \cdot_{\widehat{M}} \overline{s}(u), \text{ for } u \in M, b \in B.
\end{align*}
It is straightforward to see that the maps $l, r$ satisfy the identities (\ref{dasm1}) and (\ref{dasm2}). Finally, for any $u \in M$ and $n \in N$,
\begin{align*}
    P(u) \cdot_B Q(n) = sP(u) \cdot_{\widehat{A}} i Q (n) = \widehat{P} \overline{s} (u) \cdot_{\widehat{A}} \widehat{P} \overline{i}(n)=~& \begin{cases}  =  \widehat{P} \big( \widehat{P}(\overline{s}(u)) \cdot_{\widehat{M}} \overline{i}(n)   \big)\\
    = \widehat{P} \big(  \overline{s}(u) \cdot_{\widehat{M}} \widehat{P} \overline{i}(n) \big)
    \end{cases} \\
  =~&  \begin{cases}  =  \widehat{P} \big( {s}P(u) \cdot_{\widehat{M}} \overline{i}(n)   \big) ~=~  Q \big( P(u) \cdot_N n   \big), \\
    = \widehat{P} \big(  \overline{s}(u) \cdot_{\widehat{M}} {i}Q(n) \big) ~=~ Q \big(  l (u, Q(n)) \big).
    \end{cases}
\end{align*}
Similarly, one can show that
\begin{align*}
    Q(n) \cdot_B P(u) =  i Q (n) \cdot_{\widehat{A}} s P (u) = \widehat{P} \overline{i}(n) \cdot_{\widehat{A}} \widehat{P} \overline{s} (u) =~& \begin{cases}  = \widehat{P} \big(  \widehat{P} \overline{i}(n ) \cdot_{\widehat{M}} \overline{s}(u) \big)\\
    = \widehat{P} \big( \overline{i}(n) \cdot_{\widehat{M}} \widehat{P} \overline{s}(u) \big)
    \end{cases} \\
    =~& \begin{cases}
   = \widehat{P} \big( i Q (n) \cdot_{\widehat{M}} \overline{s}(u)   \big)   ~=~ Q \big( r (Q(n), u)   \big), \\
 =  \widehat{P} \big(  \overline{i}(n) \cdot_{\widehat{M}} s P(u)  \big) ~=~ Q (n \cdot_N P(u)).
    \end{cases}
\end{align*}
Combining all these, we get that $(N \xrightarrow{Q} B, l, r)$ is a bimodule over the relative averaging algebra $M \xrightarrow{P} A$. This is called the induced bimodule structure starting from the abelian extension $(\ref{abl-fig})$. Note that this bimodule structure is independent of the choice of section. To see this, let $(s', \overline{s}')$ be any other section of $(\ref{abl-fig})$. Then we observe that $s(a) - s'(a) \in \mathrm{ker} (p) = \mathrm{im} (i)$ and $\overline{s} (u) - \overline{s}'(u) \in \mathrm{ker} (\overline{p}) = \mathrm{im} (\overline{i})$, for $a \in A$ and $u \in M$. Let $\cdot_B'$, $\cdot_N'$ and $l', r'$ be the maps induced by the section $(s', \overline{s}')$. Then we have
\begin{align*}
 &   a \cdot_B b - a \cdot_B' b = \big(  s(a) - s'(a) \big) \cdot_{\widehat{A}} i(b ) = 0 ~~ \text{ and } ~~ b \cdot_B a - b \cdot_B' a = i(b)  \cdot_{\widehat{A}} \big(  s(a) - s'(a) \big) = 0,\\
 &   a \cdot_N n - a \cdot_N' n = \big(  s(a) - s'(a) \big) \cdot_{\widehat{M}} \overline{i}(n) = 0 ~~ \text{ and } ~~ n \cdot_N a - n \cdot_N' a = \overline{i}(n) \cdot_{\widehat{M}} \big(  s(a) - s'(a) \big) = 0,\\
  &  l (u, b) - l'( u,b) = \big( \overline{s}(u) - \overline{s}' (u)  \big) \cdot_{\widehat{M}} i(b) = 0 ~~ \text{ and } ~~ r(b, u) - r' (b,u) = i(b) \cdot_{\widehat{M}} \big( \overline{s}(u) - \overline{s}' (u)  \big) = 0.
\end{align*}
Hence our claim follows.

\begin{defn}
Let $M \xrightarrow{P} A$ be a relative averaging algebra and $N \xrightarrow{Q} B$ be a $2$-term chain complex. Two abelian extensions $\widehat{M} \xrightarrow{ \widehat{P} } \widehat{A}$ and $\widehat{M}' \xrightarrow{ \widehat{P}' } \widehat{A}'$ are said to be {\bf isomorphic} if there is an isomorphism $(\varphi, \psi) :  (\widehat{M} \xrightarrow{ \widehat{P} } \widehat{A}) \rightsquigarrow  (\widehat{M}' \xrightarrow{ \widehat{P}' } \widehat{A}') $ of relative averaging algebras that makes the following diagram commutative
\begin{align}
\xymatrixrowsep{0.36cm}
\xymatrixcolsep{0.36cm}
\xymatrix{
0 \ar[rr] &  & N \ar[rr] \ar[dd] \ar@{=}[rd] & & \widehat{M} \ar[rr] \ar[rd]^\psi \ar[dd] & & M \ar[dd] \ar[rr] \ar@{=}[rd] & & 0 \\
 & 0 \ar[rr] & & N \ar[rr] \ar[dd] & & \widehat{M}' \ar[rr] \ar[dd] & & M \ar[rr] \ar[dd] & & 0 \\
0 \ar[rr] &  & B \ar[rr] \ar@{=}[rd] & & \widehat{A} \ar[rr] \ar[rd]^\varphi & & A \ar[rr] \ar@{=}[rd] & & 0 \\
 & 0 \ar[rr] & & B \ar[rr] & & \widehat{A}' \ar[rr] & & A \ar[rr] & & 0. \\
}
\end{align}
\end{defn}

Let $\widehat{M} \xrightarrow{ \widehat{p} } \widehat{A}$ and  $\widehat{M}' \xrightarrow{ \widehat{p}' } \widehat{A}'$ be two isomorphic abelian extensions as in the above definition. Then it is easy to see that the corresponding induced bimodules on the $2$-term chain complex $N \xrightarrow{Q} B$ are the same.

\begin{notation}
Let $M \xrightarrow{P} A$ be a relative averaging algebra and $(N \xrightarrow{Q} B. l, r)$ be a given bimodule over it. We denote by $\mathrm{Ext} (M \xrightarrow{ {P} } {A} ; N \xrightarrow{Q} B ) $ the set of all isomorphism classes of abelian extensions of $M \xrightarrow{P} A$ by the $2$-term complex $N \xrightarrow{Q} B$ so that the induced bimodule coincides with the prescribed one.
\end{notation}

In the following result, we parametrize the space $\mathrm{Ext} (M \xrightarrow{ {P} } {A} ; N \xrightarrow{Q} B ) $ by the second cohomology group of the relative averaging algebra.

\begin{thm}
Let $M \xrightarrow{P} A$ be a relative averaging algebra and $(N \xrightarrow{Q} B, l, r)$ be a given bimodule over it. Then there is a bijective correspondence between $\mathrm{Ext} (M \xrightarrow{ {P} } {A} ; N \xrightarrow{Q} B ) $ and the second cohomology group $H^2_\mathrm{rAvg} ( M \xrightarrow{ {P} } {A} ; N \xrightarrow{Q} B).$
\end{thm}

\begin{proof}
Let $\widehat{M} \xrightarrow{\widehat{P}} \widehat{A}$ be an abelian extension of the relative averaging algebra $M \xrightarrow{P} A$ by the $2$-term complex $N \xrightarrow{Q} B$ representing an element in $\mathrm{Ext}(M \xrightarrow{P} A ; N \xrightarrow{Q} B)$. Let $(s, \overline{s})$ be a section. Then we define a triple $(\alpha, \beta, \gamma)$ of maps
\begin{align*}
&\alpha \in \mathrm{Hom}(A^{\otimes 2}, B), \quad \alpha (a, b) = s(a) \cdot_{\widehat{A}} s(b) - s(a \cdot b),\\
&\beta \in \mathrm{Hom}(\mathcal{A}^{1,1}, N), \quad \begin{cases}
\beta (a, u) = s(a) \cdot_{\widehat{M}} \overline{s}(u) - \overline{s} (a \cdot_M u),\\
\beta (u, a) = \overline{s}(u) \cdot_{\widehat{M}} {s}(a) - \overline{s} (u \cdot_M a),
\end{cases}\\
&\gamma \in \mathrm{Hom}(M, B), \quad \gamma (u) = (\widehat{P} \circ \widehat{s} - s \circ P)(u),
\end{align*}
for $a, b \in A$ and $u \in M$. Then it is easy to verify that $(\alpha, \beta, \gamma) \in Z^2_\mathrm{rAvg}(M \xrightarrow{P} A; N \xrightarrow{Q} B)$ is a $2$-cocycle in the cohomology complex of the relative averaging algebra $M \xrightarrow{P} A$ with coefficients in the bimodule $(N \xrightarrow{Q} B, l, r)$. Moreover, the corresponding cohomology class in $H^2_\mathrm{rAvg}(M \xrightarrow{P} A; N \xrightarrow{Q} B)$ doesn't depend on the choice of the section.

Let $\widehat{M} \xrightarrow{\widehat{P}} \widehat{ A}$ and $\widehat{M}' \xrightarrow{\widehat{P}'} \widehat{ A}'$ be two isomorphic abelian extensions. For any section $(s, \overline{s})$ of the first abelian extension, we have
\begin{align*}
p' \circ (\varphi \circ s) = p \circ s = \mathrm{id}_A \quad \text{ and } \quad \overline{p}' \circ (\psi \circ \overline{s}) = \overline{p} \circ \overline{s} = \mathrm{id}_M.
\end{align*}
Thus $(\varphi \circ s , \psi \circ \overline{s})$ is a section of the second abelian extension. If $(\alpha', \beta', \gamma') \in Z^2_\mathrm{rAvg}(M \xrightarrow{P} A; N \xrightarrow{Q} B)$ is the $2$-cocycle corresponding to the second abelian extension and its section $(\varphi \circ s, \psi \circ \overline{s})$, then
\begin{align*}
\alpha' (a, b) =~& (\varphi \circ s)(a) \cdot_{\widehat{A}'} (\varphi \circ s)(b) - (\varphi \circ s)(a \cdot b) \\
=~& \varphi \big(  s(a) \cdot_{\widehat{A}} s(b) - s(a \cdot b)  \big) = \varphi (\alpha (a, b)) = \alpha (a, b) \quad (\because \varphi|_B = \mathrm{id}_B).
\end{align*}
Similarly, one can show that $\beta' = \beta$ and $\gamma' = \gamma$. Thus, we obtain $(\alpha, \beta, \gamma) = (\alpha', \beta', \gamma')$. As a consequence, we obtain a well-defined map
\begin{align*}
\Lambda : \mathrm{Ext}(M \xrightarrow{P} A; N \xrightarrow{Q} B) ~ \rightarrow ~ H^2_\mathrm{rAvg} (M \xrightarrow{P} A; N \xrightarrow{Q} B).
\end{align*}

To obtain a map in the other direction, we take a $2$-cocycle $(\alpha, \beta, \gamma) \in Z^2_\mathrm{rAvg} (M \xrightarrow{P} A ; N \xrightarrow{Q} B)$. Take $\widehat{A} = A \oplus B$ and $\widehat{M} = M \oplus N$, and consider bilinear maps
\begin{align*}
\mu_{\widehat{A}} : \widehat{A} \times \widehat{A} \rightarrow \widehat{A}, \quad & \mu_{\widehat{A}} \big( (a, b), (a', b')   \big) = \big( a \cdot a', a \cdot_B b' + b \cdot_B a' + \alpha (a, a')   \big), \\
l_{\widehat{M}} : \widehat{A} \times \widehat{M} \rightarrow \widehat{M}, \quad & l_{\widehat{M}} \big( (a, b), (u,n)  \big) = \big(  a \cdot_M u, a \cdot_N n + r (b, u) + \beta(a, u) \big), \\
r_{\widehat{M}} : \widehat{M} \times \widehat{A} \rightarrow \widehat{M}, \quad & r_{\widehat{M}} \big( (u,n), (a,b)  \big) = \big( u \cdot_M a, l(u, b) + n \cdot_N a + \beta (u, a)  \big),
\end{align*}
for $(a, b), (a', b') \in \widehat{A}$ and $(u,n) \in \widehat{M}$. Then it is easy to see that $(\widehat{A}, \mu_{\widehat{A}})$ is an associative algebra and $(\widehat{M}, l_{\widehat{M}}, r_{\widehat{M}})$ is a bimodule over it. Finally, we define a map $\widehat{P} : \widehat{M} \rightarrow \widehat{A}$ by
\begin{align*}
\widehat{P}((u,n)) = \big( P(u), Q(n) + \gamma (u)   \big), \text{ for } (u,n) \in \widehat{M}.
\end{align*}
Then $\widehat{P}$ is a relative averaging operator. In other words, $\widehat{M} \xrightarrow{\widehat{P}} \widehat{A}$ is a relative averaging algebra. This is an abelian extension of the relative averaging algebra $M \xrightarrow{P} A$ by the $2$-term chain complex $N \xrightarrow{Q} B$, and defines an element in $\mathrm{Ext} (M \xrightarrow{P} A; N \xrightarrow{Q} B)$. Finally, let $(\alpha, \beta, \gamma)$ and $(\alpha', \beta', \gamma')$ be two cohomologous $2$-cocycles, say $(\alpha, \beta, \gamma) - (\alpha', \beta', \gamma') = \delta_\mathrm{rAvg} ((\kappa, \eta))$, for some $(\kappa, \eta) \in C^1_{\mathrm{rAvg}} (  M \xrightarrow{P} A; N \xrightarrow{Q} B )$. If $\widehat{M}' \xrightarrow{\widehat{P}'} \widehat{A}'$ is the relative averaging algebra induced by the $2$-cocycle $(\alpha', \beta', \gamma')$, then the pair of maps
\begin{align*}
(\varphi, \psi) : (\widehat{M} \xrightarrow{\widehat{P}} \widehat{A}) \rightsquigarrow (\widehat{M}' \xrightarrow{\widehat{P}'} \widehat{A}')
\end{align*}
is an isomorphism of abelian extensions, where $\varphi : \widehat{A} \rightarrow \widehat{A}'$, $\varphi (a, b) = (a, b + \kappa (a))$ and $\psi : \widehat{M} \rightarrow \widehat{M}'$, $\psi (u,n) = (u, n + \eta (u))$. This shows that there is a well-defined map
\begin{align*} 
\Upsilon : H^2_\mathrm{rAvg} (  M \xrightarrow{P} A; N \xrightarrow{Q} B ) ~ \rightarrow ~ \mathrm{Ext}( M \xrightarrow{P} A; N \xrightarrow{Q} B).
\end{align*}
Finally, the maps $\Lambda$ and $\Upsilon$ are inverses to each other. This shows the required bijection.
\end{proof}

\medskip

\section{Homotopy relative averaging algebras and homotopy diassociative algebras}\label{sec-8}

In this section, we first consider $Diass_\infty$-algebras introduced by Loday. However, our definition is more simple to use. Next, we 
introduce the homotopy relative averaging operators and homotopy relative averaging algebras. We show that a homotopy relative averaging algebra naturally induces a $Diass_\infty$-algebra structure. We first recall some basic definitions related to $A_\infty$-algebras \cite{keller}.

\begin{defn}
An {\bf $A_\infty$-algebra} is a pair $({A}, \{ \mu_k \}_{k=1}^\infty)$ consisting of a graded vector space ${A} = \oplus_{i \in \mathbb{Z}} {A}_i$ together with a collection $\{ \mu_k \}_{k=1}^\infty$ of degree $1$ graded linear maps $\mu_k : {A}^{\otimes k} \rightarrow {A}$, for $k \geq 1$, satisfying the following identities (called higher associativities)
\begin{align}\label{a-inf-iden}
    \sum_{k+l=n+1} \sum_{i=1}^{n-l+1} (-1)^{|a_1| + \cdots + |a_{i-1}|}~ \mu_k  \big(a_1, \ldots, a_{i-1}, \mu_l (a_i, \ldots, a_{i+l-1}), a_{i+l}, \ldots, a_{n} ) = 0,
\end{align}
for all $n \geq 1$ and homogeneous elements $a_1, \ldots, a_n \in {A}.$
\end{defn}

Let ${A} = \oplus_{i \in \mathbb{Z}} {A}_i$ be a graded vector space. Let $\overline{T}(A) = \oplus_{n = 1}^\infty A^{\otimes n}$ be the free tensor algebra over the graded vector space $A$. For each $n \in \mathbb{Z}$, let $C^n (A, A) := \mathrm{Hom}_n ( \overline{T}(A), A)$ be the space of all degree $n$ graded linear maps from the graded vector space $\overline{T}(A)$ to $A$. Thus, an element $\mu \in C^n (A,A)$ is given by a sum $\mu = \sum_{k = 1}^\infty \mu_k$, where $\mu_k : A^{\otimes k} \rightarrow A$ is a degree $n$ linear map, for $k \geq 1$. For $\mu = \sum_{k = 1}^\infty \mu_k \in C^m (A,A)$ and $\nu = \sum_{l = 1}^\infty \nu_l \in C^n(A,A)$, we define a bracket $[\mu, \nu] \in C^{m+n}(A,A)$ by
\begin{align*}
    [\mu, \nu]:= \sum_{s=1}^\infty \sum_{k+l = s + 1} \big(  \mu_k \circ \nu_l - (-1)^{mn}~ \nu_l \circ \mu_k  \big), \text{ where }
\end{align*}
\begin{align*}
    (\mu_k \circ \nu_l) (a_1, \ldots, a_{s}) = \sum_{i=1}^{s-l+1} (-1)^{|a_1| + \cdots + |a_{i-1}|}~ \mu_k \big(   a_1, \ldots, a_{i-1}, \nu_l (a_i, \ldots, a_{i+l-1}), a_{i+l}, \ldots, a_{s} \big).
\end{align*}
The graded vector space $\oplus_{n \in \mathbb{Z}} C^n (A,A)$ with the above bracket is a graded Lie algebra. An element $\mu = \sum_{k = 1}^\infty \mu_k \in C^1(A,A)$ is a Maurer-Cartan element of the graded Lie algebra $( \oplus_{n \in \mathbb{Z}} C^n (A,A), [~,~])$ if and only if the pair $(A, \{ \mu_k \}_{k=1}^\infty)$ is an $A_\infty$-algebra.

\medskip

Let $(A, \{ \mu_k \}_{k=1}^\infty)$ be an $A_\infty$-algebra. A {\bf representation} of this $A_\infty$-algebra is given by a pair $(M, \{ \eta_k \}_{k=1}^\infty)$ that consists of a graded vector space $M = \oplus_{i \in \mathbb{Z}} M_i$ with a collection $\{ \eta_k : \mathcal{A}^{k-1,1} \rightarrow M \}_{k=1}^\infty$ of degree $1$ linear maps satisfying the identities (\ref{a-inf-iden}) when exactly one of the variables $a_1, \ldots, a_n$ comes from $M$ and the corresponding linear operation $\mu_i$ or $\mu_j$ replaced by $\eta_i$ or $\eta_j$. Like the ungraded case, here $\mathcal{A}^{k-1,1}$ denotes the direct sum of all possible tensor powers of $A$ and $M$ in which $A$ appears $k-1$ times (and hence $M$ appears exactly once). Note that any $A_\infty$-algebra $(A, \{ \mu_k \}_{k=1}^\infty)$ can be realized as a representation of itself, where $\eta_k = \mu_k$, for $k \geq 1$.

\begin{defn}
A {\bf $Diass_\infty$-algebra} (also called a {\bf strongly homotopy diassociative algebra}) is a pair $(D, \{ \pi_k \}_{k=1}^\infty)$ consisting of a graded vector space $D= \oplus_{i \in \mathbb{Z}} D_i$ equipped with a collection of degree $1$ graded linear maps $\{ \pi_k : {\bf k}[Y_k] \otimes D^{\otimes k} \rightarrow D \}_{k=1}^\infty$ satisfying the following set of identities
\begin{align}\label{higher-diass}
    \sum_{k+l=n+1} \sum_{i=1}^{n-l+1} (-1)^{|a_1| + \cdots + |a_{i-1}|} \pi_k \big(  R_0^{k;i,l} (y); a_1, \ldots, a_{i-1}, \pi_l \big(  R_i^{k; i, l} (y); a_i, \ldots, a_{i+l-1}   \big), a_{i+l}, \ldots, a_n   \big) = 0,
\end{align}
for all $n \geq 1$, $y \in Y_n$ and homogeneous elements $a_1, \ldots, a_n \in D$. The maps $R_0^{k; i,l}$ and $R_i^{k; i, l}$ are described in (\ref{r0-maps}), (\ref{ri-maps}).
\end{defn} 

Note that any diassociative algebra can be realized as a $Diass_\infty$-algebra concentrated in degree $-1$. More precisely, if $(D, \dashv, \vdash)$ is a diassociative algebra then $s^{-1}D$ (considered as a graded vector space with $(s^{-1}D)_{-1} = D$ and $(s^{-1} D)_i = 0$ for $i \neq -1$) can be given a $Diass_\infty$-algebra structure with the operations $\{ \pi_k : {\bf k}[Y_k] \otimes (s^{-1} D)^{\otimes k} \rightarrow s^{-1} D \}_{k=1}^\infty$ given by
\begin{align*}
   \pi_2 ( \begin{tikzpicture}[scale=0.15]
\draw (6,0) -- (8,-2);  \draw (8,-2) -- (10,0); \draw (8,-2) -- (8,-4);   \draw (9,-1) -- ( 8,0);
\end{tikzpicture} ; s^{-1} a, s^{-1} b) = s^{-1} (a \dashv b), \quad  \pi_2 ( \begin{tikzpicture}[scale=0.15]
\draw (6,0) -- (8,-2);  \draw (8,-2) -- (10,0); \draw (8,-2) -- (8,-4);   \draw (7,-1) -- ( 8,0);
\end{tikzpicture} ; s^{-1} a, s^{-1} b) = s^{-1} (a \dashv b) ~~ \text{ and } \pi_k = 0 \text{ for } k \neq 2.
\end{align*}

\begin{remark}
    Let $(D, \{ \pi_k \}_{k=1}^\infty)$ be any $Diass_\infty$-algebra. Using the higher diassociative identities (\ref{higher-diass}) and the mathematical induction on $k$, we can show that
    \begin{align*}
        \pi_k (y; a_1, \ldots, a_k) = \pi_k (y'; a_1, \ldots, a_k), \text{ for } a_1, \ldots, a_k \in D,
    \end{align*}
    when both of $y, y' \in Y_k$ can be written as the grafting of a $(i-1)$-tree and $(k-i)$-tree.
\end{remark}

\begin{prop}
    Let $(A, \{ \mu_k \}_{k=1}^\infty)$ be an $A_\infty$-algebra and $(M, \{ \eta_k \}_{k=1}^\infty)$ be a representation of it. Then the graded vector space $A \oplus M$ can be equipped with a $Diass_\infty$-algebra structure with the operations $\{ \pi_k : {\bf k}[Y_k] \otimes (A \oplus M)^{\otimes k} \rightarrow A \oplus M \}_{k=1}^\infty$ given by
    \begin{align}\label{pi-k}
        \pi_k \big( y; (a_1, u_1), \ldots, (a_k, u_k) \big) = \big( \mu_k (a_1, \ldots, a_k ), \eta_k (a_1, \ldots, a_{i-1}, u_i, a_{i+1}, \ldots, a_k) \big),
    \end{align}
    for $k \geq 1$, $y \in Y_k$ (which can be uniquely written as $y = y_1 \vee y_2$ for some $(i-1)$-tree $y_1 \in Y_{i-1}$ and $(k-i)$-tree $y_2 \in Y_{k-i}$) and $(a_1, u_1), \ldots, (a_k, u_k ) \in A \oplus M.$
\end{prop}

We denote the above $Diass_\infty$-algebra simply by $A \oplus_{Diass_\infty} M$. Note that $A \oplus_{Diass_\infty} M$ generalizes the diassociative algebra of Proposition \ref{am-diass} in the homotopy context. It is important to mention that the converse of the above proposition is also true. More precisely, let $A = \oplus_{i \in \mathbb{Z}} A_i$ and $M = \oplus_{i \in \mathbb{Z}} M_i$ be two graded vector spaces equipped with two collections $\{ \mu_k : A^{\otimes k} \rightarrow A \}_{k=1}^\infty$ and $\{ \eta_k : \mathcal{A}^{k-1,1} \rightarrow M \}_{k=1}^\infty$ of degree $1$ graded linear maps. Then $(A, \{ \mu_k \}_{k=1}^\infty)$ is an $A_\infty$-algebra and $(M, \{ \eta_k \}_{k=1}^\infty)$ is a representation if and only if $(A \oplus M, \{ \pi_k \}_{k=1}^\infty)$ is a $Diass_\infty$-algebra, where the maps $\pi_k$'s are given in (\ref{pi-k}).

In the following, we construct a graded Lie algebra whose Maurer-Cartan elements correspond to $Diass_\infty$-algebra structures on a given graded vector space. Let $D = \oplus_{i \in \mathbb{Z}} D_i$ be a graded vector space. For each $n \in \mathbb{Z} $, we define the space $CY^n (D, D) := \mathrm{Hom}_n ( {\bf k}[\overline{Y}] \otimes \overline{T}(D), D )$ whose elements are of the form $\pi = \sum_{k=1}^\infty \pi_k$, where $\pi_k : {\bf k}[Y_k] \otimes D^{\otimes k} \rightarrow D$ is a degree $n$ linear map. For $\pi = \sum_{k=1}^\infty \pi_k \in CY^m (D,D)$ and $\varpi = \sum_{l=1}^\infty \varpi_l \in CY^l (D,D)$, we define an element $\{  \! [ \pi, \varpi ] \! \} \in CY^{m+n} (D,D)$ by
\begin{align}\label{di-bracket}
   \{  \! [ \pi, \varpi ] \! \} = \sum_{s=1}^\infty \sum_{k+l=s+1} (\pi_k \diamond \varpi_l - (-1)^{mn} \varpi_l \diamond \pi_k), ~ \text{ where } 
\end{align}
\begin{align*}
    (\pi_k \diamond \varpi_l) (y; a_1, \ldots, a_s) = \sum_{i=1}^{s-l+1} (-1)^{|a_1| + \cdots + |a_{i-1}|} \pi_k \big( R_0^{k; i, l} (y); a_1, \ldots,  \varpi_l (   R_i^{k; i, l} (y); a_i, \ldots, a_{i+l-1} ), \ldots, a_s \big),
\end{align*}
for $y \in Y_s = Y_{k+l-1}$ and homogeneous elements $a_1, \ldots, a_s \in D$. The bracket $\{ \! [ ~, ~ [ \! \}$ is the graded version of the Majumdar-Mukherjee bracket given in (\ref{mm-circ}). In particular, the bracket $\{ \! [ ~, ~ [ \! \}$  makes the graded space $\oplus_{n \in \mathbb{Z}} CY^n (D,D) = \oplus_{n \in \mathbb{Z}} \mathrm{Hom}_n ( {\bf k}[\overline{Y}] \otimes \overline{T}(D), D )$ into a graded Lie algebra.

\begin{prop}
    Let $D$ be a graded vector space. Then there is a one-to-one correspondence between $Diass_\infty$-algebra structures on the graded vector space $D$ and Maurer-Cartan elements of the graded Lie algebra $(\oplus_{n \in \mathbb{Z}} CY^n (D,D), \{ \! [ ~,~ ] \! \} ).$
\end{prop}

\begin{proof}
    Note that an element $\pi \in CY^1 (D,D) = \mathrm{Hom}_1 ( {\bf k}[\overline{Y}] \otimes \overline{T}(D), D )$ is equivalent to having a collection $\{\pi_k : {\bf k}[Y_k] \otimes D^{\otimes k} \rightarrow D \}_{k=1}^\infty $ of degree $1$ graded linear maps. Then it follows from (\ref{di-bracket}) that $\{ \! [ \pi, \pi ] \! \} = 0$ if and only if the collection $\{ \pi_k \}_{k=1}^\infty$ defines a $Diass_\infty$-algebra structure on $D$. 
\end{proof}

Given an $A_\infty$-algebra and a representation of it, we will now introduce the notion of a homotopy relative averaging operator. Let $(A, \{ \mu_k \}_{k=1}^\infty)$ be an $A_\infty$-algebra and $(M, \{ \eta_k \}_{k=1}^\infty)$ be a representation of it. Consider the graded Lie algebra 
\begin{align*}
    \mathfrak{g} = \big( \oplus_{n \in \mathbb{Z}} CY^n (A \oplus M, A \oplus M) = \oplus_{n \in \mathbb{Z}} \mathrm{Hom}_n (    {\bf k} [\overline{Y}] \otimes \overline{T}(A \oplus M), A \oplus M )  , \{ \! [ ~, ~] \! \}  \big)
\end{align*}
associated to the graded vector space $A \oplus M.$ Then it is easy to see that the graded subspace $\mathfrak{a} = \oplus_{n \in \mathbb{Z}} CY^n (M, A) =  \oplus_{n \in \mathbb{Z}} \mathrm{Hom}_n (    {\bf k} [\overline{Y}] \otimes \overline{T}(M), A )$ is an abelian Lie subalgebra of $\mathfrak{g}$. Let $p : \mathfrak{g} \rightarrow \mathfrak{g}$ be the projection map onto the subspace $\mathfrak{a}$. On the other hand, since $A \oplus_{Diass_\infty} M = (A \oplus M, \{ \pi_k \}_{k=1}^\infty)$ is a $Diass_\infty$-algebra, it defines a Maurer-Cartan element $\pi = \sum_{k=1}^\infty \pi_k \in CY^1 (A \oplus M, A \oplus M)$ of the graded Lie algebra $\mathfrak{g}$ (i.e. $\{ \! [ \pi, \pi ] \! \} = 0$). Further, the element $\pi \in \mathrm{ker}(p)_1$. Hence we obtain a $V$-data $(\mathfrak{g}, \mathfrak{a}, p, \pi)$. Therefore, by Theorem \ref{v-data-l} (i), the graded vector space $\mathfrak{a}$ inherits a $L_\infty$-algebra structure with the operations $\{ l_k : \mathfrak{a}^{\otimes k} \rightarrow \mathfrak{a} \}_{k=1}^\infty$ given by
\begin{align*}
    l_k (   \gamma_1, \ldots, \gamma_k ) =  p \{ \! [ \cdots \{ \! [     \{ \! [ \pi, \gamma_1   ] \! \} , \gamma_2 ] \! \}  , \ldots, a_k ] \! \}, 
\end{align*}
for homogeneous $ \gamma_1, \ldots, \gamma_k \in \mathfrak{a}$. This $L_\infty$-algebra can be seen as the homotopy analogue of the graded Lie algebra given in Theorem \ref{mc-thm-opp}. Our next definition is motivated by the Maurer-Cartan characterization of a relative averaging operator given in Theorem \ref{mc-thm-opp}.

\begin{defn}
    A {\bf homotopy relative averaging operator} on  $(M, \{ \eta_k \}_{k=1}^\infty)$ over the $A_\infty$-algebra  $(A, \{ \mu_k \}_{k=1}^\infty)$ is a Maurer-Cartan element of the $L_\infty$-algebra $( \mathfrak{a}, \{ l_k \}_{k=1}^\infty).$
\end{defn}

It follows from the above definition that a homotopy relative averaging operator is an element $P = \sum_{k=1}^\infty P_k \in \mathrm{Hom}_0 (   {\bf k}[\overline{Y}] \otimes \overline{T} (M), A)$ that satisfies
\begin{align}\label{homo-avg}
    \sum_{k=1}^\infty \frac{1}{k!} l_k (P, \ldots, P) = 0.
\end{align}
In other words, $P$ must satisfy $\sum_{ k=1}^\infty \frac{1}{k!} p \{ \! [ \cdots \{ \! [     \{ \! [ \pi, P  ] \! \} , P ] \! \}  , \ldots, P ] \! \} = 0$, which is equivalent to the condition that $p (e^{ \{ \! [ -, P ] \! \}} \pi ) = 0$. Note that a homotopy relative averaging operator can be equivalently described by a collection  $P = \{ P_k : {\bf k}[Y_k] \otimes M^{\otimes k} \rightarrow A \}_{k=1}^\infty$ of degree $0$ linear maps satisfying $p (e^{ \{ \! [ -, P ] \! \}} \pi ) = 0$.

\begin{defn}
A {\bf homotopy relative averaging algebra} is a triple $(A, M, P)$ consisting of an $A_\infty$-algebra $A = (A, \{\mu_k \}_{k=1}^\infty )$, a representation $M= (M, \{ \eta_k \}_{k=1}^\infty)$ and a homotopy relative averaging operator $P = \{ P_k \}_{k=1}^\infty$. We often denote a homotopy relative averaging algebra as above by $M \xrightarrow{ \{ P_k \}_{k=1}^\infty } A$.
\end{defn}

\begin{prop}\label{diass-inf-ind}
    Let $M \xrightarrow{  \{ P_k \}_{k=1}^\infty } A$ be a homotopy relative averaging algebra. Then $(M, \{\pi_k^P \}_{k=1}^\infty)$ is a $Diass_\infty$-algebra, where
    \begin{align*}
        \pi_k^P (y; u_1, \ldots, u_k) = (e^{ \{ \! [ -, P ] \! \}} \pi ) (y; u_1, \ldots, u_k), \text{ for } k \geq 1, y \in Y_k \text{ and } u_1, \ldots, u_k \in M.
    \end{align*}
\end{prop}

\begin{proof}
    Note that
    \begin{align*}
         \{ \! [   e^{ \{ \! [ -, P ] \! \} } \pi, e^{ \{ \! [ -, P ] \! \} } \pi ] \! \} = e^{ \{ \! [ -, P ] \! \} }  \{ \! [  \pi, \pi ] \! \} = 0 ~~~ (\text{as }  \{ \! [  \pi, \pi ] \! \} = 0).
    \end{align*}
    This shows that $ e^{ \{ \! [ -, P ] \! \} } \pi$ is a Maurer-Cartan element of the graded Lie algebra $\mathfrak{g}$. Hence the collection of maps $\{ \pi_k \}_{k=1}^\infty$ defines a $Diass_\infty$-algebra structures on $M$, where $\pi_k = (e^{ \{ \! [ -, P ] \! \} } \pi )|_{  {\bf k}[Y_k] \otimes M^{\otimes k} }$, for $k \geq 1$. This completes the proof.
\end{proof}

A homotopy relative averaging operator $\{ P_k \}_{k=1}^\infty$ is said to be {\bf strict} if $P_k = 0$ for $k \neq 1$. It follows from (\ref{homo-avg}) that a strict homotopy relative averaging operator is a degree $0$ linear map $P : M \rightarrow A$ that satisfies
\begin{align*}
    \mu_k \big(   P(u_1), \ldots, P(u_k) \big) = P \big( \eta_k (   P(u_1), \ldots, u_i, \ldots, P(u_k)  ) \big), \text{ for } k \geq 1 \text{ and } 1 \leq i \leq k.
\end{align*}
A strict homotopy relative averaging algebra is a triple that consists of an $A_\infty$-algebra, a representation and a strict homotopy relative averaging operator. In this case, Theorem \ref{diass-inf-ind} reads as follows.

\begin{thm}
    Let $M \xrightarrow{P} A$ be a strict homotopy relative averaging algebra. Then $(M, \{ \pi_k^P \}_{k=1}^\infty)$ is a $Diass_\infty$-algebra, where
    \begin{align*}
        \pi_k^P (y; u_1, \ldots, u_k) :=  \eta_k (   P(u_1), \ldots, u_i, \ldots, P(u_k)  ),
    \end{align*}
    for $k \geq 1$, $y \in Y_k$ (which can be uniquely written as $y= y_1 \vee y_2$ for some $(i-1)$-tree $y_1 \in Y_{i-1}$ and $(k-i)$-tree $y_2 \in Y_{k-i}$) and $u_1, \ldots, u_k \in M$.
\end{thm}

In the following, we show that any $Diass_\infty$-algebra is always induced from a strict strict homotopy relative averaging algebra. Let $(D, \{ \pi_k \}_{k=1}^\infty)$ be a given $Diass_\infty$-algebra. Consider the graded vector space $D/ I$ which is obtained from $D$ quotient by the homogeneous ideal $I$ generated by the set
\begin{align*}
     \{ \pi_k (y; a_1, \ldots, a_k) - \pi_k (y'; a_1, \ldots, a_k) ~|~ k \geq 1,~ y, y' \in Y_k \text{ and } a_1, \ldots, a_k \in D \}.
\end{align*}
It is easy to see that the graded vector space $D/I$ carries an $A_\infty$-algebra structure with the operations $\{ \mu_k : (D/I)^{\otimes k} \rightarrow D/I \}_{k=1}^\infty$ given by
\begin{align*}
    \mu_k (  [a_1], \ldots, [a_k]) = [\pi_k (y; a_1, \ldots, a_k)], \text{ for } k \geq 1 \text{ and } [a_1], \ldots, [a_k] \in D/I.
\end{align*}
We denote this $A_\infty$-algebra structure simply by $D_{\mathrm{Ass}_\infty}$. It is also easy to check that the $A_\infty$-algebra $D_{\mathrm{Ass}_\infty}$ has a representation on the graded vector space $D$ with the action maps
\begin{align*}
    \eta_k ([a_1], \ldots, [a_{i-1}], a_i, [a_{i+1}], \ldots, [a_k]) = \pi_k (y; a_1, \ldots, a_k),
\end{align*}
for $k \geq 1$, $[a_1], \ldots, [a_{i-1}], [a_{i+1}], \ldots, [a_k] \in D_{\mathrm{Ass}_\infty}$ and $a_i \in D$. Here $y \in Y_k$ is any $k$-tree which is the grafting of some $(i-1)$-tree and $(k-i)$-tree. Moreover, $D \xrightarrow{q} D_{\mathrm{Ass}_\infty}$ is a strict homotopy relative averaging algebra, where $q$ is the quotient map. Further, the induced $Diass_\infty$-algebra structure on $D$ coincides with the given one.

\medskip

\noindent {\bf Data availability statement.} Data sharing does not apply to this article as no datasets were generated or analysed during the current study.

\medskip

\noindent {\bf Acknowledgements.}  The author would like to thank Indian Institute of Technology (IIT) Kharagpur for providing the beautiful academic environment where the research has been carried out.


\end{document}